\documentclass{amsart}
\usepackage[bookmarksnumbered, bookmarksopen,backref]{hyperref}
\usepackage[all]{xy}    
\usepackage{graphicx}
\usepackage{amssymb}
\usepackage{mathrsfs}
\usepackage{multirow}
\usepackage{float}
\usepackage{tikz-cd}
\usepackage{adjustbox}
\usepackage{amsthm}
\usepackage{accents}
\usepackage{mathtools}
\usepackage{enumitem,cite,array}
\usepackage[new]{old-arrows}
\usepackage{tikz}
\usetikzlibrary{matrix,arrows}
\usepackage{amsmath}

\usepackage[numbers,sort&compress]{natbib} 

\newcounter{RomanNumber}
\newcommand{\MyRoman}[1]{\setcounter{RomanNumber}{#1}\Roman{RomanNumber}}

\newtheorem{theorem}{Theorem}
\newtheorem{lemma}[theorem]{Lemma}

\theoremstyle{definition}
\newtheorem{definition}[theorem]{Definition}

\newtheorem{corollary}[theorem]{Corollary}

\newtheorem{remark}[theorem]{Remark}

\theoremstyle{notation}

\newcommand{\be}{\begin{equation}}
\newcommand{\ee}{\end{equation}}

\newcommand{\Z}{\mathbb{Z}}

\newcommand{\q}{\scriptscriptstyle \mathbb{Q}}

\newcommand{\bH}{\mathbb{H}}

\newcommand{\GCh}{\mathrm{GCh}}

\usepackage[utf8]{inputenc}
\usepackage{chngcntr}
\usepackage{apptools}
\AtAppendix{\counterwithin{theorem}{section}}
\usepackage[toc,page]{appendix}

\newcommand{\qqed}{\hfill\Box}

\begin{document}
\begin{sloppypar}


\title
[Fractional structures]{Fractional structures on bundle gerbe modules and fractional classifying spaces}

\author{Fei Han}
\address{Department of Mathematics,
National University of Singapore, Singapore 119076}
\email{mathanf@nus.edu.sg}

 \author{Ruizhi Huang}
\address{Institute of Mathematics and Systems Sciences, Chinese Academy of Sciences, Beijing 100190, China}
\email{huangrz@amss.ac.cn}  
\urladdr{https://sites.google.com/site/hrzsea
}

 \author{Varghese Mathai}
\address{School of Mathematical Sciences,
University of Adelaide, Adelaide 5005, Australia}
\email{mathai.varghese@adelaide.edu.au}

\date{}

\maketitle

\begin{abstract} 
We study the homotopy aspects of the twisted Chern classes of torsion bundle gerbe modules. Using Sullivan's rational homotopy theory, we realize the twisted Chern classes at the level of classifying spaces. The construction suggests a notion, which we call fractional U-structure serving as a universal framework to study the twisted Chern classes of torsion bundle gerbe modules from the perspective of classifying spaces. Based on this, we introduce and study higher fractional structures on torsion bundle gerbe modules parallel to the higher structures on ordinary vector bundles. 
\end{abstract}

\tableofcontents

\newpage
\section*{Introduction}
\label{sec: intro}
Roughly speaking, torsion bundle gerbe modules are twisted versions of vector bundles, which can serve as cocycles of twisted K-theory where the twist is a torsion 3-class. This paper is aimed to study higher structures on torsion bundle gerbe modules from the perspective of homotopy theory.

$\, $

{\bf Higher structures.} 
The higher structures on an oriented vector bundle, like spin, string and SU, U$\langle 6\rangle$ (when the vector bundle is complex) are important in topology and geometry. For instance, in the real case, the significance of spin structure is illustrated by the classical Atiyah-Singer index theory \cite{AS}, which is the cornerstone of spin geometry (c.f. \cite{LM}). Further, the existence of a string structure on the tangent bundle of a manifold implies the modularity \cite{Zag} of the Witten genus \cite{Wit1}, whose homotopy refinement leads to the profound theory of topological modular form developed by Hopkins and Miller \cite{Hop}. In the complex case, Adams \cite{Ada} initiated the study of algebraic topology of the higher structures including SU and U$\langle 6\rangle$. For a Calabi-Yau manifold, its complex tangent bundle has SU structure. In their celebrated paper, \cite{AHS} Ando-Hopkins-Strikland constructed the $\sigma$-orientations from $MU\langle 6\rangle$ to elliptic spectra. 

On the other hand, the loop space aspects of the higher structures were motivated by quantum anomaly in physics \cite{Kil}.
The so-called loop orientable and loop spin structures are transgressed from spin and string structures respectively \cite{Mcl}. More geometrically, the additional fusive structures on higher loop structures were introduced by Stolz-Teichner \cite{ST}, and were further investigated by Waldorf \cite{Wal1} and Kottke-Melrose \cite{KM}.

$\, $

{\bf Bundle gerbe modules.} 
When the base manifold is equipped with an $S^1$-gerbe, there are theories of twisted versions of vector bundles, which form cocycles of twisted K-theory (see the developments of the theory in \cite{DK, Ros, AS1, AS2, FHT1, FHT2, FHT3} and the physics aspects in \cite{MM, Wit2}). Bundle gerbe is a model of $S^1$-gerbe and bundle gerbe module is a twisted version of vector bundles. Bundle gerbes were introduced by Murray \cite{Mu} as geometrization of degree 3 cohomology classes, while bundle gerbe modules over a given bundle gerbe were introduced by Bouwknegt, Carey, Mathai, Murray and Stevenson \cite{BCMMS} as cocycles of twisted $K$-theory.

The twisted Chern classes for twisted K-theory have been studied through a couple of approaches \cite{AS2, BCMMS}. In particular, regarding bundle gerbe modules, in \cite{BCMMS} Bouwknegt, Carey, Mathai, Murray and Stevenson introduced the twisted Chern character from the perspective of Chern-Weil theory. When the bundle gerbe modules are finite ranks, the Dixmier-Douady class of the bundle gerbe must be a torsion class \cite{BCMMS}. In this case, we call them torsion bundle gerbe modules. \footnote{We remark that Tomoda \cite{Tom} proved the splitting principle of torsion bundle gerbe modules and reproduced the twisted Chern classes from this perspective.} 
The fractional index theorem of Mathai-Melrose-Singer established in \cite{MMS, MMS2} can be viewed as the quantization of the theory of torsion bundle gerbe modules and their twisted Chern classes.

$\, $

{\bf Summary of main results.} 
The purpose of this paper is to study the homotopy aspects of torsion bundle gerbe modules and their twisted Chern classes. 
Since in this paper we only consider torsion bundle gerbes and their modules, we will drop the adjective {\em torsion}. 
Using Sullivan's rational homotopy theory \cite{Sul}, we realize twisted Chern classes at the level of classifying spaces.
More precisely, we construct a {\it relative classifying space} for bundle gerbe modules, in the sense that it classifies bundle gerbe modules up to twisted Chern classes. The construction suggests a slightly more general homotopy notion, {\it fractional U-structure}, which serves as a universal framework to study the twisted Chern classes of bundle gerbe modules from the perspective of classifying spaces. 
Indeed, the relative classifying space is exactly the classifying space of fractional U-structure, which we will refer to as {\it fractional classifying space}.
In due course, we explicitly calculate the precise formulae resolving the relation between the twisted Chern classes and the ordinary Chern classes associated to bundle gerbe modules.

An advantage of the homotopy construction is that it allows us to study higher structures on bundle gerbe modules from the perspective of classifying spaces, in a pattern similar to the higher structures on ordinary vector bundles. For bundle gerbe modules, it is natural to define the {\it fractional SU-structure} and {\it fractional U$\langle6\rangle$-structure} in terms of vanishment of the first two twisted Chern classes similar to the SU- and U$\langle6\rangle$ structure on ordinary complex vector bundles in terms of the vanishment of the first two ordinary Chern classes. It is worthwhile to remark that geometrically such fractional U$\langle 6\rangle$-condition is the obstruction to the modularity of the projective elliptic genera and graded Chern character of Witten gerbe modules \cite{HM1, HM2}. The fractional SU- and U$\langle6\rangle$-structures are the twisted counterparts of SU- and U$\langle6\rangle$-structures respectively. To study them from the perspective of classifying spaces, we carefully construct a delicate {\it relative Whitehead tower} of the relative classifying space of bundle gerbe modules, and identify the respective {\it fractional classifying spaces} of the fractional SU- and U$\langle 6\rangle$-structures. We are then able to characterize the two higher fractional structures as relative homotopy lifting problems with respect to the relative Whitehead tower, and count the number of these higher fractional structures analogous to the untwisted cases. We expect that there would be some interesting applications of the fractional U$\langle6\rangle$-structures. Indeed, as pointed out in Remark \ref{hmremark}, associated to fractional U$\langle6\rangle$ bundle gerbe modules over U$\langle6\rangle$-manifolds, there are explicitly constructed modular projective elliptic genera in \cite{HM1, HM2}.
Following the work of Ando-Hopkins-Strikland \cite{AHS}, it might lead to a fractional version of the $\sigma$-orientations.

With the relative classifying space, we can also study higher fractional structures of bundle gerbe modules from the loop space point of view in a pattern similar to loop orientable and loop spin structures. Indeed, with the classical transgression map to loop spaces, it is natural to define the {\it fractional loop U-structure} and {\it fractional loop SU-structure} in terms of vanishment of the transgressions of the first two twisted Chern classes. To study them from the perspective of classifying spaces, we also carefully construct an explicit {\it relative Whitehead tower} of the looped relative classifying space of bundle gerbe modules, and identify the respective {\it fractional classifying spaces} of the fractional loop U- and SU-structures. As in the non-loop cases, we are then able to characterize the two higher fractional loop structures as certain relative homotopy lifting problems and also count them. 

The higher fractional non-loop and loop structures can be compared either from the perspective of classifying spaces or from the perspective of transgressions. In either way, it can be shown that higher fractional non-loop structures are generally stronger than their loop counterparts. In particular, fractional SU implies fractional loop U, while fractional U$\langle6\rangle$ implies fractional loop SU. Moreover, the distinct fractional non-loop structures transgress to the corresponding distinct fractional loop structures. Under certain conditions on the topology of the base manifold, the higher fractional non-loop and loop structures are equivalent, and there are one-to-one correspondences between distinct structures. 

$\, $

{\bf A graphic illustration of higher structures.} 
Before closing the introduction, let us illustrate the hierarchy of fractional complex structures by the right diagram of (\ref{summdiag}). In contrast, the corresponding hierarchy of ordinary real structures is depicted on the left diagram of (\ref{summdiag}) for comparison. Here, a straight arrow $\mathcal{A}\rightarrow \mathcal{B}$ means that a $\mathcal{B}$-structure is lifted to an $\mathcal{A}$-structure, and hence $\mathcal{A}$ is one level higher than $\mathcal{B}$, while a wave arrow $\mathcal{C}\leadsto{\rm loop}~\mathcal{D}$ means that a $\mathcal{C}$-structure transgresses to a loop $\mathcal{D}$-structure, and hence $\mathcal{C}$ is generally stronger than loop $\mathcal{D}$. For more explicit explanations of (\ref{summdiag}), one may refer to Subsection \ref{sec: highfracintro}, \ref{sec: highloopfracuintro} and \ref{sec: highfracvsintro} for the right diagram, and to Subsection \ref{sec: backintro} for the left diagram. 

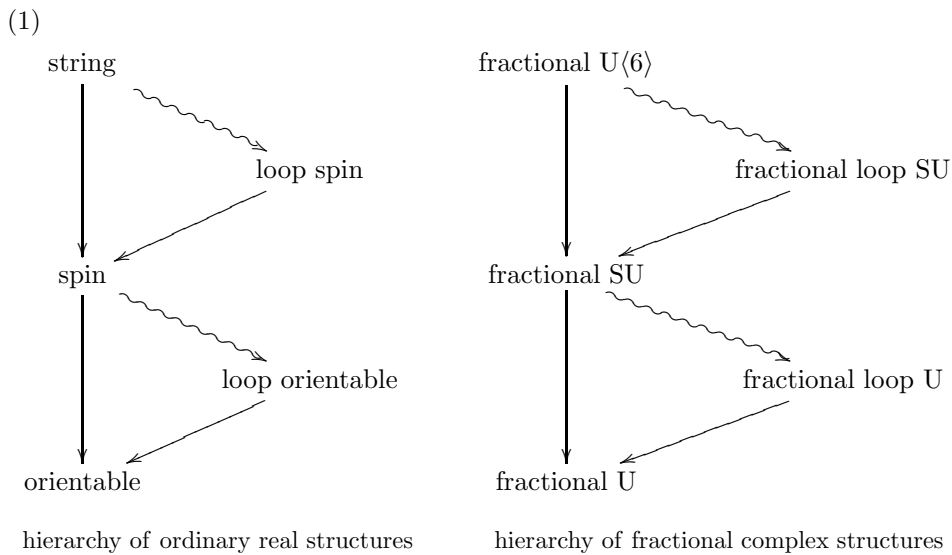
\begin{gather}
\begin{aligned}
\resizebox{\textwidth}{!}{
\xymatrix{
{\rm string}                                         \ar[dd]    \ar@{.>}@{~>}[dr]   &                  &  
{\rm fractional}~{\rm U}\langle6\rangle  \ar[dd]    \ar@{.>}@{~>}[dr]  &              \\
&       {\rm loop}~{\rm spin}  \ar[dl]  &  & {\rm fractional}~{\rm loop}~{\rm SU}  \ar[dl] \\
{\rm spin}   \ar[dd]    \ar@{.>}@{~>}[dr]                                           &    & 
{\rm fractional}~{\rm SU}   \ar[dd]    \ar@{.>}@{~>}[dr]                     &    \\
&  {\rm loop}~{\rm orientable} \ar[dl]     &&   {\rm fractional}~{\rm loop}~{\rm U} \ar[dl]    \\
{\rm orientable}                                                                           &                                                       &
{\rm fractional}~{\rm U} 
   }
}
\end{aligned}
\label{summdiag}
\end{gather}
\begin{gather*}
\begin{aligned}
\resizebox{\textwidth}{!}{
\xymatrix{ {\rm hierarchy}~{\rm of}~{\rm  ordinary}~{\rm  real}~{\rm  structures}~ &   {\rm hierarchy}~{\rm  of}~{\rm  fractional}~{\rm  complex}~{\rm  structures}~
}
}
\end{aligned}
\end{gather*}


$\, $

{\bf Organization of the paper.} 
\label{sec: organintro}
The main body of the paper starts with a comprehensive introduction of our main results in Section \ref{sec: result} with background information on the relevant theories on ordinary vector bundles. More explicitly, in Subsection \ref{sec: backintro}, we review the untwisted higher structures mainly from the perspective of classifying spaces. In Subsection \ref{sec: fracuintro}, we introduce our results on the homotopy realization of twisted Chern classes, and define the fractional U-bundle as a homotopy generalization of bundle gerbe modules. In Subsection \ref{sec: highfracintro} and \ref{sec: highloopfracuintro}, we review our results on the higher fractional non-loop and loop structures respectively, followed by their comparison in Subsection \ref{sec: highfracvsintro}. In particular, the main results are summarized in Theorem \ref{relclassthmintro}, \ref{fracsunthm}, \ref{fracu6nthm}, \ref{fracloopunthm}, and \ref{fracloopsunthm}.

The next three sections are devoted to realize the twisted Chern classes explicitly from the perspective of classifying spaces.
In Section \ref{sec: split}, we recall necessary background on bundle gerbe modules.
In Section \ref{sec: fchern}, we show the explicit formulae and combinatorics of the twisted Chern classes. 
Motivated by the computation, in Section \ref{sec: Bfchern} we apply Sullivan's rational homotopy theory to realize the twisted Chern classes from the perspective of classifying spaces, and propose a notion of fractional U-structure. This is essentially summarized in Theorem \ref{relclassthmintro}.

The last two sections are about the study higher fractional structures.
In Section \ref{sec: strongfrac}, we introduce fractional SU- and U$\langle 6\rangle$-structures, while in Section \ref{sec: weakfrac}, we introduce fractional loop U- and SU-structures. In each case, we characterize the existence of the higher fractional structures as certain relative homotopy lifting problem and count such structures. In particular, we construct the two relative Whitehead towers for fractional structures. These are summarized in Theorem \ref{fracsunthm}, \ref{fracu6nthm}, \ref{fracloopunthm}, and \ref{fracloopsunthm}. In Section \ref{sec: weakfrac}, we compare the fractional loop structures with the fractional non-loop structures. 

The paper ends up with a \hyperref[AppendixA]{table} to record the characteristic classes appeared in this paper.

$\, $

\noindent{\em Conventions:}
\label{Convention}
\begin{itemize}
\item We always use $\simeq$ to denote homotopy equivalence;
\item In this paper, the spaces $Y$ and $M$ under consideration are always assumed to be connected;
\item Throughout the paper, $H^\ast(X)$ is used to denote the singular cohomology $H^\ast(X;\mathbb{Z})$ with integral coefficient;
\item For any class $x\in H^k(X; R)$ in coefficient $R$, it is represented by a map $X\longrightarrow K(R,k)$ into Eilenberg-MacLane space under the Brown representation theorem. For the sake of simplicity, we always use the symbol $x$ to denote both the map and its represented class; further, for a given map $g: Y\rightarrow X$, we may denote $x:=g^\ast(x)$ by abuse of notation. This convention will be frequently used for liftings of universal characteristic classes.
\item The notations of characteristic classes are consistent throughout the paper. For convenience, they are summarized in the \hyperref[AppendixA]{table} at the end of the main text for reference. Following the last convention, we use the same notation for each universal characteristic classes and its liftings through appropriate maps defined in the paper, such as for $c_i$, $z_i$, $c_i^{\q}$, $z_i^{\q}$, $\bar{c}_1$ or $\bar{z}_1$.
\end{itemize}

$\, $

\noindent{\bf Acknowledgement} 
Fei Han was partially supported by the grant AcRF R-146-000-218-112 from National University of Singapore. 

Ruizhi Huang was supported in part by the National Natural Science Foundation of China (Grant nos. 11801544 and 11688101), the National Key R\&D Program of China (No. 2021YFA1002300), the Youth Innovation Promotion Association of Chinese Academy Sciences, and the ``Chen Jingrun’’ Future Star Program of AMSS. 

Varghese Mathai was supported by funding from the Australian Research Council, through the Australian Laureate Fellowship FL170100020.

\newpage
\section{Fractional structures and statement of results}
\label{sec: result}
\subsection{Background} 
\label{sec: backintro}
Before we study fractional structures on bundle gerbe modules, let us recap the relevant theories on ordinary vector bundles, for which one can refer to the left diagram of (\ref{summdiag}) for guidance.

Let $V$ be an oriented real vector bundle of rank $m$ over a given manifold $M$.
If some characteristic classes of $V$ vanish, one can expect stronger information on the geometry and topology of $V$.
For instance, $V$ is called {\it spin} if its second Stiefel-Whitney class $\omega_2(V)=0$. 
The classical Atiyah-Singer index theory \cite{AS} illustrates the significance of the spin condition on tangent bundles of manifolds by integrating geometry, topology and analysis together on manifolds using Dirac operators. 
Furthermore, $V$ is called {\it string} if it is spin and the first spin class $\frac{1}{2}p_1(V)=0$. When the tangent bundle of a manifold is string, the Witten genus \cite{Wit1} is a modular form \cite{Zag}, whose homotopy refinement leads to the profound theory of {\em topological modular form} developed by Hopkins and Miller \cite{Hop}.

From the perspective of homotopy theory, these higher structures can be understood through classifying spaces. Let
\[
\mathfrak{f}: M\stackrel{}{\longrightarrow} BSO(m),
\]
be the classifying map of $V$, where $SO(m)$ is the $m$-th special orthogonal group and $B$ is the {\it classifying functor}. The spin and string structures can be described through liftings of the {\it structure group} $SO(m)$. Indeed, there is the {\it Whitehead tower} of the {\it classifying space} $BSO(m)$
\[\label{whtowersoeq}
\cdots\stackrel{}{\longrightarrow} BString(m)\stackrel{B\mathfrak{i}_3}{\longrightarrow} BSpin(m)\stackrel{B\mathfrak{i}_2}{\longrightarrow} BSO(m),
\]
where $Spin(m)$ is the $m$-th spin group, and the string group $String(m)$ is determined by the universal group extension (cf. \cite{ST})
\[\label{stringexteqintro}
\{1\}\stackrel{}{\longrightarrow} K(\mathbb{Z},2) \stackrel{}{\longrightarrow} String(m)\stackrel{\mathfrak{i}_3}{\longrightarrow}  Spin(m) \stackrel{}{\longrightarrow} \{1\}
\]
with $K(\mathbb{Z},2)$ being the Eilenberg-MacLane space such that $\pi_2(K(\mathbb{Z},2))\cong \mathbb{Z}$. The bundle $V$ being spin is equivalent to that its structure group $SO(m)$ can be lifted to $Spin(m)$, that is, $\mathfrak{f}$ factors through $B\mathfrak{i}_2$. Similarly, $V$ being string is equivalent to that the structure group can be further lifted to $String(m)$, that is, $\mathfrak{f}$ factors through $B\mathfrak{i}_3$. 
By the standard arguments in homotopy theory, it can be shown that distinct spin and string structures are parametrized by $H^1(M;\mathbb{Z}/2\mathbb{Z})$ and $H^3(M;\mathbb{Z})$ respectively (cf. \cite{LM}).

Moreover, the physical motivation from quantum anomaly \cite{Kil} leads to the study of spin and string structures via the {\it free loop space} $LM$ with $L$ being the {\it free loop functor}. Indeed, the looped higher structures can be understood via the transgression map (see (\ref{freesuspen})) of the characteristic classes of $V$. If the transgression of $\omega_2(V)$ vanishes, $V$ is called {\it loop orientable}; further, if $V$ is spin and the transgression of $\frac{1}{2}p_1(V)$ vanishes, $V$ is called {\it loop spin} \cite{Mcl}. 

To understand these loop structures through loop spaces of classifying spaces, consider the looped classifying map
\[
L\mathfrak{f}: LM\stackrel{}{\longrightarrow} BLSO(m)
\]
of $V$. Recall that $LG$ is a {\it loop group} for a topological group $G$.
On the free loop level, there is the Whitehead tower of $BLSO(m)$ (cf. \cite{Mcl})
\[\label{whtowerlsoeq}
\cdots\stackrel{}{\longrightarrow} B\widehat{LSpin}(m)\stackrel{B\varsigma_3}{\longrightarrow} BLSpin(m)\stackrel{B\widehat{L\mathfrak{i}_2}}{\longrightarrow}
BL_0SO(m)\stackrel{B\varsigma_2}{\longrightarrow} BLSO(m),
\]
where $L_0SO(m)$ is the $0$-component of $LSO(m)$ containing the constant loop on identity, $\widehat{LSpin}(m)$ is determined by the universal circle extension
\[\label{lspinhatexteqintro}
\{1\}\stackrel{}{\longrightarrow} U(1) \stackrel{}{\longrightarrow} \widehat{LSpin}(m)\stackrel{\varsigma_3}{\longrightarrow}  LSpin(m) \stackrel{}{\longrightarrow} \{1\},
\]
and $\varsigma_2\circ \widehat{L\mathfrak{i}_2}=L\mathfrak{i}_2$.
The bundle $V$ being loop orientable is equivalent to that its structure group $LSO(m)$ can be lifted to $L_0SO(m)$ through $\varsigma_2$. 
Similarly, $V$ being loop spin is equivalent to that the structure group can be further lifted to $\widehat{LSpin}(m)$ through $\varsigma_3$.
It can be shown that distinct loop orientable and loop spin structures are parametrized by $H^0(LM;\mathbb{Z}/2\mathbb{Z})$ and $H^2(LM;\mathbb{Z})$ respectively \cite{Mcl}.

More geometrically, Stolz and Teichner gave the link of the string structure on $M$ to the fusive spin structure on $LM$ \cite{ST}. This was further developed by Waldorf \cite{Wal1, Wal2} and Kottke-Melrose \cite{KM}. In \cite{Bun}, Bunke studied the Pfaffian line bundle of certain family of Dirac operators and showed that string structures give rise to trivializations of that Pfaffian line bundle. See also the study of string structures from the differential and the twisted point of view \cite{Red, Sat, DHH}.

Now let us turn to the case of complex vector bundles. Let $E$ be a complex bundle of rank $n$ over $M$.
$E$ is called {\it SU} if its first Chern class $c_1(E)=0$. Further, $E$ is called {\it U$\mathit{\langle 6\rangle}$} if it is SU and the second Chern class $c_2(E)=0$. The $U\langle 6\rangle$ structure was first introduced from the perspective of Whitehead tower of classifying spaces by Adams \cite{Ada}. In the celebrated paper \cite{AHS}, Ando-Hopkins-Strikland constructed the $\sigma$-orientations from $MU\langle 6\rangle$ to elliptic spectra.
 
To understand these higher complex structures through classifying spaces, consider the classifying map 
\[
f: M\stackrel{}{\longrightarrow} BU(n)
\]
of $E$ with $U(n)$ being the $n$-th unitary group. There is the Whitehead tower of $BU(n)$
\begin{equation}\label{whtowerueq}
\cdots\stackrel{}{\longrightarrow} BU\langle6\rangle(n)\stackrel{Bi_3}{\longrightarrow} BSU(n)\stackrel{Bi_2}{\longrightarrow} BU(n),
\end{equation}
where $SU(n)$ is the $n$-th special unitary group and $U\langle6\rangle(n)$ is determined by the universal group extension (\cite{Sin, AHS})
\begin{equation}\label{u6exteqintro}
\{1\}\stackrel{}{\longrightarrow} K(\mathbb{Z},2) \stackrel{}{\longrightarrow} U\langle 6 \rangle (n)\stackrel{i_3}{\longrightarrow}  SU(n) \stackrel{}{\longrightarrow} \{1\}.
\end{equation}
The complex bundle $E$ being SU is equivalent to that $f$ can be lifted to $BSU(n)$ through $Bi_2$. Similarly, $E$ being U$\langle 6\rangle$ is equivalent to that $f$ can be further lifted to $BU\langle6\rangle(n)$ through $Bi_3$. One can similarly use the standard arguments in homotopy theory to show that the distinct SU and U$\langle6\rangle$ structures are parametrized by $H^1(M;\mathbb{Z})$ and $H^3(M;\mathbb{Z})$ respectively.

As in the real case, one can consider these higher complex structures after looping.
Indeed, there is a loop class $z_1(LE)\in H^1(LM;\mathbb{Z})$ as the transgression of $c_1(E)$. We call $E$ {\it loop U} if $z_1(LE)=0$. Further, suppose $E$ is SU. There is a loop class $z_2(LE)\in H^3(LM;\mathbb{Z})$ as the transgression of $c_2(E)$. We call $E$ is {\it loop SU} if $z_2(LE)=0$.

To understand these loop complex structures through loop spaces of classifying spaces, consider the looped classifying map
\[
Lf: LM\stackrel{}{\longrightarrow} BLU(n)
\]
of $E$.
On the free loop level there is the Whitehead tower of $BLU(n)$ (cf. \cite{Mcl})
\begin{equation}\label{whtowerlueq}
\cdots\stackrel{}{\longrightarrow} B\widehat{LSU}(n)\stackrel{B\iota_3}{\longrightarrow} BLSU(n)\stackrel{B\widehat{Li_2}}{\longrightarrow}
BL_0U(n)\stackrel{B\iota_2}{\longrightarrow} BLU(n),
\end{equation}
where $L_0U(n)$ is the $0$-component of $LU(n)$ containing the constant loop, $\widehat{LSU}(n)$ is determined by the universal circle extension
\begin{equation}\label{lsuhatexteqintro}
\{1\}\stackrel{}{\longrightarrow} U(1) \stackrel{}{\longrightarrow} \widehat{LSU}(n)\stackrel{\iota_3}{\longrightarrow}  LSU(n) \stackrel{}{\longrightarrow} \{1\},
\end{equation}
and $\iota_2\circ \widehat{Li_2}=Li_2$.
The bundle $E$ being loop U is equivalent to that its structure group $LU(n)$ can be lifted to $L_0U(n)$ through $\iota_2$. 
Similarly, $E$ being loop SU is equivalent to that the structure group can be further lifted to $\widehat{LSU}(n)$ through $\iota_3$.
It can be shown that the distinct loop U and loop SU structures are parametrized by $H^0(LM;\mathbb{Z})$ and $H^2(LM;\mathbb{Z})$ respectively (cf. \cite{Mcl}).

\subsection{Twisted Chern classes and fractional U-structure} 
\label{sec: fracuintro}
In this subsection, we will introduce Theorem \ref{relclassthmintro}, one of our main results, which realizes the twisted Chern classes from the perspective of classifying spaces. In particular, it provides a fractional classifying space of bundle gerbe modules up to twisted Chern classes, which brings us to a general homotopy notion, the fractional U-structure. The latter provides us a homotopy framework to study higher fractional structure in the sequel. 

Let $M$ be a smooth manifold. A {\it bundle gerbe} over $M$ consists of the following data: 
\begin{itemize}
\item[(i)] a {\it locally split map} $\pi: Y\rightarrow M$; 
\item[(ii)] a complex line bundle $L$ over $Y^{[2]}=Y\times_{\pi} Y$, the fiber product of $Y$ with itself over $\pi$; 
\item[(iii)] an isomorphism $$L_{(y_1, y_2)}\otimes L_{(y_2, y_3)}\to L_{(y_1, y_3)} $$ for every $(y_1, y_2)$ and $(y_2, y_3)$ in $Y^{[2]}$. 
\end{itemize}
The bundle gerbe has a characteristic class $d(L)=d(L, Y)\in H^3(M;\mathbb{Z})$, the {\it Dixmier-Douady class}. The Dixmier-Douady classs is the obstruction to the gerbe being trivial.

Let $E\rightarrow Y$ be a complex vector bundle. 
Let $\pi_i: Y^{[2]}\rightarrow Y$ be the map which omits the $i$-th element.
Suppose that $E$ is a {\it bundle gerbe module} of $L$, i.e. there is a complex bundle isomorphism
\[
\phi: L\otimes \pi_1^\ast(E)\cong \pi_2^\ast(E),
\]
which is compatible with the bundle gerbe multiplication in the sense that the two maps
$$L_{(y_1, y_2)}\otimes (L_{(y_2, y_3)}\otimes E_{y_3})\to L_{(y_1, y_2)}\otimes E_{y_2}\to E_{y_1} $$
and 
$$(L_{(y_1, y_2)}\otimes L_{(y_2, y_3)})\otimes E_{y_3}\to L_{(y_1, y_3)}\otimes E_{y_3}\to E_{y_1} $$
are the same. 

Assume the rank of $E$ is $n$, a finite positive integer. It is shown in \cite{BCMMS} that the Dixmier-Douady class $d(L)$ must be torsion, i.e. the bundle gerbe $(L,Y)$ is {\it flat}; and if $d(L)$ is a torsion class of order $l$, then $l|n$ and the tensor power $(Y, L^{\otimes l})$ has Dixmier-Douady class zero and it is therefore a trivial bundle gerbe. Fix a trivilization of $(Y, L^{\otimes l})$, or more precisely, fix a complex line bundle $\xi$ on $Y$ such that 
$$ L^{\otimes l}\cong \delta(\xi):=\pi_1^\ast(\xi^\ast)\otimes \pi_2^\ast(\xi).$$ 

By developing the splitting principle for bundle gerbe modules \cite{Tom}, Tomoda reproduced the twisted Chern classes $c_k^{l, \xi}(E)$ and showed that the Chern character constructed from these twisted Chern classes coincides with the twisted Chern character of Bouwknegt-Carey-Mathai-Murray-Stevenson \cite{BCMMS}. 

Let $\xi$ be classified by $a\in H^2(Y;\mathbb{Z})\cong [Y, BU(1)]$, the group of the homotopy classes of based maps, and $E$ classified by $f: Y\rightarrow BU(n)$. The data of bundle gerbe module $E$ over $L$ gives a map
\[\label{pseudoclassmapeq}
f^a:=(a ,f): Y\stackrel{}{\longrightarrow} BU(1)\times BU(n),
\]
which we call {\it pseudo-classifying map} of $E$. To emphasize the classifying maps, we denote the twisted Chern class $c_k^{l, \xi}(E)$ by $c_k^{l, a}(E)$ from now on. 

Denote by $X_{\mathbb{Q}}$ the rationalization of any nilpotent space $X$. 
Using Sullivan's rational homotopy theory \cite{FHT01}, we realize the twisted Chern classes at the level of classifying spaces. More precisely, we have the following theorem which will be proved in Subsection \ref{sec: fracclassifyingsp}. 
\begin{theorem}\label{relclassthmintro}
Let $n$ and $l$ be positive integers such that $l|n$, there is a canonical map 
\[
\phi_{l|n}: BU(1)\times BU(n) \stackrel{}{\longrightarrow} BU(n)_\mathbb{Q},
\]
such that for any bundle gerbe module $E$ of rank $n$ over a bundle gerbe $L$ of oder $l$,
\begin{itemize}
\item[(i)]~the pseudo-classifying map $f^a=(a, f)$ of $E$ descents to a map $f^l$ on $M$ through $\pi$ in the following homotopy commutative diagram
\begin{gather}
\begin{aligned}
\xymatrix{
Y \ar[d]^{\pi} \ar[r]^<<<<<{f^a} & 
BU(1)\times BU(n) \ar[d]^{\phi_{l|n}} \\
M\ar[r]^<<<<<<<<{f^l} &
BU(n)_\mathbb{Q},
}
\end{aligned}
\label{fracudiagintro}
\end{gather}
\item[(ii)]the $k$-th twisted Chern class $c_k^{l,a}(E)\in H^{2k}(M;\mathbb{Q})$ satisfies
 \[
c_k^{l,a}(E)=f^{l\ast}(c_k^{\q}),
\] 
for each $0\leq k\leq n$ with $c_k^{\q}$ being the $k$-th universal rational Chern class. 
\end{itemize}
\end{theorem}

Let $c_k\in H^{2k}(BU(n);\mathbb{Z})$ denote the $k$-th universal integral Chern class. Denote by $g$ the first Chern class $c_1\in H^2(BU(1);\mathbb{Z})$. It is clear that $f^{a\ast}(g)=a$.
The following theorem illustrates the precise relation between the twisted Chern classes, rational Chern classes and integral Chern classes. 

\begin{theorem} [\protect Lemma \ref{pichernlemma} and \ref{uniphiclasslemma}]\label{fracchernforthm}
With the same conditions and notations in Theorem \ref{relclassthmintro},
\begin{equation}\label{picherneqintro}
\begin{split}
&\phi_{l|n}^\ast(c_k^{\q})=\sum\limits_{i=0}^{k} \big(\frac{-1}{l}\big)^i \binom{n-k+i}{i} g^i c_{k-i},\\
&\pi^\ast(c_k^{l,a}(E))=\sum\limits_{i=0}^{k} \big(\frac{-1}{l}\big)^i \binom{n-k+i}{i} a^i c_{k-i}(E).
\end{split}
\end{equation}
\end{theorem}

On the level of characteristic classes, the pair of maps ($f^a$, $f^l$) can be viewed as the classifying map of the bundle gerbe module $E$ mapping to the relative classifying space, the map $BU(1)\times BU(n) \stackrel{\phi_{l|n}}{\longrightarrow} BU(n)_\mathbb{Q}$.  Motivated by the above results, we make a definition for the situation slightly more general than bundle gerbe modules, which extracts the homotopy information of bundle gerbe modules, and further allows us to introduce higher structures on bundle gerbe modules.

\begin{definition}\label{fracundefintro} Let $\pi: Y\to M$ be a locally split map and $\xi$ a complex line bundle over $Y$ classified by $a:Y\to BU(1)$. Let $E$ be a rank $n$ complex vector bundle over $Y$ classified by a map $f: Y\rightarrow BU(n)$. Let $l$ be a positive integer such that $l|n$. $E$ is called an {\em $(a,\frac{1}{l})$-fractional $U(n)$-bundle} over $M$, or simply a {\em fractional U-bundle} if the pseudo-classifying map $f^a=(a, f)$ descents to a map $f^l$ on $M$ through $\pi$ in a homotopy commutative diagram (\ref{fracudiagintro}). 
\end{definition}

We call the map $BU(1)\times BU(n) \stackrel{\phi_{l|n}}{\longrightarrow} BU(n)_\mathbb{Q}$ the {\it fractional classifying space} of fractional U-structure, and the pair of maps ($f^a$, $f^l$) the {\it classifying map} of the fractional bundle $E$ with the {\it relative structure group} ($U(1)\times U(n)$, $U(n)_{\mathbb{Q}}$). We define the $k$-th {\it fractional Chern class} $c_k^{l,a}(E)$ of the fractional U-bundle $E$ to be $f^{l\ast}(c_k^{\q})\in H^{2k}(M;\mathbb{Q})$,
for each $0\leq k\leq n$ with $c_k^{\q}$ being the $k$-th universal rational Chern class. 

By Theorem \ref{relclassthmintro}, one can see that when the fractional bundle $E$ comes from a bundle gerbe module, the fractional Chern classes coincide with the twisted Chern classes. This justifies our choice of notation. We summarize this immediate but important fact in the following theorem.
\begin{theorem}\label{E>=fracunthm}
Let $E$ be a bundle gerbe module of rank $n$ over a flat bundle gerbe $(L, Y)$ of order $l$. Then $E$ admits a fractional U-structure determined by Diagram (\ref{fracudiagintro}). Moreover, for each $0\leq k\leq n$, its $k$-th twisted Chern class coincides with its $k$-th fractional Chern class. ~$\qqed$
\end{theorem}
In the sequel we will not distinguish these two terminologies, and usually adopt fractional Chern class to emphasize its non-integrality illustrated explicitly in Theorem \ref{fracchernforthm}.

\subsection{Higher fractional structures}
\label{sec: highfracintro}
In this subsection, we discuss higher fractional structures as depicted in Diagram (\ref{summdiag}).
We first introduce the notion of fractional SU- and U$\langle6\rangle$-structures in terms of the
 fractional Chern classes. 
Then by explicitly constructing a relative Whitehead tower of the fractional classifying space up to level $2$, we characterize the homotopy aspects of these higher fractional structures in Theorem \ref{fracsunthm} and \ref{fracu6nthm}. 

Parallel to the classical higher structures on ordinary vector bundles as recapped in Subsection \ref{sec: backintro}, it is natural to define higher fractional structures on fractional U-bundles as follows,
\begin{definition}\label{highfracstrudef}
Suppose $E$ is a given $(a,\frac{1}{l})$-fractional $U(n)$-bundle determined by Diagram (\ref{fracudiagintro}). 
\begin{itemize}
\item[(i)]~
$E$ is called an {\it $(a,\frac{1}{l})$-fractional $SU(n)$-bundle}, or simply a {\it fractional SU-bundle} if its first fractional Chern class $c_1^{l,a}(E)=0$;
\item[(ii)]~
$E$ is called an {\it $(a,\frac{1}{l})$-fractional $U\langle6\rangle(n)$-bundle}, or simply a {\it fractional U$\langle6\rangle$-bundle} if it is $(a,\frac{1}{l})$-fractional SU and its second fractional Chern class $c_2^{l,a}(E)=0$.
\end{itemize}
\end{definition}
\begin{remark}\label{hmremark}
Before proceeding to the homotopy aspects of these higher fractional structures, we would like to remark that the degree 4 component of the twisted Chern character of a bundle gerbe module $E$
$$\mathrm{Ch}_L(E)^{[4]}=\frac{1}{2}(c_1^{l,a}(E)^2- 2c_2^{l,a}(E))$$
is the obstruction to the modularity of graded Chern character of Witten gerbe modules 
\be \Theta_2(E)=\bigotimes_{v=1}^\infty
\Lambda_{-q^{v-{1\over2}}}(E)\otimes \bigotimes_{v=1}^\infty
\Lambda_{-q^{v-{1\over2}}}(\bar E),\ \ \ \Theta_3(E)=\bigotimes_{v=1}^\infty
\Lambda_{q^{v-{1\over2}}}(E)\otimes \bigotimes_{v=1}^\infty
\Lambda_{q^{v-{1\over2}}}(\bar E),\ee
which are elements in $K(Y)[[q^{1/2}]]$ and appear in the construction of {\em projective elliptic genera} \cite{HM1, HM2} (when $E$ is an ordinary vector bundle over $M$, see Liu \cite{Liu}). Let $\{2\pi \sqrt{-1} x_i\},\, 1\leq i\leq n,$ be the formal Chern roots of $E$ and $q=e^{2\pi  \sqrt{-1}\tau}$, $\tau \in \bH$, the upper half plane. Recall $a=c_1(\xi)$. In terms of the Jacobi theta functions, the graded Chern characters of the Witten gerbe modules $\Theta_2(E)$ and $\Theta_3(E)$ can be expressed as
$$ \GCh (\Theta_2(E))=\prod_{i=1}^n \theta_2\left(x_i-\frac{a}{l}, \tau\right),$$ 
$$\GCh (\Theta_3(E))=\prod_{i=1}^n \theta_3\left(x_i-\frac{a}{l}, \tau\right).$$
They are cohomology classes in $H^{even}(Y)[[q^{1/2}]]$ and descend to $H^{even}(M)[[q^{1/2}]]$. When $\mathrm{Ch}_L(E)^{[4]}=0$, the normalized 
graded Chern characters $\frac{\GCh (\Theta_2(E))}{\theta_2^n(0, \tau)}, \frac{\GCh (\Theta_3(E))}{\theta_3^n(0, \tau)}$, which are $q$-series with coefficients in rational cohomology classes, have the property that the degree $2p$ component of them are modular forms of weight $p$ over certain index 2 subgroups of the modular group $SL(2, \Z)$. 
\end{remark}

To understand the higher fractional structures in Definition \ref{highfracstrudef} via classifying spaces, we will construct a relative Whitehead tower up to level $2$ for the fractional classifying space
$$BU(1)\times BU(n) \stackrel{\phi_{l|n}}{\longrightarrow} BU(n)_\mathbb{Q}.$$ 

By Theorem \ref{fracchernforthm}, we see that 
$$\phi_{l|n}^\ast(c^{\q}_1)=c_1-sg.$$ Let us first kill $c_1-sg$. Denote $s=\frac{n}{l}$. The $s$-th covering of $U(n)$ is a compact Lie group denoted by $U(n)_l$.
There is a group extension ((\ref{u1unexteq}))
\[\label{u1unexteqintro}
\{1\}\stackrel{}{\longrightarrow}U(n)_l \stackrel{i_{2l}}{\longrightarrow} U(1)\times U(n) \stackrel{}{\longrightarrow} U(1)\stackrel{}{\longrightarrow}\{1\},
\]
determined by the class $c_1-sg\in H^2(BU(1)\times BU(n);\mathbb{Z})$ ({\em here and later by writing numberings like ((\ref{u1unexteq})) and etc, we are indicating the places of origins for the reader's convenience}). The map $i_{2l}=(d_l, \rho_s)$ is explicitly defined in (\ref{unextdiag}).
As graded ring ((\ref{hbunleq}))
\begin{equation}\label{hbunleqintro}
H^\ast(BU(n)_l;\mathbb{Z})\cong \mathbb{Z}[\overline{c}_1, c_2,\ldots, c_n],
\end{equation}
such that ((\ref{fracsuchernreleq}))
\begin{equation}\label{bi2lclassintro}
Bi_{2l}^\ast(g)=\overline{c}_1, \ \ ~ Bi_{2l}^\ast(c_1)=s\overline{c}_1, \ \ ~ Bi_{2l}^\ast(c_i)=c_i, ~{\rm for}~{\rm each}~i\geq 2.
\end{equation}
For convenience denote $c_1:=s\overline{c}_1$ for later use. 

Let 
\begin{equation}\label{whtoweruqeq}
\cdots \stackrel{}{\longrightarrow}  BU\langle 6 \rangle(n)_{\mathbb{Q}}  \stackrel{Bi_{3\mathbb{Q}}}{\longrightarrow} 
BSU(n)_{\mathbb{Q}} \stackrel{Bi_{2\mathbb{Q}}}{\longrightarrow} 
BU(n)_{\mathbb{Q}}
\end{equation}
be the rationalization of the Whitehead tower (\ref{whtowerueq}). It follows that there is a homotopy pullback diagram ((\ref{unifracsudiag}))
\begin{gather*}
\begin{aligned}
\xymatrix{
BU(n)_{l\mathbb{Q}} \ar[r]^<<<<<<{Bi_{2l\mathbb{Q}}} \ar[d]^{\widetilde{\phi}_2}&
BU(1)_{\mathbb{Q}}\times BU(n)_{\mathbb{Q}} \ar[d]^{\phi_{l|n\mathbb{Q}}}\\
BSU(n)_{\mathbb{Q}} \ar[r]^{Bi_{2\mathbb{Q}}} &
BU(n)_{\mathbb{Q}},  
}
\end{aligned}
\label{unifracsudiagQ}
\end{gather*}
where $\phi_{l|n\mathbb{Q}}$ is the rationalization of $\phi_{l|n}$, and $\widetilde{\phi}_2$ is defined by the homotopy pullback. Define a map $\phi_2$ as the composition
\begin{equation}\label{phi2defintro}
\phi_2: BU(n)_l\stackrel{r}{\longrightarrow}BU(n)_{l\mathbb{Q}} \stackrel{\widetilde{\phi}_2}{\longrightarrow} BSU(n)_{\mathbb{Q}}, 
\end{equation}
where $r$ is the rationalization. The upshot is that the map $\phi_2$ is viewed as the level $1$ lifting of the fractional classifying space $\phi_{l|n}$.

To move up one level further, we need to kill the class ((\ref{phi2c2eq}))
\[
\phi_2^\ast(c^{\q}_2)=c_2-\frac{s(n-1)}{2l}\overline{c}_1^2.
\]
Construct the topological group $U\langle 6 \rangle (n)_l$ defined by a group extension ((\ref{u6nlexteq}); compare to its integral counterpart (\ref{u6exteqintro}))
\[\label{u6nlexteqintro}
\{1\}\stackrel{}{\longrightarrow} K(\mathbb{Q},2) \stackrel{}{\longrightarrow} U\langle 6 \rangle (n)_l\stackrel{i_{3l}}{\longrightarrow}  U(n)_l \stackrel{}{\longrightarrow} \{1\}
\]
corresponding to the rational class $c_2-\frac{s(n-1)}{2l}\overline{c}_1^2\in H^4(BU(n)_l;\mathbb{Q})$.
It follows that the homotopy pullback of $\phi_2$ along the lifting $Bi_{3\mathbb{Q}}$ in (\ref{whtoweruqeq})
defines a map
\[
\phi_3: BU\langle 6 \rangle(n)_l\to BU\langle 6 \rangle(n)_{\mathbb{Q}}.
\]
The homotopy pullback ((\ref{u6nldiag}))
\begin{gather*}
\begin{aligned}
\xymatrix{
U\langle 6 \rangle(n)_l \ar[d]^{}  \ar[r]^{i_{3l}}&
U(n)_l \ar[d]^{\Omega \phi_2} \\
U\langle 6 \rangle(n)_{\mathbb{Q}}  \ar[r]^{i_{3\mathbb{Q}}}&
SU(n)_{\mathbb{Q}},
}
\end{aligned}
\label{u6nldiagintro}
\end{gather*}
provides a more concrete construction of the topological group $U\langle 6 \rangle(n)_l$.
The upshot is that the map $\phi_3$ is viewed as the level $2$ lifting of the fractional classifying space $\phi_{l|n}$.

To summarize, we have constructed the relative Whitehead tower of the classifying space of the relative structure group ($U(1)\times U(n)$, $U(n)_{\mathbb{Q}}$)
\begin{gather}
\begin{aligned}
\xymatrix{
BU\langle 6 \rangle(n)_l \ar[d]^{\phi_3}  \ar[r]^{Bi_{3l}}&
BU(n)_l \ar[r]^<<<<<<{Bi_{2l}} \ar[d]^{\phi_2}&
BU(1)\times BU(n) \ar[d]^{\phi_{l|n}}\\
BU\langle 6 \rangle(n)_{\mathbb{Q}}  \ar[r]^{Bi_{3\mathbb{Q}}}&
BSU(n)_{\mathbb{Q}} \ar[r]^{Bi_{2\mathbb{Q}}} &
BU(n)_{\mathbb{Q}}.
}
\end{aligned}
\label{relwhtowerueq}
\end{gather}
The effect of the right square in Diagram (\ref{relwhtowerueq}) on characteristic classes is to kill the relative class $(c_1-sg, c^{\q}_1)\in (H^{2}(BU(1)\times BU(n)), H^2(BU(n)_{\mathbb{Q}}))$. 
The effect of the left square in Diagram (\ref{relwhtowerueq}) on characteristic classes is to kill the relative class $(c_2-\frac{s(n-1)}{2l}\overline{c}_1^2, c^{\q}_2)\in (H^{2}(BU(n)_l;\mathbb{Q}), H^2(BSU(n)_{\mathbb{Q}}))$.

Let us turn to fractional bundles. The following theorem characterizes the fractional SU-structures from the perspective of classifying spaces, and will be proved in Subsection \ref{sec: fracsu}. Recall by (\ref{hbunleqintro}) we have $\bar{c}_1$, and $c_i$ ($1\leq i\leq n$) as the classes of $BU(n)_l$ satisfying (\ref{bi2lclassintro}).
\begin{theorem}\label{fracsunthm}
Let $E$ be an $(a,\frac{1}{l})$-fractional $U(n)$-bundle determined by Diagram (\ref{fracudiagintro}) and $l>1$. 
Then $E$ admits an $(a,\frac{1}{l})$-fractional $SU(n)$-structure if and only if there exists a pair of maps $(f^a_2, f_2^l)$ such that the following diagram commutes up to homotopy
\begin{gather}
\begin{aligned}
\xymatrixcolsep{1.5pc}
\xymatrixrowsep{1.0pc}
\xymatrix{
& BU(n)_l \ar[dr]^{Bi_{2l}} \ar[dd]^>>>>>>{\phi_2} &\\
Y\ar@{.>}[ur]^<<<<<<{f^a_2} \ar[dd]^{\pi} \ar[rr]^>>>>>>>>{f^a}|!{[r];[r]}\hole && BU(1)\times BU(n) \ar[dd]^{\phi_{l|n}} \\
& BSU(n)_{\mathbb{Q}}\ar[dr]^{Bi_{2\mathbb{Q}}} \\
M\ar@{.>}[ur]^<<<<<<{f^l_2} \ar[rr]^{f^l}& & BU(n)_{\mathbb{Q}}.
}
\end{aligned}
\label{fracsunliftdiag}
\end{gather}
If such relative lift exists, then 
\begin{equation}\label{fracsunchernforeq}
f_2^{a\ast}(\overline{c}_1)=a=\frac{1}{s}c_1(E), \ \ ~ f_2^{a\ast}(c_i)=c_i(E), \ \ ~ f_2^{l\ast}(c_i^{\q})=c_i^{l,a}(E),~{\rm for}~{\rm each}~i\geq 2.
\end{equation}
Moreover, the fractional SU-structures on $E$ are in one-to-one correspondence with the elements of $H^1(M;\mathbb{Z})$.
\end{theorem}

Motivated by Theorem \ref{fracsunthm}, we call the map $BU(n)_l \stackrel{\phi_{2}}{\longrightarrow} BSU(n)_\mathbb{Q}$ the {\it fractional classifying space} of fractional SU-structure, and say that the relative structure group ($U(1)\times U(n)$, $U(n)_{\mathbb{Q}}$) of $E$ can be lifted to the pair of groups ($U(n)_l$, $SU(n)_{\mathbb{Q}}$). We call the pair of the liftings ($f_2^a$, $f_2^l$) in Diagram (\ref{fracsunliftdiag}) the {\it classifying map} of the fractional SU-structure. For a fractional U-bundle $E$ defined by Diagram (\ref{fracudiagintro}) (Lemma \ref{obfracsulemma})
\begin{equation}\label{fracsuobintro}
(c_1(E)-sa, c_1^{l,a}(E))=(f^{a\ast}, f^{l\ast})(c_1-sg, c^{\q}_1).
\end{equation} 

The following theorem characterizes the fractional U$\langle 6\rangle$-structures via classifying spaces, and will be proved in Subsection \ref{sec: fracu6}. To trace characteristic classes in the following diagram (\ref{fracu6nliftdiag}), denote $\bar{c}_1:=Bi_{3l}^\ast(\overline{c}_1)$ and $c_i:=Bi_{3l}^\ast(c_i)$ for each $1\leq i\leq n$ as classes of $BU\langle 6\rangle (n)_l$.

\begin{theorem}\label{fracu6nthm}
Let $E$ be an $(a,\frac{1}{l})$-fractional $SU(n)$-bundle determined by Diagram (\ref{fracsunliftdiag}) and $l>1$. 
Then $E$ admits an $(a,\frac{1}{l})$-fractional $U\langle 6\rangle (n)$-structure if and only if there exists a pair of maps $(f^a_3, f_3^l)$ such that the following diagram commutes up to homotopy
\begin{gather}
\begin{aligned}
\xymatrixcolsep{1.5pc}
\xymatrixrowsep{1.0pc}
\xymatrix{
& BU\langle 6\rangle (n)_l \ar[dr]^{Bi_{3l}} \ar[dd]^>>>>>>{\phi_3} &\\
Y\ar@{.>}[ur]^<<<<<<{f^a_3} \ar[dd]^{\pi} \ar[rr]^>>>>>>>>>>{f^a_2}|!{[r];[r]}\hole && BU(n)_l\ar[dd]^{\phi_2} \\
& BU\langle 6\rangle (n)_{\mathbb{Q}}\ar[dr]^{Bi_{3\mathbb{Q}}} \\
M\ar@{.>}[ur]^<<<<<<{f^l_3} \ar[rr]^{f^l_2}& & BSU(n)_{\mathbb{Q}}.
}
\end{aligned}
\label{fracu6nliftdiag}
\end{gather}
If such relative lift exists then 
\begin{equation}\label{fracu6nchernforeq}
\begin{split}
&f_3^{a\ast}(\bar{c}_1)=a=\frac{1}{s}c_1(E), \ \ ~f_3^{a\ast}(c_2)=c_2(E)=\frac{s(n-1)}{2l}a^2, \ \ ~ \\
&f_3^{a\ast}(c_i)=c_i(E), \ \ ~ f_3^{l\ast}(c_i^{\q})=c_i^{l,a}(E),~{\rm for}~{\rm each}~i\geq 3.
\end{split}
\end{equation}
Moreover, the fractional $U\langle 6\rangle$-structures on $E$ are in one-to-one correspondence with the elements of $H^3(M;\mathbb{Z})$.
\end{theorem}

Motivated by Theorem \ref{fracu6nthm}, we call the map $BU\langle 6\rangle (n)_l \stackrel{\phi_{3}}{\longrightarrow} BU\langle 6\rangle (n)_{\mathbb{Q}}$ the {\it fractional classifying space} of fractional U$\langle6\rangle$-structure, and say that the relative structure group ($U(n)_l$, $SU(n)_{\mathbb{Q}}$) of $E$ can be lifted to the pair of groups ($U\langle 6\rangle (n)_l$, $U\langle 6\rangle (n)_{\mathbb{Q}}$). We call the pair of the liftings ($f_3^a$, $f_3^l$) in Diagram (\ref{fracu6nliftdiag}) the {\it classifying map} of the fractional U$\langle6\rangle$-structure.
For a fractional SU-bundle $E$ defined by Diagram (\ref{fracsunliftdiag}) (Lemma \ref{obfracu6lemma})
\begin{equation}\label{fracu6obintro}
(c_2(E)-\frac{s(n-1)}{2l}a^2, c_2^{l,a}(E))=(f_2^{a\ast}, f_2^{l\ast})(c_2-\frac{s(n-1)}{2l}c_1^2, c^{\q}_2).
 \end{equation}

\subsection{Higher fractional loop structures} 
\label{sec: highloopfracuintro}
As in the untwisted cases, the fractional structures can be also understood from the perspective of free loop spaces as depicted in Diagram (\ref{summdiag}). Here we only study higher fractional loop structures from homotopy theoretical point of view. Geometric considerations on additional fusive structures will be explored in a separate forthcoming work.

More precisely, in this subsection we first introduce some necessary loop classes by the transgression procedure. Then we define the notion of fractional loop U- and SU-structures in terms of transgressed fractional characteristic classes. Lastly, by introducing a relative Whitehead tower of the looped fractional classifying space up to level $2$, we characterize higher fractional loop structures in Theorem \ref{fracloopunthm} and \ref{fracloopsunthm}.

To study looped characteristic classes, recall there is the canonical fibration
\begin{equation}\label{generalfreefibintro}
\Omega X\stackrel{}{\rightarrow} LX\stackrel{p}{\rightarrow} X,
\end{equation}
where $p$ is the evaluation map at the based point. It determines the transgression or the free suspension ((\ref{freesuspen}))
\[
\nu: H^\ast(X)\stackrel{}{\longrightarrow} H^{\ast-1}(LX),
\]
which is a derivation and is functorial in the space $X$.
It can be showed that ((\ref{h3bluneq}))
\[\label{h3bluneqintro}
H^{\leq 4}(BLU(n))\cong \mathbb{Z}^{\leq 4}[z_1,c_1,z_2,c_2]
\]
such that $c_i:=p^\ast(c_i)$ ($i=1$, $2$) correspond to the first two universal Chern classes, and (Lemma \ref{tranlemma})
\begin{equation}\label{traneqintro}
\nu(c_1)=z_1, \ \ \nu(c_2)=z_2+z_1c_1, \ \  \nu(c_1^2)=2z_1c_1.
\end{equation}
Similarly, there are universal loop classes given by
\[
H^{\leq 4}(BLU(n)_{\mathbb{Q}})\cong \mathbb{Z}^{\leq 4}[z_1^{\q},c_1^{\q},z_2^{\q},c_2^{\q}], ~ \ \  ~  H^\ast(BLU(1))\cong \mathbb{Z}[g]\{h\}~ {\rm such}~{\rm that}~\nu(g)=h.
\]
Let $E$ be a fractional U-bundle determined by Diagram (\ref{fracudiagintro}).
By naturality the above universal loop classes give the corresponding loop classes ($i=1$, $2$; cf. \hyperref[AppendixA]{table} and Subsection \ref{sec: freesus})
\[
\begin{split}
&c_i(LE):=p^\ast(c_i(E))\in H^{2i}(LY;\mathbb{Z}), ~  \  \ \ \ \  \ \ ~ z_i(LE)\in H^{2i-1}(LY;\mathbb{Z}),  \\
&c_i^{l,a}(LE):=p^\ast(c_i^{l,a}(E))\in H^{2i}(LM;\mathbb{Q}), ~ \ \  ~ z_i^{l,a}(LE)\in H^{2i-1}(LM;\mathbb{Q}),\\
&a:=p^\ast(a)\in H^{2}(LY;\mathbb{Z}),~ \ \  ~\mathfrak{a}\in H^{1}(LY;\mathbb{Z})~ {\rm such}~{\rm that}~\nu(a)=\mathfrak{a}.
\end{split}
\]

\begin{definition}\label{highfracloopstrudef}
Suppose $E$ is a given $(a,\frac{1}{l})$-fractional $U(n)$-bundle determined by Diagram (\ref{fracudiagintro}). 
\begin{itemize}
\item[(i)]~
$E$ has an {\it $(a,\frac{1}{l})$-fractional loop $U(n)$-structure}, or simply a {\it fractional loop U-structure} if its first transgressed fractional class $z_1^{l,a}(LE)=0$;
\item[(ii)]~
$E$ has an {\it $(a,\frac{1}{l})$-fractional loop $SU(n)$-structure}, or simply a {\it fractional loop SU-structure} if it is $(a,\frac{1}{l})$-fractional SU and its second transgressed fractional class $z_2^{l,a}(LE)=0$.
\end{itemize}
\end{definition}

To understand the higher fractional loop structures in Definition \ref{highfracloopstrudef} via classifying spaces, we construct a relative Whitehead tower up to level $2$ for the looped fractional classifying space
$$BLU(1)\times BLU(n) \stackrel{L\phi_{l|n}}{\longrightarrow} BLU(n)_\mathbb{Q}.$$

By Theorem \ref{fracchernforthm} and the naturality of the free suspension we see that 
\[
L\phi_{l|n}^\ast(z^{\q}_1)=z_1-sh. 
\]
Let us first kill $z_1-sh$. Recall $s=\frac{n}{l}$ and $l>1$. Consider the group $\overline{LU}(n)_l$ defined by a pullback diagram ((\ref{Lunpulldiag}))
\begin{gather*}
\begin{aligned}
\xymatrix{
\overline{LU}(n)_l \ar[r]^{} \ar[d]^{} &
\Omega U(1) \ar[d]^{\Omega\tau_s} \\
LU(n) \ar[r]^{\varepsilon} &
\Omega U(1),
}
\end{aligned}
\label{Lunpulldiagintro}
\end{gather*}
where $\Omega U(1)$ is a based loop group, $\varepsilon(A(t))=\frac{{\rm det}(A(t))}{{\rm det}(A(1))}$, and $\tau_s(z)=z^s$. 
If we view the components of $LU(n)$ are indexed by $\mathbb{Z}$ under the isomorphism $\pi_0(LU(n))\cong\mathbb{Z}$,
it is not hard to see that $\overline{LU}(n)_l$ is the subgroup of $LU(n)$ consisting of all the components indexed in $s\mathbb{Z}$.
There is a homotopy group extension ((\ref{lu1unexteq2}))
\[\label{lu1unexteq2intro}
\{1\}\stackrel{}{\longrightarrow} U(1)\times \overline{LU}(n)_l\stackrel{\iota_{2l}}{\longrightarrow} LU(1)\times LU(n)\stackrel{}{\longrightarrow} \Omega U(1)\stackrel{}{\longrightarrow}\{1\}
\]
determined by the class $z_1-sh\in H^1(BLU(1)\times BLU(n))$, and $\iota_{2l}=(\kappa_l, \psi_s)$ is explicitly defined in (\ref{u1lunextdiag}). 
For later use, denote 
\begin{equation}\label{barunlclassintro}
\overline{z}_1:=(B\iota_{2l})^\ast(h),  ~ \ \ ~ z_2:=(B\iota_{2l})^\ast(z_2), ~\ \ ~ c_i:=(B\iota_{2l})^\ast(c_i), ~\ ~ i=1, 2.
\end{equation}

Let
\begin{equation}\label{whtowerluqeq}
\cdots\stackrel{}{\longrightarrow} B\widehat{LSU}(n)_{\mathbb{Q}}\stackrel{B\iota_{3\mathbb{Q}}}{\longrightarrow} BLSU(n)_{\mathbb{Q}}\stackrel{B\widehat{Li_{2}}_\mathbb{Q}}{\longrightarrow}
BL_0U(n)_{\mathbb{Q}}\stackrel{B\iota_{2\mathbb{Q}}}{\longrightarrow} BLU(n)_{\mathbb{Q}}
\end{equation}
be the rationalization of the Whitehead tower (\ref{whtowerlueq}). 
Recall here $\iota_{2}: L_0U(n)\rightarrow LU(n)$ is the inclusion of the component corresponding to $0\in \mathbb{Z}\cong\pi_0(LU(n))$, and the loop map $Li_2: LSU(n)\rightarrow LU(n)$ ((\ref{whtowerueq})) factors through $L_0U(n)$ to define the map $\widehat{Li_2}: LSU(n)\rightarrow L_0U(n)$.
It follows that there is a homotopy pullback diagram ((\ref{lunifracsudiag}))
\begin{gather*}
\begin{aligned}
\xymatrix{
BU(1)_{\mathbb{Q}}\times B\overline{LU}(n)_{l\mathbb{Q}} \ar[r]^<<<<<{B\iota_{2l\mathbb{Q}}} \ar[d]^{\widetilde{\xi}_2}&
BLU(1)_{\mathbb{Q}}\times BLU(n)_{\mathbb{Q}}\ar[d]^{L\phi_{l|n\mathbb{Q}}}\\
BL_0U(n)_{\mathbb{Q}} \ar[r]^{B\iota_{2\mathbb{Q}}} &
BLU(n)_{\mathbb{Q}},
}
\end{aligned}
\label{lunifracsudiagQ}
\end{gather*}
where $\phi_{l|n\mathbb{Q}}$ is the rationalization of $\phi_{l|n}$, and $\widetilde{\xi}_2$ is defined by the homotopy pullback. Define a map $\xi_2$ as the composition
\begin{equation}\label{xi2defintro}
\xi_2: BU(1)\times B\overline{LU}(n)_{l}\stackrel{r}{\longrightarrow}BU(1)_{\mathbb{Q}}\times B\overline{LU}(n)_{l\mathbb{Q}} \stackrel{\widetilde{\phi}_2}{\longrightarrow} BL_0U(n)_{\mathbb{Q}}, 
\end{equation}
where $r$ is the rationalization. The upshot is that the map $\xi_2$ is viewed as the level $1$ lifting of the looped fractional classifying space $L\phi_{l|n}$.

In Definition \ref{highfracloopstrudef}, we define fractional loop SU-structure for fractional SU-bundles instead of fractional loop U-bundles. Accordingly, before moving up one level further to fractional loop SU, there is an intermediate step from fractional loop U to the looping of fractional SU. This is roughly depicted in the right diagram of (\ref{summdiag}), and is also clear from (\ref{whtowerluqeq}). 
The relation of these two structures is explicitly described in Subsection \ref{sec: fracloopuvssu}. In particular, in the universal case passing from fractional loop U to fractional SU is equivalent to kill the class (Lemma \ref{xi2classlemma})
\[
\xi_2^\ast(c^{\q}_1)=c_1-sg,
\]
where $c^{\q}_1=B\iota_{2\mathbb{Q}}^\ast(c^{\q}_1)\in H^2(BL_0U(n)_{\mathbb{Q}})$ ((\ref{h3bl0unqeq})). 
This can be achieved by the homotopy commutative diagram ((\ref{factorLunifracsudiag})) 
\begin{gather*}
\begin{aligned}
\xymatrix{
BLU(n)_l \ar[r]^<<<<<{B\widehat{Li_{2l}}} \ar[d]^{L\phi_2}&
BU(1)\times B\overline{LU}(n)_l  \ar[d]^{\xi_2}\\
BLSU(n)_{\mathbb{Q}} \ar[r]^{B\widehat{Li_2}_{\mathbb{Q}}} &
BL_0U(n)_{\mathbb{Q}},
}
\end{aligned}
\label{fracloopuvssuintro}
\end{gather*}
which means that the level $1$ lifting $\xi_2$ in the loop case can be pulled back to the level $1$ lifting $\phi_2$ after looping.
It can be showed that ((\ref{h3blunleq}))
\[\label{h3blunleqintro}
H^{\leq 4}(BLU(n)_l)\cong \mathbb{Z}^{\leq 4}[\overline{z}_1,\overline{c}_1,z_2,c_2],
\]
such that ((\ref{bhatl2ilclasseq}))
\[\label{blrhosclasseqintro}
(B\widehat{Li_{2l}})^\ast(\overline{z}_1)=\overline{z}_1, (B\widehat{Li_{2l}})^\ast(g)=\overline{c}_1, (B\widehat{Li_{2l}})^\ast(z_2)=z_2,(B\widehat{Li_{2l}})^\ast(c_2)=c_2.
\]

We can now move up one level further from $L\phi_2$. For this we need to kill the class ((\ref{lphi2z2eq})) 
\[
L\phi_2^\ast(z^{\q}_2)=z_2+\frac{s}{l}\overline{z}_1\overline{c}_1.
\]
Construct the topological group $\widehat{LSU}(n)_l $ defined by a group extension ((\ref{hatsunlexteq}); compare to its integral counterpart (\ref{lsuhatexteqintro}))
\[\label{hatsunlexteqintro}
\{1\}\stackrel{}{\longrightarrow} K(\mathbb{Q},1) \stackrel{}{\longrightarrow} \widehat{LSU}(n)_l \stackrel{\iota_{3l}}{\longrightarrow}  LU(n)_l  \stackrel{}{\longrightarrow} \{1\}
\]
corresponding to the rational class $z_2+\frac{s}{l}\overline{z}_1\overline{c}_1\in H^3(BLU(n)_l;\mathbb{Q})$.
It follows that the homotopy pullback of $L\phi_2$ along the lifting $B\iota_{3\mathbb{Q}}$ in (\ref{whtowerluqeq})
defines a map
\[
\xi_3: B\widehat{LSU}(n)_l\to B\widehat{LSU}(n)_{\mathbb{Q}}.
\] 
The homotopy pullback ((\ref{hatsunldiag}))
\begin{gather*}
\begin{aligned}
\xymatrix{
\widehat{LSU}(n)_l \ar[d]^{\Omega\xi_3}  \ar[r]^{\iota_{3l}}&
LU(n)_l \ar[d]^{\Omega L\phi_2} \\
\widehat{LSU}(n)_{\mathbb{Q}}  \ar[r]^{ \iota_{3\mathbb{Q}}}&
LSU(n)_{\mathbb{Q}},
}
\end{aligned}
\label{hatsunldiagintro}
\end{gather*}
provides a more concrete construction of the topological group $\widehat{LSU}(n)_l$. The upshot is that the map $\xi_3$ is viewed as the level $2$ lifting of the looped fractional classifying space $L\phi_{l|n}$.

To summarize, we have constructed the relative Whitehead tower of the classifying space of the relative structure group ($LU(1)\times LU(n)$, $LU(n)_{\mathbb{Q}}$)
\begin{gather}
\begin{aligned}
\xymatrix{
B\widehat{LSU}(n)_l \ar[d]^{\xi_3}  \ar[r]^{B\iota_{3l}}&
BLU(n)_l \ar[r]^<<<<<{B\widehat{Li_{2l}}} \ar[d]^{L\phi_2}&
BU(1)\times B\overline{LU}(n)_l \ar[r]^<<<<{B\iota_{2l}} \ar[d]^{\xi_2}&
BLU(1)\times BLU(n) \ar[d]^{L\phi_{l|n}}\\
B\widehat{LSU}(n)_{\mathbb{Q}}  \ar[r]^<<<<{B \iota_{3\mathbb{Q}}}&
BLSU(n)_{\mathbb{Q}} \ar[r]^{B\widehat{Li_2}_{\mathbb{Q}}} &
BL_0U(n)_{\mathbb{Q}} \ar[r]^{B\iota_{2\mathbb{Q}}} &
BLU(n)_{\mathbb{Q}}.
}
\end{aligned}
\label{relwhtowerlueqintro}
\end{gather}
The effect of the rightmost square in Diagram (\ref{relwhtowerlueqintro}) on characteristic classes is to kill the relative class $(z_1-sh, z^{\q}_1)\in (H^{1}(BLU(1)\times BLU(n)), H^1(BLU(n)_{\mathbb{Q}}))$.
The effect of the middle square in Diagram (\ref{relwhtowerlueqintro}) on characteristic classes is to kill the relative class $(c_1-sg, c^{\q}_1)\in (H^{1}(BU(1)\times B\overline{LU}(n)_l, H^1(BL_0U(n)_{\mathbb{Q}}))$.
The effect of the leftmost square in Diagram (\ref{relwhtowerlueqintro}) on characteristic classes is to kill the relative class $(z_2+\frac{s}{l}\overline{z}_1\overline{c}_1, z^{\q}_2)\in (H^{3}(BLU(n)_l), H^4(BLSU(n)_{\mathbb{Q}}))$.

Let us turn to the case of fractional bundles after applying the free loop functor to Diagram (\ref{fracudiagintro}).
The following theorem characterizes fractional loop U-structures from the perspective of classifying spaces, and will be proved in Subsection \ref{sec: fracloopu}.
Recall by (\ref{barunlclassintro}) we have $\bar{z}_1$, $z_2$, $c_1$, $c_2$ as the classes of $B\overline{LU}(n)_l$. 
\begin{theorem}\label{fracloopunthm}
Let $E$ be an $(a,\frac{1}{l})$-fractional $U(n)$-bundle determined by Diagram (\ref{fracudiagintro}) and $l>1$. 
Then $E$ admits an $(a,\frac{1}{l})$-fractional loop $U(n)$-structure if and only if there exists a pair of maps $(g^a_2, g_2^l)$ such that the following diagram commutes up to homotopy
\begin{gather}
\begin{aligned}
\xymatrixcolsep{1.4pc}
\xymatrixrowsep{1.0pc}
\xymatrix{
& BU(1)\times B\overline{LU}(n)_l \ar[dr]^{B\iota_{2l}} \ar[dd]^>>>>>>{\xi_2} &\\
LY\ar@{.>}[ur]^<<<<<<{g^a_2} \ar[dd]^{L\pi} \ar[rr]^>>>>>>>>>>{Lf^a}|!{[r];[r]}\hole && BLU(1)\times BLU(n) \ar[dd]^{L\phi_{l|n}} \\
& BL_0U(n)_{\mathbb{Q}}\ar[dr]^{B\iota_{2\mathbb{Q}}} \\
LM\ar@{.>}[ur]^<<<<<<{g^l_2} \ar[rr]^{Lf^l}& & BLU(n)_{\mathbb{Q}}.
}
\end{aligned}
\label{fracloopunliftdiag}
\end{gather}
If such relative lift exists then 
\begin{equation}\label{fracloopunchernforeq}
\begin{split}
&g_2^{a\ast}(g)=a, \ \ ~ g_2^{a\ast}(\overline{z}_1)=\mathfrak{a}=\frac{1}{s}z_1(LE), \ \ ~ 
 g_2^{a\ast}(z_2)=z_2(LE),\ \ ~ \\
&g_2^{a\ast}(c_1)=c_1(LE),  \ \ ~ 
g_2^{a\ast}(c_2)=c_2(LE), \\
& g_2^{l\ast}(z^{\q}_2)=z^{l,a}_2(LE), \ \ ~
~g_2^{l\ast}(c^{\q}_i)=c^{l,a}_i(LE), ~ \  ~ i=1, 2.
\end{split}
\end{equation}
Moreover, the fractional loop U-structures on $E$ are in one-to-one correspondence with the elements of $H^0(LM;\mathbb{Z})$.
\end{theorem}

Motivated by Theorem \ref{fracloopunthm}, we call the map $BU(1)\times B\overline{LU}(n)_l \stackrel{\xi_{2}}{\longrightarrow} BL_0U(n)_\mathbb{Q}$ the {\it fractional classifying space} of fractional loop U-structure, and say that the relative structure group ($LU(1)\times LU(n)$, $LU(n)_{\mathbb{Q}}$) of $LE$ can be lifted to the pair of groups ($U(1)\times \overline{LU}(n)_l $, $L_0U(n)_{\mathbb{Q}}$). We call the pair of the liftings ($g_2^a$, $g_2^l$) in Diagram (\ref{fracloopunliftdiag}) the {\it classifying map} of the fractional loop U-structure.
For a fractional U-bundle $E$ defined by Diagram (\ref{fracudiagintro}) after looping (Lemma \ref{obfracloopulemma})
\begin{equation}\label{fracloopuobintro}
(z_1(LE)-s\mathfrak{a}, z_1^{l,a}(LE))=(Lf^{a\ast}, Lf^{l\ast})(z_1-sh, z^{\q}_1).
\end{equation}

The following theorem characterizes fractional loop SU-structures from the perspective of classifying spaces, and will be proved in Subsection \ref{sec: fracloopsu}. To trace the characteristic classes through (\ref{fracloopsunliftdiag}), denote $\bar{z}_1=B\iota_{3l}^\ast(\overline{z}_1)$, $\bar{c}_1=B\iota_{3l}^\ast(\overline{c}_1)$, $z_2=B\iota_{3l}^\ast(z_2)$, and $c_2=B\iota_{3l}^\ast(c_2)$ as classes of $B\widehat{LSU}(n)_l$.

\begin{theorem}\label{fracloopsunthm}
Let $E$ be an $(a,\frac{1}{l})$-fractional $SU(n)$-bundle determined by Diagram (\ref{fracsunliftdiag}) and $l>1$. 
Then $E$ admits an $(a,\frac{1}{l})$-fractional loop $SU(n)$-structure if and only if there exists a pair of maps $(g^a_3, g_3^l)$ such that the following diagram commutes up to homotopy
\begin{gather}
\begin{aligned}
\xymatrixcolsep{1.5pc}
\xymatrixrowsep{1.0pc}
\xymatrix{
& B\widehat{LSU}(n)_l \ar[dr]^{B\iota_{3l}} \ar[dd]^>>>>>{\xi_3} &\\
LY\ar@{.>}[ur]^<<<<<<{g^a_3} \ar[dd]^{L\pi} \ar[rr]^>>>>>>>>>>{Lf^a_2}|!{[r];[r]}\hole && BLU(n)_l\ar[dd]^{L\phi_2} \\
& B\widehat{LSU}(n)_{\mathbb{Q}}\ar[dr]^{B\iota_{3\mathbb{Q}}} \\
LM\ar@{.>}[ur]^<<<<<<{g^l_3} \ar[rr]^{Lf^l_2}& & BLSU(n)_{\mathbb{Q}}.
}
\end{aligned}
\label{fracloopsunliftdiag}
\end{gather}
If such relative lift exists then 
\begin{equation}\label{fracloopsunchernforeq}
\begin{split}
&g_3^{a\ast}(\bar{z}_1)=\mathfrak{a}=\frac{1}{s}z_1(LE),  \ \ ~ g_3^{a\ast}(\bar{c}_1)=a=\frac{1}{s}c_1(LE),    \\
&g_3^{a\ast}(z_2)=z_2(LE)=-\frac{1}{n}z_1(LE)c_1(LE)=-\frac{s}{l}\mathfrak{a}a,\\
&g_3^{a\ast}(c_2)=c_2(LE), \ \ ~g_3^{l\ast}(c^{\q}_2)=c^{l,a}_2(LE).
\end{split}
\end{equation}
Moreover, the fractional loop SU-structures on $E$ are in one-to-one correspondence with the elements of $H^2(LM;\mathbb{Z})$.
\end{theorem}
Motivated by Theorem \ref{fracloopsunthm}, we call the map $B\widehat{LSU}(n)_l\stackrel{\xi_{3}}{\longrightarrow} B\widehat{LSU}(n)_{\mathbb{Q}}$ the {\it fractional classifying space} of fractional loop SU-structure, and say that the relative structure group ($LU(n)_l$, $LSU(n)_{\mathbb{Q}}$) of $LE$ can be lifted to the pair of groups ($\widehat{LSU}(n)_l$, $\widehat{LSU}(n)_{\mathbb{Q}}$). We call the pair of the liftings ($g_3^a$, $g_3^l$) in Diagram (\ref{fracloopsunliftdiag}) the {\it classifying map} of the fractional loop SU-structure.
For a fractional SU-bundle $E$ defined by Diagram (\ref{fracsunliftdiag}) we have (Lemma \ref{obfracloopsulemma})
\begin{equation}\label{fracloopsuobintro}
\begin{split}
&(z_2(LE)+\frac{1}{n}z_1(LE)c_1(LE),z_2^{l,a}(LE))
=(z_2(LE)+\frac{s}{l}\mathfrak{a}a, z_2^{l,a}(LE))\\
&=(Lf_2^{a\ast}, Lf_2^{l\ast}) (z_2+\frac{s}{l}\overline{z}_1\overline{c}_1, z^{\q}_2).
\end{split}
 \end{equation}

\subsection{Higher fractional structures vs higher fractional loop structures} 
\label{sec: highfracvsintro}
Recall that the hierarchy of fractional loop and non-loop structures is roughly depicted in the right diagram of (\ref{summdiag}).
There are two ways to compare the fractional non-loop structures with the fractional loop structures, either from the perspective of classifying spaces or from the perspective of free suspension.

From the perspective of classifying spaces, we note that the bottom row of Diagram (\ref{relwhtowerlueqintro}) is a refinement of the bottom row of Diagram (\ref{relwhtowerueq}) after looping, if we extend Diagram (\ref{relwhtowerlueqintro}) one step further to the left to reach $BLU\langle 6\rangle(n)_\mathbb{Q}$. In Section \ref{sec: weakfrac} this point will be clearly explained.

Moreover, in Theorem \ref{fracloopuvssuthm1} we show that for a given fractional U-bundle $E$, if $E$ admits a fractional SU-structure, then $E$ is fractional loop U. Conversely, suppose $E$ admits a fractional loop U-structure, then it can be lifted to a fractional SU-structure if and only if the homotopy lifting problem
\begin{gather*}
\begin{aligned}
\xymatrix{
& BLSU(n)_{\mathbb{Q}} \ar[dr]^{B\widehat{Li_2}_{\mathbb{Q}}} \\
LM \ar@{.>}[ur]^<<<<<<<<{} \ar[rr]^<<<<<<<<<<<<<<<<<<{g_2^{l}}&&
BL_0U(n)_{\mathbb{Q}}
}
\end{aligned}
\label{fracloopun+liftdiagintro}
\end{gather*}
has a solution, where $g_2^l$ is one component of the classifying map for the given fractional loop U-structure. 

Similarly, in Theorem \ref{fracloopsuvsu6thm1} we show that for a given fractional SU-bundle $E$, if $E$ admits a fractional U$\langle6\rangle$-structure, then $E$ is fractional loop SU. Conversely, suppose $E$ admits a fractional loop SU-structure, then it can be lifted to a fractional U$\langle6\rangle$-structure if and only if the homotopy lifting problem
\begin{gather*}
\begin{aligned}
\xymatrix{
& BLU\langle 6 \rangle(n)_{\mathbb{Q}}\ar[dr]^{B\widehat{Li_3}_{\mathbb{Q}}} \\
LM \ar@{.>}[ur]^<<<<<<<<{} \ar[rr]^<<<<<<<<<<<<<<<<<<{g_3^{l}}&&
B\widehat{LSU}(n)_{\mathbb{Q}}
}
\end{aligned}
\label{fracloopsun+liftdiagintro}
\end{gather*}
has a solution, where $g_3^l$ is one component of the classifying map for the given fractional loop SU-structure. 

From the perspective of free suspension, one can not only compare the existence of fractional loop or non-loop structures, but can also compare the amounts of possible structures. 

Firstly, the obstructions to the non-loop structures are transgressed by the free suspension $\nu$  to the obstructions to loop structures. Indeed, by the formulae of free suspension (\ref{traneqintro}) and its naturality, it is straightforward to show that 
for a fractional U-bundle $E$, the obstructions to fractional SU (\ref{fracsuobintro}) and loop U (\ref{fracloopuobintro}) satisfy
\[
\nu(c_1(E)-sa, c_1^{l,a}(E))=(z_1(LE)-s\mathfrak{a}, z_1^{l,a}(LE)),
\]
while for a fractional SU-bundle $E$, the obstructions to fractional U$\langle6\rangle$ (\ref{fracu6obintro}) and loop SU (\ref{fracloopsuobintro}) satisfy
\[
\nu(c_2(E)-\frac{s(n-1)}{2l}a^2, c_2^{l,a}(E))=(z_2(LE)+\frac{1}{n}z_1(LE)c_1(LE),z_2^{l,a}(LE)).
\]
In particular, the existence of fractional SU and U$\langle6\rangle$ structures implies the existence of fractional loop U and SU structures respectively. The converse statement may not be true in general. Nevertheless, it is true if one imposes constraints on the topology of $M$ as showed in Theorem \ref{fracloopuvssuthm2} and Theorem \ref{fracloopsuvsu6thm2}. In particular, when $M$ is rationally simply connected, fractional loop U implies fractional SU, while when $M$ is rationally $2$-connected, fractional SU implies fractional U$\langle6\rangle$. Hence in either situation the fractional loop structure is equivalent to its non-loop counterpart.

Secondly, as showed in Theorem \ref{fracloopuvssuthm2} and Theorem \ref{fracloopsuvsu6thm2}, the free suspension $\nu$ transgresses the parameter spaces of non-loop structures to those of loop structures. For instance, when $M$ is connected, there is a one-to-one correspondence between the distinct fractional SU-structures and the fractional loop U-structures on $E$ through the isomorphic free suspension
\[
\nu: H^1(M;\mathbb{Z})\stackrel{\cong}{\longrightarrow}H^0(LM;\mathbb{Z}).
\]
When $M$ is $2$-connected, there is a one-to-one correspondence between the distinct fractional U$\langle 6\rangle$-structures and the fractional loop SU-structures on $E$ through the isomorphic free suspension
\[
\nu: H^3(M;\mathbb{Z})\stackrel{\cong}{\longrightarrow}H^2(LM;\mathbb{Z}).
\] 

\numberwithin{equation}{section}
\numberwithin{theorem}{section}


\section{Bundle gerbe modules and splitting principle}
\label{sec: split}
Let $\pi: Y\rightarrow M$ be a {\it locally split map}. Denote by $Y^{[2]}=Y\times_{\pi} Y$ the fiber product of $Y$ with itself over $\pi$. Let $(L, Y)$ be a {\it bundle gerbe}, where $L$ is a complex line bundle $L\rightarrow Y^{[2]}$. There is the {\it Dixmier-Douady class} $d(L)=d(L, Y)\in H^3(M;\mathbb{Z})$ served as the obstruction to the gerbe being trivial. 
Suppose $d(L)$ is a torsion class of order $l$. In this case, the bundle gerbe $(L,Y)$ is called a {\it flat bundle gerbe}.

Since $l d(L)=0$, there exists a class $a\in H^2(Y;\mathbb{Z})$ corresponding to a complex line bundle $L_a$ over $Y$ such that 
\[\label{atrivialeq}
L^l\cong \delta(L_a):=\pi_1^\ast(L_a^\ast)\otimes \pi_2^\ast(L_a)
\]
is a {\it trivialization} of the bundle gerbe $(L^l, Y)$, the $l$-th tensor power of $L$, where $\pi_i: Y^{[2]}\rightarrow Y$ is the map which omits the $i$-th element, $L_a^\ast$ is the dual of $L_a$. The choice of such class $a$ is not unique in general. Indeed, suppose $x\in H^2(M;\mathbb{Z})$
 is any class with its associated complex line bundle $J_x$ over $M$. Then 
$$K:= L_a\otimes\pi^\ast(J_x)$$ is another trivialization of $L^l$ corresponding to the class $a+\pi^\ast(x)\in H^2(Y;\mathbb{Z})$. Conversely, any two trivializations differ from each other by a pull-back of a complex bundle over $B$.

Let $E\rightarrow Y$ be a finite rank complex vector bundle. Suppose that $E$ is a {\it bundle gerbe module} of $L$, in the sense that, there is a complex bundle isomorphism
\[
L\otimes \pi_1^\ast(E)\cong \pi_2^\ast(E).
\]
It is clear that $E^{\otimes l}$ is a bundle gerbe module over the trivial bundle gerbe $L^l$. Recall $L_a$ is a chosen trivialization of $L^l$, then the tensor bundle $E^{\otimes l}\otimes L_a^\ast$ over $Y$ descents to a bundle over $M$ denoted by $E^{\otimes l}/L_a$, that is,
\begin{equation}\label{ldescenteq}
\pi^\ast(E^{\otimes l}/L_a)=E^{\otimes l}\otimes L_a^\ast.
\end{equation}

In \cite{Tom} Tomoda has proved the splitting principle of bundle gerbe modules, which allows him to reproduce the twisted Chern classes and the twisted Chern character from the perspective of manifold topology. Indeed, for the bundle gerbe $L$ with the trivialization $L_a$ of $L^l$ and the bundle gerbe module $E$ of rank $n$, Tomoda constructed a commutative diagram of projective bundles 
\begin{gather}
\begin{aligned}
\xymatrix{
\widetilde{\mathbb{P}}(E) \ar[d]^{\Pi} \ar[r]^{\tilde{p}} &
Y \ar[d]^{\pi}\\
\mathbb{P}(E) \ar[r]^{p} &
M,
}
\end{aligned}
\label{relPEdiag}
\end{gather}
where $\widetilde{\mathbb{P}}(E)$ is the canonical {\it projectivization} of $E$ over $Y$ which descents to a {\it projectivization} $\mathbb{P}(E)$ over $M$ by the module structure of $E$. In particular, $\tilde{p}$ and $p$ are fibre bundles with fibre the complex projective space $P^{n-1}(\mathbb{C})$, and the square (\ref{relPEdiag}) is a pullback. Moreover $(\tilde{p}^\ast(L), \widetilde{\mathbb{P}}(E))$ is a bundle gerbe as the pullback of $(L, Y)$.
Let $\gamma_{E}$ be the {\it tautological line bundle} over $\widetilde{\mathbb{P}}(E)$. It is a rank $1$ bundle gerbe module of $\tilde{p}^\ast(L)$. Define the {\it twisted Euler class} $e^a(\gamma_{E})$ by
\begin{equation}\label{twisteeq}
e^a(\gamma_{E}):=\frac{1}{l}e\big((\gamma_{E}^{\otimes l})/\tilde{p}^\ast(L_a)\big)\in H^2(\mathbb{P}(E);\mathbb{Q}),
\end{equation}
where $e\big((\gamma_{E}^{\otimes l})/\tilde{p}^\ast(L_a)\big)$ is the Euler class of $(\gamma_{E}^{\otimes l})/\tilde{p}^\ast(L_a)$ defined by (\ref{ldescenteq}) for the bundle gerbe module $\gamma_{E}$ of $\tilde{p}^\ast(L)$ with a trivialization $\tilde{p}^\ast(L_a)$ of $\tilde{p}^\ast(L)^l$. By the Leary-Hirsch theorem, it can be showed that there exists a unique $(n+1)$-tuple
\[\label{tcherneq}
(c_0^{l,a}(E)=1, c_1^{l,a}(E),\ldots, c_n^{l,a}(E))\in \prod_{k=0}^{n} H^{2k}(M;\mathbb{Q}),
\]
such that
\begin{equation}\label{tchernreleq}
\sum\limits_{k=0}^{n} (-1)^k p^\ast(c_k^{l,a}(E)) \cup  e^a(\gamma_{E})^{n-k}=0.
\end{equation}
It is convenient to set $c_k^{l,a}(E)=0$ when $k>n$.
\begin{definition}\label{fraccherndef}
Let $(L,Y)$ be a flat bundle gerbe of order $l$ with a bundle gerbe module $E$ of rank $n$. Suppose $L_a$ is a trivialization of $L$. 
We define 
\[
c_k^{l,a}(E)\in H^{2k}(M;\mathbb{Q})
\]
satisfying (\ref{tchernreleq}) to be the {\it $k$-th fractional Chern classes} of $E$. Moreover, define the {\it total fractional Chern classes} of $E$ by
\[
c^{l,a}(E)=1+c_1^{l,a}(E)+\cdots+c_n^{l,a}(E).
\]
\end{definition}
\begin{remark}
The class $c_k^{l,a}(E)$ was usually referred to as the {\it twisted Chern class} of the bundle gerbe module $E$.
Here, we prefer the terminology fractional Chern class in order to emphasize that the image $\pi^\ast(c_k^{l,a}(E))$ is generally a fractional class. We will see this in Section \ref{sec: fchern} in details, and also its influence on the universal construction of the fractional Chern classes in Section \ref{sec: Bfchern}.
\end{remark}


\section{Fractional Chern classes} \label{sec: fchern}
In this section, based on the material in Section \ref{sec: split} we show explicit formulae of the fractional Chern classes defined in Definition \ref{fraccherndef}.

Let $(L,Y)$ be a flat bundle gerbe of order $l$ with a bundle gerbe module $E$ of rank $n$. Suppose $L_a$ is a trivialization of $L$. We have the fractional Chern classes $c_k^{l,a}(E)$ defined by the equation (\ref{tchernreleq}). Since the top exterior product bundle $\wedge^n(E)$ is a trivialization of $(L^n, Y)$
\[\label{wedgeneeq}
L^n\cong \delta(\wedge^n(E))=\pi_1^\ast(\wedge^n(E)^\ast)\otimes \pi_2^\ast(\wedge^n(E)),
\]
one has $nd(L)=0$. Since $l$ is the order of the torsion class $d(L)$, it follows that $l~|~n$. Set $n=ls$. 

By (\ref{twisteeq}), (\ref{ldescenteq}) and Diagram (\ref{relPEdiag}), we see that 
\begin{equation}\label{evsrooteq}
\begin{split}
\Pi^\ast(e^a(\gamma_{E}))
&=\frac{1}{l} e \Big(\Pi^\ast\big((\gamma_{E}^{\otimes l})/\tilde{p}^\ast(L_a)\big)\Big)\\
&=\frac{1}{l} e \big((\gamma_{E}^{\otimes l})\otimes \tilde{p}^\ast(L_a)^\ast\big)\\
&=\tilde{p}^\ast(c_1(E)-\frac{1}{l}a).
\end{split}
\end{equation}
Let $x_i^{l,a}$ ($1\leq i\leq n$) be the {\it formal fractional Chern roots} of the bundle gerbe module $E$, that is 
\begin{equation}\label{fracsymkeq}
c_k^{l,a}(E)=\sigma_k(x_1^{l,a},\ldots, x_n^{l,a}),
\end{equation}
where $\sigma_k$ is the $k$-th elementary symmetric polynomial in $n$ variables. Let $x_i$ ($1\leq i\leq n$) be the formal Chern roots of the complex vector bundle $E$,
that is 
\[
c_k(E)=\sigma_k(x_1,\ldots, x_n).
\]
Then by the above computation, we can choose the roots such that 
\begin{equation}\label{pichernrooteq}
\pi^\ast(x_i^{l, a})=x_i-\frac{1}{l}a.
\end{equation}
The pullback of the fractional Chern classes through $\pi$ satisfies
\begin{equation}\label{picherneq1}
\pi^\ast(c^{l,a}(E))=\prod_{i=1}^{n}(1+x_i-\frac{1}{l}a).
\end{equation}
\begin{lemma}\label{pichernlemma}
For each $1\leq k\leq n$,
\[
\pi^\ast(c_k^{l,a}(E))=\sum\limits_{i=0}^{k} \big(\frac{-1}{l}\big)^i \binom{n-k+i}{i} a^i c_{k-i}(E).
\]
\end{lemma}
\begin{proof}
By (\ref{picherneq1}), we can compute
\[
\begin{split}
\pi^\ast(c_k^{l,a}(E))
&=\sum\limits_{1\leq i_1<\ldots <i_k\leq n} (x_{i_1}-\frac{a}{l}) \cdots (x_{i_k}-\frac{a}{l}) \\
&=\sum\limits_{1\leq i_1<\ldots <i_k\leq n} \sum\limits_{i=0}^{k} \sum\limits_{1\leq j_1<\ldots <j_{k-i}\leq n \atop \{j_1,\ldots, j_{k-i}\} \subseteq \{i_1,\ldots, i_k\}} \big(\frac{-a}{l}\big)^i (x_{j_1}\cdots x_{j_{k-i}})\\
&=\sum\limits_{i=0}^{k} \big(\frac{-a}{l}\big)^i \frac{\binom{n}{k} \binom{k}{k-i}}{\binom{n}{k-i}}c_{k-i}(E)\\
&=\sum\limits_{i=0}^{k} \big(\frac{-a}{l}\big)^i \binom{n-k+i}{i} c_{k-i}(E).
\end{split}
\]
\end{proof}
In particular, one can compute the first two fractional Chern classes explicitly,
\begin{equation}\label{12fraccherneq}
\begin{split}
\pi^\ast(c_1^{l,a}(E))&=c_1(E)-sa,\\
\pi^\ast(c_2^{l,a}(E))&=c_2(E)-\frac{n-1}{l}a c_1(E)+\frac{s(n-1)}{2l}a^2.
\end{split}
\end{equation}

If a complex line bundle $L_b$ over $Y$, corresponding to $b\in H^2(Y;\mathbb{Z})$, gives another trivialization of $L^l$, then 
\[\label{abxeq}
b=a+\pi^\ast(x)
\]
 for some $x\in H^2(M;\mathbb{Z})$. By \cite[Proposition 5]{Tom}, $e^b(\gamma_{E})=e^a(\gamma_{E})-\frac{1}{l}p^\ast(x)$. It follows from (\ref{evsrooteq}) and (\ref{pichernrooteq}) that for the fractional formal Chern roots $\pi^\ast(x_i^{l,b})=x_i-\frac{1}{l}b=x_i-\frac{1}{l}a-\frac{1}{l}\pi^\ast(x)$, and 
\[
x_i^{l,b}=x_i^{l,a}-\frac{1}{l}x.
\]
Then by the same computation in the proof of Lemma \ref{pichernlemma}, we see that
\[\label{fracchernchangeeq}
c_k^{l,b}(E)=\sum\limits_{i=0}^{k} \big(\frac{-1}{l}\big)^i \binom{n-k+i}{i} x^i c_{k-i}^{l,a}(E),
\]
for each $1\leq k\leq n$. In particular,
\[\label{12fracchernchangeeq}
\begin{split}
c_1^{l,b}(E)&=c_1^{l,a}(E)-sx,\\
c_2^{l,b}(E)&=c_2^{l,a}(E)-\frac{n-1}{l}x c_1^{l,a}(E)+\frac{s(n-1)}{2l}x^2.
\end{split}
\]


\section{Fractional classifying spaces and fractional U-structure}\label{sec: Bfchern}
In this section we realize the fractional Chern classes in Section \ref{sec: fchern} from the perspective of classifying spaces, from which we propose a notion of fractional U-structure. In particular we prove Theorem \ref{relclassthmintro}.

Let $(L,Y)$ be a flat bundle gerbe of order $l$ with a bundle gerbe module $E$ of rank $n$. Suppose $L_a$ is a trivialization of $L$. We have the fractional Chern classes $c_k^{l,a}(E)$ which can be computed as the $k$-th elementary symmetric polynomial (\ref{fracsymkeq}) of the formal fractional Chern roots $x_i^{l,a}$ of the bundle gerbe module $E$. 
\subsection{Fractional classifying spaces}
\label{sec: fracclassifyingsp}
Denote by $X_\mathbb{Q}$ the rationalization of a nilpotent space $X$. Recall $BU(1)_{\mathbb{Q}} \simeq K(\mathbb{Q}, 2)$ the Eilenberg–MacLane space with second homotopy group isomorphic to $\mathbb{Q}$.
Motivated by (\ref{pichernrooteq}) we can construct a map
\[\label{itorusphieq}
\widetilde{\chi}: BU(1)\times BU(1)\stackrel{}{\longrightarrow} BU(1)_{\mathbb{Q}} 
\]
represented by the cohomology class $x-\frac{1}{l}g\in H^2(BU(1)\times BU(1);\mathbb{Q})$, where the two canonical generators $g$ and $x$ correspond to the two factors of $BU(1)\times BU(1)$ respectively. Then there is a composition map $\chi$ defined by
\begin{equation}\label{torusphieq}
\begin{split}
\chi: ~~~& BU(1)\times (\underbrace{BU(1)\times \cdots \times BU(1)}_{n}) \\
&\stackrel{\Delta^n\times {\rm id}}{\longrightarrow} 
(\underbrace{BU(1)\times \cdots \times BU(1)}_{n})\times (\underbrace{BU(1)\times \cdots \times BU(1)}_{n})\\
&\stackrel{\mathop{\prod}\limits_{i=1}^{n}\chi_{i,n+i}}{\longrightarrow}
\underbrace{BU(1)_{\mathbb{Q}} \times \cdots \times BU(1)_{\mathbb{Q}} }_{n},
\end{split}
\end{equation}
where $\Delta^n$ is the $n$-fold diagonal map, each $\chi_{i,n+i}$ is a copy of the map $\widetilde{\chi}$ from the $i$-th and $(n+i)$-th factors of the domain to the $i$-th factor of the codomain. Following the notation of the formal (fractional) Chern roots in Section \ref{sec: fchern}, we denote by $g, x_1, \ldots, x_n$ the generators of $H^2(BU(1)\times (\underbrace{BU(1)\times \cdots \times BU(1)}_{n}))$ for the domain of $\chi$ corresponding to each factor. Similarly denote by $\widetilde{x}_1,\ldots, \widetilde{x}_n$ the generators of $H^2(\underbrace{BU(1)_{\mathbb{Q}} \times \cdots \times BU(1)_{\mathbb{Q}} }_{n})$ for the codomain. Hence, by the construction of $\widetilde{\chi}$ we can choose $\widetilde{x}_i$ such that
\begin{equation}\label{torusphiclasseq}
\chi^\ast(\widetilde{x}_i)=x_i-\frac{1}{l}g.
\end{equation}

We would like to extend the map $\chi$ in (\ref{torusphieq}) of maximal tori to a map of Lie groups $\phi=\phi_{l|n}: BU(1)\times BU(n)\stackrel{}{\longrightarrow} BU(n)_{\mathbb{Q}}$. For this purpose we may apply Sullivan's rational homotopy theory, for the details of which one can refer to the standard reference \cite{FHT01}. For a map $h: X\stackrel{}{\longrightarrow} Z$ between nilpotent spaces. Denote by $\widehat{h}: (\Lambda V_Z, d)\stackrel{}{\longrightarrow} (\Lambda V_X, d)$ a {\it Sullivan representative} of $h$, where $(\Lambda V_X, d)$ and $(\Lambda V_Z, d)$ are {\it Sullivan models} of $X$ and $Z$ respectively.

The domain $BU(1)\times BU(n)$ of $\phi$ has Sullivan model $(\Lambda (g, c_1,\ldots, c_n), 0)$, where each $c_k$ is a representative of the {\it the universal Chern class} $c_k\in H^{2k}(BU(n))$. The codomain $BU(n)_{\mathbb{Q}}$ of $\phi$ has a Sullivan model $(\Lambda (\widetilde{c}_1,\ldots, \widetilde{c}_n), 0)$, where each $\widetilde{c}_k$ is a representative of the {\it the universal rational Chern class} $c_k^{\q}\in H^{2k}(BU(n)_{\mathbb{Q}})$.
Consider the diagram
\begin{gather}
\begin{aligned}
\xymatrix{
(\Lambda (g, x_1,\ldots, x_n), 0) &
(\Lambda (g, c_1,\ldots, c_n), 0) \ar[l]_{\widehat{{\rm id}\times BI}}\\
(\Lambda (\widetilde{x}_1,\ldots, \widetilde{x}_n), 0) \ar[u]_{\widehat{\chi}}&
(\Lambda (\widetilde{c}_1,\ldots, \widetilde{c}_n), 0) \ar[l]_{\widehat{BI}} \ar@{.>}[u]_{\widehat{\phi}},
}
\end{aligned}
\label{phisullivandiag}
\end{gather}
where $\widehat{()}$ is to take a Sullivan representative, 
\[
I: \underbrace{U(1)\times \cdots \times U(1)}_{n})\stackrel{}{\longrightarrow} U(n)
\]
is the inclusion of a maximal torus, $B$ is the classifying functor, and $\widehat{\phi}$ will be defined momentarily. Indeed, since $c_k$ and $\widetilde{c}_k$ are the $k$-th elementary symmetric polynomials of 
the roots $x_i$ and $\widetilde{x}_i$ respectively, we have by (\ref{torusphiclasseq})
\[
\begin{split}
\widehat{\chi}\circ\widehat{BI}(\widetilde{c}_k)
&=\widehat{\chi}(\sigma_k(\widetilde{x}_1,\ldots, \widetilde{x}_n))\\
&=\sigma_k(x_1-\frac{g}{l},\ldots, x_n-\frac{g}{l})\\
&\subseteq \Lambda (g, c_1,\ldots, c_n).
\end{split}
\]
Hence, there is a unique morphism $\widehat{\phi}$ such that Diagram (\ref{phisullivandiag}) commutes.
Indeed by the same computation in the proof of Lemma \ref{pichernlemma},
\[
\widehat{\phi}(\widetilde{c}_k)=\sum\limits_{i=0}^{k} \big(\frac{-1}{l}\big)^i \binom{n-k+i}{i} g^i c_{k-i}.
\]
Then we can define the composition map $\phi$ by
\begin{equation}\label{phidefeq}
\phi=\phi_{l|n}: BU(1)\times BU(n)\stackrel{r}{\longrightarrow}(BU(1)\times BU(n))_\mathbb{Q}\stackrel{\phi^\prime}{\longrightarrow} BU(n)_\mathbb{Q},
\end{equation}
where $r$ is the rationalization, and $\phi^\prime$ is the geometric realization of $\widehat{\phi}$. It is clear that the following lemma holds.
\begin{lemma}\label{uniphiclasslemma}
For the universal Chern classes,
\begin{equation}\label{uniphiclasseq}
\phi^\ast(c_k^{\q})=\sum\limits_{i=0}^{k} \big(\frac{-1}{l}\big)^i \binom{n-k+i}{i} g^i c_{k-i}.
\end{equation}
\end{lemma}

The map $\phi$ serves as a universal object for the bundle gerbe module $E$ of $L$. To be more precise, let $f: Y\rightarrow BU(n)$ be the classifying map of the complex vector bundle $E$. We have the map
\[
f^a:=a\times f: Y\stackrel{}{\longrightarrow} BU(1)\times BU(n)
\]
classifying $L_a\oplus E$, where $a$ is represented by $a\in H^2(Y;\mathbb{Z})$.
We need to construct a nice classifying map $f^l: M\rightarrow BU(n)_{\mathbb{Q}}$. Let $\widehat{\pi}: (\Lambda V_M, d)\stackrel{}{\longrightarrow} (\Lambda (V_M\oplus W), d)$ be a {\it relative Sullivan model} of $\pi$. 
Let $\widehat{f^a}: \Lambda (g, c_1,\ldots, c_n), 0)\stackrel{}{\longrightarrow}  (\Lambda (V_M\oplus W), d)$ be a Sullivan representative of $f^a$. Since by (\ref{uniphiclasseq}) and Lemma \ref{pichernlemma} the class of $\widehat{f^a}\circ \widehat{\phi}(\widetilde{c}_k)$ is $f^{a\ast}\circ \phi^\ast(c_k^{\q})=\pi^\ast(c_k^{l,a}(E))$,
we have 
\[
\widehat{f^a}\circ \widehat{\phi}(\widetilde{c}_k)=\widehat{\pi}(c_k^{l,a}(E))+d z_k
\]
for some $z_k\in \Lambda (V_M\oplus W)$, where by abuse of notation we use $c_k^{l,a}(E)$ to denote a representative of the fractional Chern class. Define
\[
\widehat{f^l}: (\Lambda (\widetilde{c}_1,\ldots, \widetilde{c}_n), 0) \stackrel{}{\longrightarrow} (\Lambda V_M, d)
\]
by $\widehat{f^l}(\widetilde{c}_k)=c_k^{l,a}(E)$. We need to show that the diagram
\begin{gather}
\begin{aligned}
\xymatrix{
 (\Lambda (V_M\oplus W), d) &
(\Lambda (g, c_1,\ldots, c_n), 0) \ar[l]_{\widehat{f^a}}\\
(\Lambda V_M, d) \ar[u]_{\widehat{\pi}}&
(\Lambda (\widetilde{c}_1,\ldots, \widetilde{c}_n), 0) \ar[l]_<<<<<<<{\widehat{f^l}} \ar[u]_{\widehat{\phi}},
}
\end{aligned}
\label{flsullivandiag}
\end{gather}
is commutative up to homotopy. Indeed, recall $\Lambda (t, dt)$ with ${\rm deg}(t)=0$ is a model of standard $1$-simplex. Denote by $\epsilon_i: \Lambda (t, dt)\rightarrow \mathbb{Q}$ the evaluation morphism defined by $\epsilon_i(t)=i$. Define a homotopy
\[
H: (\Lambda (\widetilde{c}_1,\ldots, \widetilde{c}_n), 0)\stackrel{}{\longrightarrow} (\Lambda (V_M\oplus W), d)\otimes \Lambda (t, dt)
\]
by $H(\widetilde{c}_k)=\widehat{\pi}(c_k^{l,a}(E))+d(tz_k)$. $H$ is well defined since $c_k^{l,a}(E)$ is a cycle, and $({\rm id}\otimes \epsilon_i(0))\circ H=\widehat{\pi}\circ \widehat{f^l}$ and $({\rm id}\otimes \epsilon_i(1))\circ H=\widehat{f^a}\circ \widehat{\phi}$. This shows that Diagram (\ref{flsullivandiag}) commutes up to homotopy, and we can take the geometric realization of Diagram (\ref{flsullivandiag}) to obtain a composition map
\[
f^l: M\stackrel{r}{\longrightarrow}M_\mathbb{Q}\stackrel{f^{l\prime}}{\longrightarrow} BU(n)_\mathbb{Q}
\]
with $r$ the rationalization and $f^{l\prime}$ a geometric realization of $\widehat{f^l}$. Moreover, there is the homotopy commutative diagram
\begin{gather}
\begin{aligned}
\xymatrix{
Y \ar[d]^{\pi} \ar[r]^<<<<<{f^a} & 
BU(1)\times BU(n) \ar[d]^{\phi=\phi_{l|n}} \\
M\ar[r]^<<<<<<<<{f^l} &
BU(n)_\mathbb{Q}
}
\end{aligned}
\label{fldefdiag}
\end{gather}
as the geometric realization of Diagram (\ref{flsullivandiag}). From the definition of $\widehat{f^l}$, it is clear that
\[
c_k^{l,a}(E)=f^{l\ast}(c_k^{\q}).
\]
In particular, we have showed Theorem \ref{relclassthmintro}.
The result illustrates that $\phi$ is a ``classifying space'' of the bundle gerbe module $E$ on the level of characteristic classes. Hence, we may call the universal map $\phi: BU(1)\times BU(n) \stackrel{}{\longrightarrow}BU(n)_\mathbb{Q}$ the {\it fractional classifying space} for rank $n$ bundle gerbe modules of any flat bundle gerbe of order $l$, and the pair of the maps $(f^a, f^l)$ the {\it classifying map} of the bundle gerbe module $E$. 

\begin{remark}\label{fracclassifyurmk}
\begin{itemize}
\item
We should emphasize that the defined fractional classifying space does not classify the bundle gerbes modules themselves. It is even unclear how to construct a bundle gerbe module from Diagram (\ref{fldefdiag}) for a given bundle gerbe $L$. 
\item As we explained, the terminologies, fractional classifying space and fractional classifying map, should be understood on the level of characteristic classes for bundle gerbe modules. In other words, it provides a uniform way to understand the fractional Chern classes of bundle gerbe modules from the perspective of classifying spaces, and hence serves as a universal object for the fractional Chern classes. 
\item This point of view particularly allows us to construct and study higher fractional complex structures in the sequel, which are parallel to the higher spin or spin$^c$ structures of real vector bundles and the higher complex structure of complex vector bundles (for instance, see \cite{DHH}, or Subsection \ref{sec: backintro}).
\end{itemize}
\end{remark}

\subsection{Fractional U-structure}
\label{sec: fracu}
Following the idea in Remark \ref{fracclassifyurmk}, Diagram (\ref{fldefdiag}) itself suggests a notion of fractional structure without the background of bundle gerbes and their modules. It serves as the starting point for the higher fractional structures introduced in the sequel.
\begin{definition}\label{fracundef}
Let $\pi: Y\rightarrow M$ be a map. Let $L_a$ be a complex line bundle determined by $a\in H^2(Y;\mathbb{Z})$ and $E$ be a complex vector bundle of rank $n$ classified by $f: Y\rightarrow BU(n)$. Let $l$ be a positive integer such that $l|n$. If we have a homotopy commutative diagram (\ref{fldefdiag}) with $f^a=a\times f$ and $\phi=\phi_{l|n}$ defined in (\ref{phidefeq}), we call $(E,\pi)$ or simply $E$ an {\it $(a, \frac{1}{l})$-fractional $U(n)$-bundle}, or simply a {\it fractional U-bundle}.
\end{definition}
If we ideally interpret the map $f^l: M\rightarrow BU(n)_{\mathbb{Q}}$ as the classifying map of a {\it rational vector bundle} over $M$, then a fractional U-bundle is roughly a proper vector bundle over $Y$ with a descent rational vector bundle over $M$. We may call the pair of groups ($U(1)\times U(n)$, $U(n)_{\mathbb{Q}}$) the {\it relative structure group} of the fractional bundle $E$. In terms of principal bundles, there exists the morphism of fibrations induced from Diagram (\ref{fldefdiag})
\begin{gather}
\begin{aligned}
\xymatrix{
U(1)\times U(n) \ar[r]^<<<<<{j}  \ar[d]^{\Omega \phi} &
P_U(f^a) \ar[r]^<<<<<{p^a} \ar[d]^{\Phi}&
Y\ar[d]^{\pi}\\
U(n)_{\mathbb{Q}} \ar[r]^{j} &
P_{U_{\mathbb{Q}}}(f^l) \ar[r]^<<<<<{p^l}&
M,
}
\end{aligned}
\label{fracuprindiag}
\end{gather}
where the top row is the principal bundle of $L_a\oplus E$ classified by $f^a$, and the bottom row is the principal fibration induced from $f^l$. As before the latter fibration can be ideally viewed as the {\it rational principal bundle} of the rational vector bundle classified by $f^l$, and Diagram (\ref{fracuprindiag})
can be viewed as the {\it fractional principal bundle} of the fractional U-bundle $E$.

When $E$ admits a fractional U-structure in Definition \ref{fracundef}, we have the {\it $k$-th fractional Chern classes} $c_k^{l,a}(E)$ defined by 
\[
c_k^{l,a}(E):=f^{l\ast}(c_k^{\q}).
\]
By Diagram (\ref{fldefdiag}) and (\ref{uniphiclasseq})
\[\label{fracunclasspieq}
\pi^\ast(c_k^{l,a}(E))=\sum\limits_{i=0}^{k} \big(\frac{-1}{l}\big)^i \binom{n-k+i}{i} a^i c_{k-i}(E).
\]
When the fractional bundle $E$ is derived from a bundle gerbe module structure, it is clear that the fractional Chern class coincides with the twisted Chern class. This justifies our choice of notation.


\section{Higher fractional structures}
\label{sec: strongfrac}
In this section, we introduce fractional SU-structure and fractional U$\langle6\rangle$-structure as higher fractional U-structures analogous to the spin and string structures as higher orientations. 
We work with general fractional U-bundles and the constructions and arguments here can be applied to bundle gerbe modules automatically.
In particular, we characterize the two structures as lifting problems on constructed fractional classifying spaces, and prove Theorem \ref{fracsunthm} and Theorem \ref{fracu6nthm}.

\subsection{Fractional SU-structure} \label{sec: fracsu}
Let $\pi: Y\rightarrow M$ be a map. Let $L_a$ be a complex line bundle determined by $a\in H^2(Y;\mathbb{Z})$ and $E$ be a complex vector bundle of rank $n$ classified by $f: Y\rightarrow BU(n)$. 
Suppose $E$ admits an $(a,\frac{1}{l})$-fractional $U(n)$-structure determined by Diagram (\ref{fldefdiag}). Then $E$ has the fractional Chern class $c_k^{l,a}(E)=f^{l\ast}(c_k^{\q})\in H^{2k}(M;\mathbb{Q})$ for each $1\leq k\leq n$. 
\begin{definition}\label{fracsundef}
Let $E$ be an $(a,\frac{1}{l})$-fractional $U(n)$-bundle as above. $E$ has an {\it $(a, \frac{1}{l})$-fractional $SU(n)$-structure}, or simply a {\it fractional SU-structure} if $c_1^{l,a}(E)=0$.
\end{definition}

Let us study the universal case. Recall that $n=ls$. Consider the finite covering 
\[
\mathbb{Z}/s\mathbb{Z}\stackrel{}{\longrightarrow} U(n)_l \stackrel{\rho_s}{\longrightarrow} U(n)
\]
of $U(n)$ corresponding to the subgroup $s\mathbb{Z}\subseteq \mathbb{Z}\cong \pi_1(U(n))$, which defines the Lie group $U(n)_l$ and the covering map $\rho_s$. Since $U(n)\cong U(1)\times_{\mathbb{Z}/n\mathbb{Z}} SU(n)$, it is clear that
\begin{equation}\label{Unlisoeq}
 U(n)_l\cong U(1)\times_{\mathbb{Z}/l\mathbb{Z}} SU(n).
\end{equation}
In particular, $U(n)_n=U(n)$ and $U(n)_1=U(1)\times SU(n)$ as Lie groups.
Define a Lie group homomorphism
\[
d_l: U(n)_l\stackrel{}{\longrightarrow} U(1)
\]
by $d_l([z, A])=z^l$. Then under the isomorphism (\ref{Unlisoeq}), $d_n$ is the canonical determinant homomorphism ${\rm det}: U(n)\longrightarrow U(1)$, and $\rho_1: U(1)\times SU(n)\longrightarrow U(n)$ is the Lie group homomorphism defined by $\rho_1(z, A)=z A$. Hence there is the morphism of Lie group extensions
\begin{gather}
\begin{aligned}
\xymatrix{
\{1\} \ar[r] &
 SU(n)\ar@{=}[d] \ar[r]^{i_{2}^\prime} &
 U(n)_l \ar[d]^{\rho_s} \ar[r]^{d_l} &
 U(1) \ar[d]^{\tau_s} \ar[r] &
 \{1\} \\
\{1\} \ar[r]&
SU(n)\ar[r]^{i_2} &
U(n)  \ar[r]^{{\rm det}} &
 U(1)  \ar[r] &
  \{1\} ,
}
\end{aligned}
\label{unextdiag}
\end{gather}
where $i_2$ and $i_{2}^\prime$ are the standard inclusions, and $\tau_s(z)=z^s$ for any $z\in U(1)$. In particular, the right square of Diagram (\ref{unextdiag}) is a pullback of Lie groups, which determines the Lie group extension
\begin{equation}\label{u1unexteq}
\{1\}\stackrel{}{\longrightarrow}U(n)_l \stackrel{(d_l, \rho_s)}{\longrightarrow} U(1)\times U(n) \stackrel{\mu_s}{\longrightarrow} U(1)\stackrel{}{\longrightarrow}\{1\},
\end{equation}
where $\mu_s(z, A)=(\tau_{-s}\times {\rm det})(z, A)=z^{-s}\cdot {\rm det}(A)$ for any $(z, A)\in U(1)\times U(n)$. 

From now on let us suppose $l>1$. Recall that we have the class $g\in H^2(BU(1))$ such that $(B{\rm det})^\ast(g)=c_1$.
Applying the classifying functor $B$ to Diagram (\ref{unextdiag}), it is easy to compute that as graded rings
\begin{equation}\label{hbunleq}
H^\ast(BU(n)_l;\mathbb{Z})\cong \mathbb{Z}[\overline{c}_1, c_2,\ldots, c_n],
\end{equation}
such that $(Bi_{2}^\prime)^\ast(c_i)\in H^{2i}(BSU(n);\mathbb{Z})$ is the $i$-th universal Chern class for each $i\geq 2$, and
\begin{equation}\label{fracsuchernreleq}
(Bd_l)^\ast(g)=\overline{c}_1, \ \ ~ (B\rho_s)^\ast(c_1)=s\overline{c}_1, \ \ ~ (B\rho_s)^\ast (c_i)=c_i,~{\rm for}~{\rm each}~i\geq 2.
\end{equation}
Hence, when a rank $n$ complex bundle $\zeta$ satisfies that $s~|~c_1(\zeta)$, its structure group $U(n)$ can be lifted to $U(n)_l$.

It is clear that $B\mu_s$ represents the class $c_1-sg\in H^2(BU(1)\times BU(n))$. 
Since $\phi^\ast(c^{\q}_1)=c_1-sg$ by (\ref{uniphiclasseq}), we can construct the following homotopy commutative diagram of fibrations
\begin{gather}
\begin{aligned}
\xymatrix{
BU(n)_l \ar[r]^<<<<<<{B(d_l, \rho_s)} \ar[d]^{\phi_2}&
BU(1)\times BU(n) \ar[r]^<<<<{B\mu_s} \ar[d]^{\phi}&
K(\mathbb{Z},2) \ar[d]^{r}\\
BSU(n)_{\mathbb{Q}} \ar[r]^{Bi_{2\mathbb{Q}}} &
BU(n)_{\mathbb{Q}}  \ar[r]^{c^{\q}_1}&
K(\mathbb{Q},2),
}
\end{aligned}
\label{unifracsudiag}
\end{gather}
where the top row is obtained from (\ref{u1unexteq}) by taking the classifying functor $B$, $r$ is the rationalization, and $\phi_2$ is induced from $\phi$. 
Indeed, let $\phi_{\mathbb{Q}}: BU(1)_{\mathbb{Q}}\times BU(n)_{\mathbb{Q}}\stackrel{}{\rightarrow} BU(n)_{\mathbb{Q}}$ be the rationalization of $\phi$.
The homotopy pullback of $\phi_{\mathbb{Q}}$ along $Bi_{2\mathbb{Q}}$ gives a map $\widetilde{\phi}_2: BU(n)_{l\mathbb{Q}}\stackrel{}{\rightarrow} BSU(n)_{\mathbb{Q}}$. Then $\phi_2$ can be chosen to be the composition
\begin{equation}\label{phi2def}
\phi_2: BU(n)_l\stackrel{r}{\longrightarrow}BU(n)_{l\mathbb{Q}} \stackrel{\widetilde{\phi}_2}{\longrightarrow} BSU(n)_{\mathbb{Q}}, 
\end{equation}
where $r$ is the rationalization.

Now let us turn to the fractional bundle $E$ determined by Diagram (\ref{fldefdiag}).
We summarize the formulae for the involved obstructions in the following lemma.
\begin{lemma}\label{obfracsulemma}
\[
f^{a\ast}(c_1-sg)=c_1(E)-sa, \ \ ~  f^{l\ast}(c^{\q}_1)=c_1^{l,a}(E). 
\]
\end{lemma}
\begin{proof}
$f^{l\ast}(c^{\q}_1)=c_1^{l,a}(E)$ is obtained by Theorem \ref{E>=fracunthm}. Then with Diagram (\ref{fldefdiag}) $f^{a\ast}(c_1-sg)=c_1(E)-sa$ follows from (\ref{uniphiclasseq}) and (\ref{12fraccherneq}).
\end{proof}

We are ready to prove Theorem \ref{fracsunthm}.
\begin{proof}[Proof of Theorem \ref{fracsunthm}]
The bottom square in Diagram (\ref{fracsunliftdiag}) is Diagram (\ref{fldefdiag}), while with the convention $i_{2l}=(d_l,\rho_s)$ the top right square in Diagram (\ref{fracsunliftdiag}) is the left square in Diagram (\ref{unifracsudiag}).

If $E$ is a fractional SU-bundle, then $c^{l,a}_1(E)=0$, and hence $\pi^\ast(c^{l,a}_1(E))=c_1(E)-sa=0$ by Lemma \ref{obfracsulemma}. With Diagram (\ref{unifracsudiag}), this is equivalent to that both the compositions $c^{\q}_1\circ f^l$ and $B\mu_s\circ f^a$ are null homotopic. It follows that there exist $f^a_2: Y\rightarrow BU(n)$ and $f^l_2: M\rightarrow BSU(n)_{\mathbb{Q}}$ such that $B(d_l, \rho_s)\circ f^a_2\simeq f^a$ and $Bi_{2\mathbb{Q}}\circ f^l_2\simeq f^l$, in other words, the front and back triangles in Diagram (\ref{fracsunliftdiag}) commute up to homotopy. 

In particular,
\begin{equation}\label{supfeq1}
Bi_{2\mathbb{Q}}\circ\phi_2\circ f^a_2\simeq 
\phi\circ B(d_l, \rho_s)\circ f^a_2\simeq \phi\circ f^a\simeq f^l\circ \pi\simeq Bi_{2\mathbb{Q}}\circ f_2^l\circ \pi.
\end{equation}
On the other hand, by Sullivan's rational homotopy theory it is easy to see that $BU(n)_{\mathbb{Q}}\simeq BSU(n)_{\mathbb{Q}} \times BS^1_{\mathbb{Q}}$ such that $Bi_{2\mathbb{Q}}$ admits a left homotopy inverse $q$. It follows from (\ref{supfeq1}) that
\[ 
\phi_2\circ f^a_2\simeq q\circ Bi_{2\mathbb{Q}}\circ\phi_2\circ f^a_2\simeq q\circ Bi_{2\mathbb{Q}}\circ f_2^l\circ \pi\simeq f_2^l\circ \pi,
\]
that is, the left top square commutes up to homotopy. Hence we have showed that Diagram (\ref{fracsunliftdiag}) exists and commutes up to homotopy if $E$ is a fractional SU-bundle. The proof of the converse statement is similar and omitted. Moreover, (\ref{fracsunchernforeq}) follows immediately from (\ref{fracsuchernreleq}) and Theorem \ref{E>=fracunthm}.

We are left to count the number of the fractional SU-structures. Suppose $E$ is a fractional SU-bundle. From Diagram (\ref{fracuprindiag}) and Diagram (\ref{fldefdiag}), by the dual Blakers-Massey theorem \cite[Theorem C.3]{DHH} there is the commutative diagram of exact sequences
\begin{gather}
\begin{aligned}
\xymatrix{
0\ar[r]  &
H^1(Y) \ar[r]^<<<<<{p^{a\ast}} &
H^1(P_U(f^a)) \ar[r]^<<<<<{\delta\circ j^\ast}&
H^2(BU(1)\times BU(n)) \ar[r]^<<<<<{f^{a\ast}} &
H^2(Y) \\
0\ar[r]  &
H^1(M) \ar[r]^<<<<<{p^{l\ast}}  \ar[u]_{\pi^\ast}&
H^1(P_{U_{\mathbb{Q}}}(f^l)) \ar[r]^<<<<<<<<{\delta\circ j^\ast}   \ar[u]_{\Phi^\ast} &
H^2(BU(n)_{\mathbb{Q}}) \ar[r]^<<<<<<<<{f^{l\ast}} \ar[u]_{\phi^\ast}&
H^2(M)\ar[u]_{\pi^\ast},
}
\end{aligned}
\label{H12fracuprindiag}
\end{gather}
where the transgressions $\delta: H^{1}(U(1)\times U(n))\stackrel{}{\rightarrow}H^2(BU(1)\times BU(n))$ and $\delta: H^1(U(n)_{\mathbb{Q}})\stackrel{}{\rightarrow}H^2(BU(n)_{\mathbb{Q}})$ are isomorphisms. Since $f^{l\ast}(c^{\q}_1)=c_1^{l,a}(E)=0$ and $f^{a\ast}(c_1-sg)=c_1(E)-sa=0$, a fractional SU-structure is determined by the choice of the pair $(P, \Phi^\ast(P))\in H^1(P_{U_{\mathbb{Q}}}(f^l))\times H^1(P_U(f^a))$ such that $\delta\circ j^\ast(P)=c^{\q}_1$. From Diagram (\ref{H12fracuprindiag}), there are exactly $H^1(M)$ many of such choices. Hence the fractional SU-structures on $E$ are in one-to-one correspondence with the elements of $H^1(M)$. This completes the proof of the theorem.
\end{proof}
By Theorem \ref{fracsunthm} $\phi_2$ is a ``classifying space'' of the fractional SU-bundle $E$ on the level of characteristic classes. Hence, we may call the universal map $\phi_2: BU(n)_l \stackrel{}{\longrightarrow}BSU(n)_\mathbb{Q}$ the {\it fractional classifying space} of the fractional SU-structure, and the fractional SU-bundle $E$ has the {\it classifying map} $(f^a_2, f^l_2)$.

Geometrically, in Diagram (\ref{fracsunliftdiag}) the structure group $U(n)_{\mathbb{Q}}$ of the rational vector bundle determined by $f^l$ is lifted to $SU(n)_{\mathbb{Q}}$ through $Bi_{2\mathbb{Q}}$, while the structure group $U(1)\times U(n)$ of the vector bundle $L_a\oplus E$ is lifted to $U(n)$ through $B(d_l, \rho_s)$. Hence a fractional U-bundle $E$ admits a fractional SU-structure if and only if its relative structural group ($U(1)\times U(n)$, $U(n)_{\mathbb{Q}}$) can be lifted to ($U(n)$, $SU(n)_{\mathbb{Q}}$).
In terms of principal bundles, this means that the fractional principal bundle (\ref{fracuprindiag}) of the fractional U-bundle $E$ can be lifted to the fractional principal bundle 
\begin{gather}
\begin{aligned}
\xymatrix{
U(n)_l \ar[r]^<<<<<<{j_{2}}  \ar[d]^{\Omega \phi_2} &
P_{U_l}(f_2^a) \ar[r]^<<<<<{p_2^a} \ar[d]^{\Phi_2}&
Y\ar[d]^{\pi}\\
SU(n)_{\mathbb{Q}} \ar[r]^<<<<<{j_2} &
P_{SU_{\mathbb{Q}}}(f_2^l) \ar[r]^<<<<<{p_2^l}&
M,
}
\end{aligned}
\label{fracsuprindiag}
\end{gather}
where the top row is the principal bundle classified by $f_2^a$, and the bottom row is the principal fibration induced from $f_2^l$.

\subsection{Fractional U$\langle 6 \rangle$-structure}
\label{sec: fracu6}
Let $\pi: Y\rightarrow M$ be a map. Let $L_a$ be a complex line bundle determined by $a\in H^2(Y;\mathbb{Z})$ and $E$ be a complex vector bundle of rank $n$ classified by $f: Y\rightarrow BU(n)$. 
Suppose $l>1$ and $E$ admits an $(a,\frac{1}{l})$-fractional $SU(n)$-structure as described in Definition \ref{fracsundef}.
\begin{definition}\label{fracu6ndef}
Let $E$ be an $(a,\frac{1}{l})$-fractional $SU(n)$-bundle. $E$ has an {\it $(a, \frac{1}{l})$-fractional $U\langle 6 \rangle(n)$-structure}, or simply a {\it fractional U$\langle 6 \rangle$-structure} if $c_2^{l,a}(E)=0$.
\end{definition}

Let us study the universal case. Consider the fractional classifying space $\phi_2: BU(n)_l\stackrel{}{\rightarrow} BSU(n)_\mathbb{Q}$ of fractional SU-structure.
\begin{lemma}\label{phi2chernlemma}
For each $k\geq 2$,
\[
\phi_2^\ast(c_k^{\q})=\sum\limits_{i=0}^{k-2} \big(\frac{-1}{l}\big)^i \binom{n-k+i}{i} \overline{c}_1^i c_{k-i}+  \big(\frac{-1}{l}\big)^k (1-k)\binom{n}{k} \overline{c}_1^k.
\]
\end{lemma}
\begin{proof}
Recall $n=ls$. By Diagram (\ref{unifracsudiag}), (\ref{uniphiclasseq}) and (\ref{fracsuchernreleq}), we have 
\[
\begin{split}
\phi_2^\ast(c_k^{\q})
&=\phi_2^\ast\circ Bi_{2\mathbb{Q}}^\ast(c_k^{\q})
=B(d_l, \rho_s)^\ast\circ \phi^\ast(c_k^{\q})\\
&=B(d_l, \rho_s)^\ast\Big(\sum\limits_{i=0}^{k} \big(\frac{-1}{l}\big)^i \binom{n-k+i}{i} g^i c_{k-i}\Big)\\
&=\sum\limits_{i=0}^{k-2} \big(\frac{-1}{l}\big)^i \binom{n-k+i}{i} \overline{c}_1^i c_{k-i}+\big(\frac{-1}{l}\big)^{k-1} \binom{n-1}{k-1} s\overline{c}_1^{k}+\big(\frac{-1}{l}\big)^k \binom{n}{k} \overline{c}_1^k\\
&=\sum\limits_{i=0}^{k-2} \big(\frac{-1}{l}\big)^i \binom{n-k+i}{i} \overline{c}_1^i c_{k-i}+  \big(\frac{-1}{l}\big)^k (1-k)\binom{n}{k} \overline{c}_1^k.
\end{split}
\]
\end{proof}
In particular, for $k=2$,
\begin{equation}\label{phi2c2eq}
\phi_2^\ast(c^{\q}_2)=c_2-\frac{s(n-1)}{2l}\overline{c}_1^2.
\end{equation}
We may define a topological group $U\langle 6 \rangle(n)_l$ by the pullback
\begin{gather}
\begin{aligned}
\xymatrix{
U\langle 6 \rangle(n)_l \ar[d]^{\Omega \phi_3}  \ar[r]^{i_{3l}}&
U(n)_l \ar[d]^{\Omega \phi_2} \\
U\langle 6 \rangle(n)_{\mathbb{Q}}  \ar[r]^{i_{3\mathbb{Q}}}&
SU(n)_{\mathbb{Q}},
}
\end{aligned}
\label{u6nldiag}
\end{gather}
where $i_{3\mathbb{Q}}$ is the rationalization of the group extension
\begin{equation}\label{u6nexteq}
\{1\}\stackrel{}{\longrightarrow} K(\mathbb{Z},2) \stackrel{}{\longrightarrow} U\langle 6 \rangle (n)\stackrel{i_3}{\longrightarrow}  SU(n) \stackrel{}{\longrightarrow} \{1\}
\end{equation}
with suitable group structure on $K(\mathbb{Z},2)$ (see Remark \ref{u6stringrmk} for more detials), and $\phi_3=B\Omega\phi_3: BU\langle 6 \rangle(n)_l \stackrel{}{\rightarrow} BU\langle 6 \rangle(n)_{\mathbb{Q}}$ is the induced map between classifying spaces. Hence there is the induced group extension
\begin{equation}\label{u6nlexteq}
\{1\}\stackrel{}{\longrightarrow} K(\mathbb{Q},2) \stackrel{}{\longrightarrow} U\langle 6 \rangle (n)_l\stackrel{i_{3l}}{\longrightarrow}  U(n)_l \stackrel{}{\longrightarrow} \{1\}.
\end{equation}
In particular, let us denote $\bar{c}_1=Bi_{3l}^\ast(\overline{c}_1)$ and $c_i=Bi_{3l}^\ast(c_i)$ for each $1\leq i\leq n$. From (\ref{u6nldiag}) and (\ref{phi2c2eq}) we have that 
\begin{equation}\label{barc2eq}
c_2=\frac{s(n-1)}{2l}\bar{c}_1^2.
\end{equation}

Now let us turn to the fractional SU-bundle $E$ determined by Diagram (\ref{fracsunliftdiag}) in Theorem \ref{fracsunthm}. In particular, the SU-structure on $E$ is classified by the pair of maps $(f_2^a, f_2^l)$.
We summarize the formulae for the involved obstructions in the following lemma.
\begin{lemma}\label{obfracu6lemma}
\[
f_2^{a\ast}(c_2-\frac{s(n-1)}{2l}\overline{c}_1^2)=c_2(E)-\frac{s(n-1)}{2l}a^2, \ \ ~  f_2^{l\ast}(c^{\q}_2)=c_2^{l,a}(E). 
\]
Moreover, $c_1(E)=sa$.
\end{lemma}
\begin{proof}
First $c_1(E)=sa$ by (\ref{fracsunchernforeq}) since $E$ is fractional SU.
$f_2^{l\ast}(c^{\q}_2)=c_2^{l,a}(E)$ is obtained by Theorem \ref{E>=fracunthm}. Then with Diagram (\ref{fracsunliftdiag}) $f_2^{a\ast}(c_2-\frac{s(n-1)}{2l}\overline{c}_1^2)=c_2(E)-\frac{s(n-1)}{2l}a^2$ follows from (\ref{fracsunchernforeq}).
\end{proof}

We are ready to prove Theorem \ref{fracu6nthm}.

\begin{proof}[Proof of Theorem \ref{fracu6nthm}]
By Theorem \ref{fracsunthm} the bottom square in Diagram (\ref{fracu6nliftdiag}) commutes up to homotopy. The top right square in Diagram (\ref{fracu6nliftdiag}) is Diagram (\ref{u6nldiag}) after applying the classifying functor $B$. 

If $E$ is a fractional U$\langle 6\rangle$-bundle, then $c^{l,a}_1(E)=0$ and $c^{l,a}_2(E)=0$. It follows that there exists a map $f_3^l: M\rightarrow BU\langle 6\rangle (n)_{\mathbb{Q}}$ such that $Bi_{3\mathbb{Q}}\circ f_3^l\simeq f_2^l$, and we obtain the front triangle in Diagram (\ref{fracu6nliftdiag}). Then since the top right square in Diagram (\ref{fracu6nliftdiag}) is a homotopy pullback, there exists a unique $f_3^a: Y\rightarrow BU\langle6\rangle (n)_l$ up to homotopy such that $Bi_{3l}\circ f_3^a\simeq f_2^a$ and $\phi_3 \circ f_3^a\simeq f_3^l\circ\pi$. Thus the back triangle and the top left square in Diagram (\ref{fracu6nliftdiag}) commute up to homotopy. We have showed that Diagram (\ref{fracu6nliftdiag}) exists and commutes up to homotopy if $E$ is a fractional U$\langle 6\rangle$-bundle. The proof of the converse statement is similar and omitted. Moreover, (\ref{fracu6nchernforeq}) follows immediately from (\ref{fracsunchernforeq}), (\ref{barc2eq}) and Diagram (\ref{fracu6nliftdiag}).

We are left to count the number of the fractional U$\langle 6\rangle$-structures. Suppose $E$ is a fractional U$\langle 6\rangle$-bundle. From Diagram (\ref{fracsuprindiag}) and the bottom square of Diagram (\ref{fracu6nliftdiag}), by the dual Blakers-Massey theorem \cite[Theorem C.3]{DHH} there is the commutative diagram
\begin{gather}
\begin{aligned}
\xymatrix{
 &
H^3(Y) \ar[r]^<<<<<<{p_2^{a\ast}} &
H^3(P_{U_l}(f_2^a)) \ar[r]^<<<<<<{\delta\circ j_{2}^\ast}&
H^4(BU(n)_l) \ar[r]^<<<<<<{f_2^{a\ast}} &
H^4(Y) \\
0\ar[r]  &
H^3(M) \ar[r]^<<<<<{p_2^{l\ast}}  \ar[u]_{\pi^\ast}&
H^3(P_{SU_{\mathbb{Q}}}(f_2^l)) \ar[r]^<<<<<{\delta\circ j_2^\ast}   \ar[u]_{\Phi_2^\ast} &
H^4(BSU(n)_{\mathbb{Q}}) \ar[r]^<<<<<{f_2^{l\ast}} \ar[u]_{\phi_2^\ast}&
H^4(M)\ar[u]_{\pi^\ast},
}
\end{aligned}
\label{H34fracuprindiag}
\end{gather}
where the second row is exact with the isomorphic transgression $\delta: H^3(SU(n)_\mathbb{Q})\stackrel{}{\rightarrow} H^4(BSU(n)_\mathbb{Q})$, and in the top row the transgression $\delta: H^3(U(n))\stackrel{}{\rightarrow} H^4(BU(n))$ is a monomorphism.
Since $f_2^{l\ast}(c^{\q}_2)=c_2^{l,a}(E)=0$ and $f_2^{a\ast}\circ \phi_2^\ast(c^{\q}_2)=f_2^{a\ast}(c_2-\frac{s(n-1)}{2l}\overline{c}_1^2)=c_2(E)-\frac{s(n-1)}{2l}a^2=0$ by Lemma \ref{obfracu6lemma}, a fractional U$\langle 6\rangle$-structure is determined by the choice of the pair $(P, \Phi_2^\ast(P))\in H^3(P_{SU_{\mathbb{Q}}}(f_2^l))\times H^3(P_{U_l}(f_2^a))$ such that $\delta\circ j_2^\ast(P)=c^{\q}_2$. From Diagram (\ref{H34fracuprindiag}), there are exactly $H^3(M)$ many of such choices. Hence the fractional U$\langle 6\rangle$-structures on $E$ are in one-to-one correspondence with the elements of $H^3(M)$. This completes the proof of the theorem.
\end{proof}
By Theorem \ref{fracu6nthm} $\phi_3$ is a ``classifying space'' of the fractional U$\langle 6\rangle$-bundle $E$ on the level of characteristic classes. Hence, we may call the universal map $\phi_3: BU\langle 6\rangle(n)_l \stackrel{}{\longrightarrow}BU\langle 6\rangle(n)_\mathbb{Q}$ the {\it fractional classifying space} of the fractional U$\langle 6\rangle$-structure, and the fractional U$\langle 6\rangle$-bundle $E$ has the {\it classifying map} $(f^a_3, f^l_3)$.

Geometrically, in Diagram (\ref{fracu6nliftdiag}) the structure group $SU(n)_{\mathbb{Q}}$ of the rational vector bundle determined by $f_2^l$ is lifted to $U\langle 6\rangle(n)_\mathbb{Q}$ through $Bi_{3\mathbb{Q}}$, while the structure group $U(n)$ of the vector bundle $L_a\oplus E$ is lifted to $U\langle 6\rangle(n)_l$ through $Bi_{3l}$. Hence a fractional SU-bundle $E$ admits a fractional U$\langle 6\rangle$-structure if and only if its relative structural group ($U(n)$, $SU(n)_{\mathbb{Q}}$) can be lifted to ($U\langle 6\rangle(n)_l$, $U\langle 6\rangle(n)_\mathbb{Q}$).
In terms of principal bundles, this means that the fractional principal bundle (\ref{fracsuprindiag}) of the fractional U-bundle $E$ can be lifted to the fractional principal bundle 
\begin{gather}
\begin{aligned}
\xymatrix{
U\langle 6\rangle(n)_l \ar[r]^<<<<<<{j_3}  \ar[d]^{\Omega \phi_3} &
P_{U\langle 6\rangle_l}(f_3^a) \ar[r]^<<<<<{p_3^a} \ar[d]^{\Phi_3}&
Y\ar[d]^{\pi}\\
U\langle 6\rangle(n)_\mathbb{Q}  \ar[r]^<<<<<{j_3} &
P_{U\langle 6\rangle_\mathbb{Q} }(f_3^l) \ar[r]^<<<<<{p_3^l}&
M,
}
\end{aligned}
\label{fracu6prindiag}
\end{gather}
where the top row is the principal bundle classified by $f_3^a$, and the bottom row is the principal fibration induced from $f_3^l$.

\begin{remark}\label{u6stringrmk}
The group extension (\ref{u6nexteq}) can be compared with the extension
\begin{equation}\label{stringnexteq}
\{1\}\stackrel{}{\longrightarrow}K(\mathbb{Z},2) \stackrel{}{\longrightarrow}String(2n)\stackrel{\mathfrak{q}_3}{\longrightarrow} Spin(2n)\stackrel{}{\longrightarrow}\{1\},
\end{equation}
where $String(2n)$ is the string group. Indeed, recall there is the standard Lie group homomorphism $r: SU(n)\stackrel{}{\rightarrow}Spin(2n)$ such that $Br^\ast(p_1)=-2c_2$ where $p_1$ is the first universal Pontryagin class. Since $\frac{p_1}{2}\in H^4(BSpin(2n))$ is a generator, we can define $U\langle 6\rangle(n)$ by the pullback of groups
\begin{gather}
\begin{aligned}
\xymatrix{
U\langle 6\rangle(n) \ar[r]^{} \ar[d]^{i_3} &
String(2n) \ar[d]^{\mathfrak{q}_3} \\
SU(n) \ar[r]^<<<<<<{r} &
Spin(2n).
}
\end{aligned}
\label{suspindiag}
\end{gather}
In particular, any group model of $String(2n)$ induces a group structure of $U\langle 6\rangle(n)$ by Diagram (\ref{suspindiag}). Indeed, in the literature there are more elaborated descriptions of string group. For instance, Stolz and Teichner \cite{ST} modelled $String$ as a topological group in terms of group extension by a projective unitary group $PU(A)$ as a model of $K(\mathbb{Z}, 2)$. Moreover, Nikolaus, Sachse and Wockel \cite{NSW} constructed an infinite-dimensional Lie group model for $String$.
\end{remark}


\section{Fractional loop structures}
\label{sec: weakfrac}
Motivated by the way that loop spin structures, or weak string structures, can be studied in terms of spin structures on loop spaces, we define fractional loop structures of various levels, which in general are weaker than the corresponding notions introduced in Section \ref{sec: strongfrac}. 
Indeed, we will introduce the fractional loop U-structure and the fractional loop SU-structure as the fractional complex analogues of the loop orientation and loop spin. 
For this purpose, we need to study the involved free suspension to analyze the related loop characteristic classes.
We then introduce and study the two fractional loop structures for general fractional U-bundle after looping, and the constructions and arguments here can be applied to bundle gerbe modules automatically. In particular, we characterize the two structures as lifting problems on constructed fractional classifying spaces and prove Theorem \ref{fracloopunthm} and Theorem \ref{fracloopsunthm}. 
Finally, we compare these two loop structures with their non-loop counterparts in Section \ref{sec: strongfrac} from both the perspective of classifying space and the perspective of free suspension.

\subsection{The free suspension}
\label{sec: freesus}
To begin with, let us recall the notion of free transgression. Let $X$ be a pointed space. There is the canonical fibration
\begin{equation}\label{generalfreefib}
\Omega X\stackrel{i}{\rightarrow} LX\stackrel{p}{\rightarrow} X.
\end{equation}
We define the free evaluation map
\begin{equation}\label{freeevalu}
{\rm ev}: S^1\times LX \rightarrow X
\end{equation}
by ${\rm ev}((t, \lambda))=\lambda(1)$. 
The \textit{free suspension} (called {\it transgression} by geometers)
\begin{equation}\label{freesuspen}
\nu: H^{n+1}(X) \rightarrow H^{n} (LX)
\end{equation}
is then determined by the formula ${\rm ev}^\ast(x)= 1\otimes p^\ast(x)+ s\otimes \nu(x)$ for any $x\in H^{n+1}(X)$, where $s\in H^1(S^1)$ is a generator.
It is not hard to check that the free suspension is natural and satisfies the following properties (see \cite[Section 3]{Kuri} or \cite[Section $2$]{KK}):
\begin{itemize}  
\item[(1)] $i^\ast\circ \nu =\sigma^\ast: H^{n+1}(X) \rightarrow H^{n} (\Omega X)$;
\item[(2)] $\nu (xy)= \nu(x)p^\ast(y)+(-1)^{|x|}\nu(y)p^\ast(x)$, for any $x$ and $y\in H^{n+1}(X)$,
\end{itemize}
where $\sigma^\ast$ is the classical cohomology suspension (see \cite[Appendix B]{DHH}). The Property (2) means that $\nu$ is a module derivation under $p^\ast$. Since $p^\ast$ is always injective, we may omit it and simply write $\nu(x)y$ for $\nu(x)p^\ast(y)$, etc.

In order to compute the free suspension for $BU(n)$, let us recall the corresponding result for $BSpin^c(2n)$ studied in \cite[Section 2]{DHH}. First we have
\[
H^{\leq 4}(BSpin^c(2n))\cong \mathbb{Z}^{\leq 4}[t, q_1], \ \  ~ H^{\leq 3}(BLSpin^c(2n))\cong \mathbb{Z}^{\leq 3}[s,t,\mu]
\]
where ${\rm deg}(t)=2$ with $p^\ast(t)=t$ by abuse of notation, ${\rm deg}(q_1)=4$ is the first universal spin$^c$ class, and $\sigma^\ast(t)=i^\ast(s)$, $\sigma^\ast(q_1)=i^\ast(\mu)$. In particular, ${\rm deg}(s)=1$ and ${\rm deg}(\mu)=3$. Moreover, by \cite[Lemma 2.1]{DHH}
\begin{equation}\label{spinnueq}
\nu(t)=s, \ \ \nu(q_1)=\mu-st, \ \  \nu(t^2)=2st.
\end{equation}
Now let us turn to $BU(n)$. Recall there is the standard inclusion of Lie groups $r: U(n)\stackrel{}{\rightarrow} Spin^c(2n)$ such that 
\[
Br^\ast(t)=c_1, \ \ Br^\ast(p_1)=c_1^2-2c_2.
\]
Moreover, since by \cite{Duan} $2q_1+t^2=p_1\in H^\ast(BSpin^c(n))$ we see that $Br^\ast(q_1)=-c_2$.
For the free loop group, by an augment of Serre spectral sequence it is clear that
\begin{equation}\label{h3bluneq}
H^{\leq 4}(BLU(n))\cong \mathbb{Z}^{\leq 4}[z_1,c_1,z_2,c_2]
\end{equation}
where $p^\ast(c_i)=c_i$ with $i=1$ or $2$ by abuse of notation, and $\sigma^\ast(c_1)=i^\ast(z_1)$ and $\sigma^\ast(c_2)=i^\ast(z_2)$. In particular, ${\rm deg}(z_1)=1$ and ${\rm deg}(z_2)=3$. The classes $z_1$, $c_1$, $z_2$ and $c_2$ are the {\it universal loop $U(n)$-classes} in the low degrees.

\begin{lemma}\label{tranlemma}
The free suspension for $BU(n)$ satisfies
\[
\nu(c_1)=z_1, \ \ \nu(c_2)=z_2+z_1c_1, \ \  \nu(c_1^2)=2z_1c_1.
\]
\end{lemma}
\begin{proof}
The formulae $\nu(c_1)=z_1$ and $\nu(c_1^2)=2z_1c_1$ can be obtained easily from Property (1) and (2) of $\nu$, or by the following similar argument for $\nu(c_2)$. Indeed, since $\nu$ is natural we have the commutative diagram
\begin{gather}
\begin{aligned}
\xymatrix{
H^{4}(BSpin^c(2n)) \ar[r]^{\nu} \ar[d]^{Br^\ast} &
H^{3}(BLSpin^c(2n)) \ar[d]^{BLr^\ast}\\
H^{4}(BU(n))\ar[r]^{\nu} &
H^{3}(BLU(n)).
}
\end{aligned}
\label{nuspincudiag}
\end{gather}
It is easy to check that by our choice of generators $BLr^\ast(s)=z_1$, $BLr^\ast(t)=c_1$.
Also since $\sigma^\ast$ is natural we see that for $\mu\in H^3(\Omega BSpin^c(2n)\simeq Spin^c(2n))$
\[
r^\ast (\mu)=r^\ast\circ \sigma^\ast(q_1)=\sigma^\ast\circ Br^\ast (q_1)=\sigma^\ast(-c_2)=-z_2.
\]
Again by our choice of generators it implies that $BLr^\ast(\mu)=-z_2$. Hence by Diagram (\ref{nuspincudiag}) and (\ref{spinnueq})
\[
\nu(-c_2)=\nu \circ Br^\ast(q_1)=BLr^\ast \circ \nu(q_1)=BLr^\ast(\mu-st)=-z_2-z_1c_1.
\]
This proves the lemma.
\end{proof}
Let us consider the free transgression for $BLU(n)_l$ in (\ref{Unlisoeq}) and (\ref{unextdiag}).
From (\ref{hbunleq}) and (\ref{fracsuchernreleq}) it is easy to check that 
\begin{equation}\label{h3blunleq}
H^{\leq 4}(BLU(n)_l)\cong \mathbb{Z}^{\leq 4}[\overline{z}_1,\overline{c}_1,z_2,c_2],
\end{equation}
such that 
\begin{equation}\label{blrhosclasseq}
\begin{split}
&(BLd_l)^\ast(g)=\overline{c}_1, (BLd_l)^\ast(h)=\overline{z}_1, (BL\rho_s)^\ast(z_1)=s\overline{z}_1, \\
&(BL\rho_s)^\ast(c_1)=s\overline{c}_1, (BL\rho_s)^\ast(z_2)=z_2,(BL\rho_s)^\ast(c_2)=c_2.
\end{split}
\end{equation}
\begin{corollary}\label{tranllemma}
The free suspension for $BU(n)_l$ satisfies
\[
\nu(\overline{c}_1)=\overline{z}_1, \ \ \nu(c_2)=z_2+s^2\overline{z}_1\overline{c}_1, \ \  \nu(\overline{c}_1^2)=2\overline{z}_1\overline{c}_1.
\]
\end{corollary}
\begin{proof}
By the naturality of free transgression, there is the commutative diagram
\begin{gather*}
\begin{aligned}
\xymatrix{
H^{\ast}(BU(n)) \ar[r]^<<<<<{\nu} \ar[d]^{B\rho_s^\ast} &
H^{\ast-1}(BLU(n)) \ar[d]^{BL\rho_s^\ast}\\
H^{\ast}(BU(n)_l)\ar[r]^<<<<<{\nu} &
H^{\ast-1}(BLU(n)_l).
}
\end{aligned}
\label{nuuuldiag}
\end{gather*}
Together with Lemma \ref{tranlemma} and (\ref{blrhosclasseq}) the corollary follows. 
\end{proof}
\subsection{Fractional loop U-structure} \label{sec: fracloopu}
Let $\pi: Y\rightarrow M$ be a map. Let $L_a$ be a complex line bundle determined by $a\in H^2(Y;\mathbb{Z})$ and $E$ be a complex vector bundle of rank $n$ classified by $f: Y\rightarrow BU(n)$. 
Suppose $E$ admits an $(a,\frac{1}{l})$-fractional $U(n)$-structure determined by Diagram (\ref{fldefdiag}). Recall $E$ has the fractional Chern class $c_k^{l,a}(E)=f^{l\ast}(c_k^{\q})$ for each $1\leq k\leq n$. By applying the free loop functor $L$ to Diagram (\ref{fracuprindiag}) and Diagram (\ref{fldefdiag}), we obtain the morphism of loop fibrations
\begin{gather}
\begin{aligned}
\xymatrix{
LU(1)\times LU(n) \ar[r]^<<<<<{Lj}  \ar[d]^{\Omega L\phi} &
LP_U(f^a) \ar[r]^<<<<<{Lp^a} \ar[d]^{L\Phi}&
LY\ar[d]^{L\pi}\\
LU(n)_{\mathbb{Q}} \ar[r]^{Lj} &
LP_{U_{\mathbb{Q}}}(f^l) \ar[r]^<<<<<{Lp^l}&
LM,
}
\end{aligned}
\label{Lfracuprindiag}
\end{gather}
classified by 
\begin{gather}
\begin{aligned}
\xymatrix{
LY \ar[d]^{L\pi} \ar[r]^<<<<<{Lf^a} & 
BLU(1)\times BLU(n) \ar[d]^{L\phi} \\
LM\ar[r]^<<<<<<<<{Lf^l} &
BLU(n)_\mathbb{Q},
}
\end{aligned}
\label{Lfldefdiag}
\end{gather}
where we have used the fact that $BLG\simeq LBG$ for any topological group $G$.
We may call Diagram (\ref{Lfracuprindiag}) the {\it fractional loop principal bundle} of $LE$ with the {\it relative structure group} $(LU(1)\times LU(n), LU(n)_{\mathbb{Q}})$ and the {\it classifying map} $(Lf^a, Lf^l)$ mapping to the {\it fractional loop classifying space} $L\phi$. 

By (\ref{h3bluneq}), we have
\begin{equation}\label{h3blunqeq}
H^{\leq 4}(BLU(n);\mathbb{Q})\cong \mathbb{Q}^{\leq 4}[z^{\q}_1,c^{\q}_1,z^{\q}_2, c^{\q}_2],
\end{equation}
where as before we use $x^{\q}$ to denote its counterpart $x$ in the rational cohomology.
From (\ref{h3blunqeq}) we can define the {\it fractional loop U-classes} of $LE$ by
\begin{equation}\label{fracloopuclasseq}
z_i^{l,a}(LE)=Lf^{l\ast}(\widetilde{z}_i), \ \ ~ c_i^{l,a}(LE)=Lf^{l\ast}(c_i^{\q}), 
\end{equation}
for $i=1$ or $2$.
Moreover, by (\ref{h3bluneq}) we have the {\it loop U-classes}
\begin{equation}\label{loopunclasseq}
z_i(LE)=Lf^{a\ast}(z_1), \ \ ~ c_i(LE)=Lf^{a\ast}(c_1),
\end{equation}
for $i=1$ or $2$.
Then by the naturality of the free suspension and Lemma \ref{tranlemma}, we have that
\begin{equation}\label{Etran1eq}
\begin{split}
& \nu(c_1^{l,a}(E))=z_1^{l,a}(LE), \ \ ~ \nu((c_1^{l,a}(E))^2)=2z_1^{l,a}(LE)c_1^{l,a}(LE),  \\
&\nu(c_2^{l,a}(E))=z_2^{l,a}(LE)+z_1^{l,a}(LE)c_1^{l,a}(LE),
\end{split}
\end{equation}
and 
\begin{equation}\label{Etran2eq}
\begin{split}
&\nu(c_1(E))=z_1(LE), \ \ ~  \nu(c_1^2(E))=2z_1(LE)c_1(LE), \\
&\nu(c_2(E))=z_2(LE)+z_1(LE)c_1(LE).
\end{split}
\end{equation}

\begin{definition}\label{fracloopundef}
Let $E$ be an $(a,\frac{1}{l})$-fractional $U(n)$-bundle as above. $E$ has an {\it $(a, \frac{1}{l})$-fractional loop $U(n)$-structure}, or simply a {\it fractional loop U-structure} if $z_1^{l,a}(LE)=0$.
\end{definition}

Let us study the universal case. Recall that $LU(1)\cong U(1)\times \Omega S^1\simeq U(1)\times \mathbb{Z}$ as groups, and then $LU(n)\cong LSU(n)\times U(1)\times \Omega S^1\simeq LSU(n)\times U(1)\times \mathbb{Z}$ as spaces. Define the group homomorphism
\begin{equation}\label{thetaseq}
\theta: LU(1)\stackrel{}{\longrightarrow} \Omega S^1
\end{equation}
by $\theta(z(t))=\frac{z(t)}{z(1)}$.
Consider the two group homomorphisms
\[
\theta_s: LU(1)\stackrel{L\tau_s}{\longrightarrow}LU(1)\stackrel{\theta}{\longrightarrow} \Omega S^1, \ \ \
\varepsilon: LU(n) \stackrel{L{\rm det}}{\longrightarrow} LU(1)\stackrel{\theta}{\longrightarrow} \Omega S^1.
\]
Then $\theta_s(z(t))=\big(\frac{z(t)}{z(1)}\big)^s$ and $\varepsilon(A(t))=\frac{{\rm det}(A(t))}{{\rm det}(A(1))}$. Up to homotopy $\theta_s$ can be identified with the composition
\[
U(1)\times \mathbb{Z}\stackrel{q_2}{\longrightarrow}\mathbb{Z}\stackrel{\times s}{\longrightarrow}\mathbb{Z},
\]
where $q_2$ is the projection onto the second direct summand. 

Let $\overline{LU}(n)_l$ be the group defined by pullback diagram
\begin{gather}
\begin{aligned}
\xymatrix{
\overline{LU}(n)_l \ar[r]^{\overline{\varepsilon}} \ar[d]^{\overline\rho_s} &
\Omega S^1 \ar[d]^{\Omega\tau_s} \\
LU(n) \ar[r]^{\varepsilon} &
\Omega S^1.
}
\end{aligned}
\label{Lunpulldiag}
\end{gather}
If we view the components of $LU(n)$ are indexed by $\mathbb{Z}$ under the isomorphism $\pi_0(LU(n))\cong\mathbb{Z}$,
it is not hard to see that $\overline{LU}(n)_l$ is the subgroup of $LU(n)$ consisting of all the components indexed in $s\mathbb{Z}$.
Define $\kappa_l$ as the composition
\begin{equation}\label{kappaldef}
\kappa_l: U(1)\times \overline{LU}(n)_l \stackrel{{\rm id}\times \overline{\varepsilon}}{\longrightarrow} U(1)\times \Omega S^1\stackrel{\cong}{\longrightarrow} LU(1),
\end{equation}
and $\psi_s$ as the composition
\begin{equation}\label{psisdef}
\psi_s: U(1)\times \overline{LU}(n)_l \stackrel{p_2}{\longrightarrow} \overline{LU}(n)_l \stackrel{\overline{\rho}_s}{\longrightarrow} LU(n),
\end{equation}
where $p_2$ is the projection onto the second factor. 
It follows that there is the morphism of (homotopy) group extensions 
\begin{gather}
\begin{aligned}
\xymatrix{
\{1\} \ar[r] &
 L_0U(n)\ar@{=}[d] \ar[r]^<<<<{\iota_{2l}} &
 U(1)\times \overline{LU}(n)_l \ar[d]^{\psi_s} \ar[r]^<<<<{\kappa_l} &
 LU(1)\ar[d]^{\theta_s} \ar[r] &
 \{1\} \\
  \{1\} \ar[r]&
 L_0U(n) \ar[r]^{\iota_{2}} &
 LU(n)  \ar[r]^<<<<<<<<<{\varepsilon} &
 \Omega S^1 \ar[r] &
 \{1\}, 
}
\end{aligned}
\label{u1lunextdiag}
\end{gather}
where $L_0U(n)$ denotes the component of $LU(n)$ corresponding to $0\in \mathbb{Z}\cong\pi_0(LU(n))$, $\iota_{2}$ is the obvious inclusion, and $\iota_{2l}$ is the induced map. The right square of Diagram (\ref{u1lunextdiag}) is a pullback and determines the group extension 
\begin{equation}\label{lu1unexteq2}
\{1\}\stackrel{}{\longrightarrow} U(1)\times \overline{LU}(n)_l\stackrel{(\kappa_l,\psi_s)}{\longrightarrow} LU(1)\times LU(n)\stackrel{\theta_{-s}\times \varepsilon}{\longrightarrow} \Omega S^1\stackrel{}{\longrightarrow}\{1\}.
\end{equation}

By our convention of notation we may denote 
\begin{equation}\label{Hblu1eq}
H^\ast(BLU(1))\cong \mathbb{Z}[g]\{h\}
\end{equation} 
where $g$ and $h$ are the generators corresponding to $g\in H^2(BU(1))$ and $h\in H^1(B\Omega S^1\simeq S^1)$ respectively. Moreover, Diagram (\ref{Lunpulldiag}) gives the finite covering
\[
\mathbb{Z}/s\mathbb{Z}\stackrel{}{\longrightarrow} B\overline{LU}(n)_l \stackrel{B\overline{\rho}_s}{\longrightarrow} BLU(n).
\]
It follows that there is a class $\overline{z}_1\in H^1(B\overline{LU}(n)_l)$ such that 
\begin{equation}\label{barz1defeq}
(B\overline{\varepsilon})^\ast(h)=\overline{z}_1, (B\overline{\rho}_s)^\ast(z_1)=s\overline{z}_1.
\end{equation}
With (\ref{h3bluneq}) denote $z_2:=(B\overline{\rho}_s)^\ast(z_2)$ and $c_i:=(B\overline{\rho}_s)^\ast(c_i)$ for $i=1$, $2$.
Consider Diagram (\ref{u1lunextdiag}) after applying the classifying functor $B$. 
By the above argument, we have 
\begin{equation}\label{fracloopuchernreleq}
\begin{split}
&B\kappa_l^\ast (h)=\overline{z}_1, \ \ B\kappa_l^\ast(g)=g, \ \   \\
&B\psi_s^\ast (c_1)=c_1, \ \ ~B\psi_s^\ast(z_1)=s\overline{z}_1, \ \ ~ \\
&B\psi_s^\ast (c_2)=c_2, \ \ ~ B\psi_s^\ast (z_2)=z_2, \\
& B\varepsilon^\ast(h)=z_1, \ \ ~ B(\theta_{-s}\times \varepsilon)^\ast(h)=z_1-sh.
\end{split}
\end{equation}
Since the transgression $\nu$ is natural, by Lemma \ref{tranlemma} and (\ref{uniphiclasseq}) we have 
\[
L\phi^\ast(z^{\q}_1)=L\phi^\ast \circ \nu (c^{\q}_1)=\nu\circ \phi^\ast(c^{\q}_1)=\nu(c_1-sg)=z_1-sh.
\]
Hence there is the homotopy commutative diagram of fibrations
\begin{gather}
\begin{aligned}
\xymatrix{
BU(1)\times B\overline{LU}(n)_l \ar[r]^<<<<<{B(\kappa_l,\psi_s)} \ar[d]^{\xi_2}&
BLU(1)\times BLU(n) \ar[r]^<<<<<{B(\theta_{-s}\times \varepsilon)} \ar[d]^{L\phi}&
S^1 \ar[d]^{r}\\
BL_0U(n)_{\mathbb{Q}} \ar[r]^{B\iota_{2\mathbb{Q}}} &
BLU(n)_{\mathbb{Q}}  \ar[r]^<<<<<<<{z^{\q}_1}&
S^1_{\mathbb{Q}},
}
\end{aligned}
\label{lunifracsudiag}
\end{gather}
where the top row is obtained from (\ref{lu1unexteq2}) by taking the classifying functor $B$, $z^{\q}_1$ is represented by $B\varepsilon_\mathbb{Q}$, and $\xi_2$ is induced from $L\phi$. 
Indeed, let $\phi_{\mathbb{Q}}: BU(1)_{\mathbb{Q}}\times BU(n)_{\mathbb{Q}}\stackrel{}{\rightarrow} BU(n)_{\mathbb{Q}}$ be the rationalization of $\phi$.
The homotopy pullback of $L\phi_{\mathbb{Q}}$ along $B\iota_{2\mathbb{Q}}$ gives a map $\widetilde{\xi}_2: BU(1)_{\mathbb{Q}}\times B\overline{LU}(n)_{l\mathbb{Q}}\stackrel{}{\rightarrow} BL_0U(n)_{\mathbb{Q}}$. Then $\xi_2$ can be chosen to be the composition
\begin{equation}\label{xi2def}
\xi_2: BU(1)\times B\overline{LU}(n)_{l}\stackrel{r}{\longrightarrow}BU(1)_{\mathbb{Q}}\times B\overline{LU}(n)_{l\mathbb{Q}}\stackrel{\widetilde{\xi}_2}{\longrightarrow} BL_0U(n)_{\mathbb{Q}}, 
\end{equation}
where $r$ is the rationalization. 

Now let us turn to the fractional loop bundle $LE$ determined by Diagram (\ref{fldefdiag}).
Besides the loop $U(n)$-classes in (\ref{loopunclasseq}), by (\ref{Hblu1eq}) we have the class
$a\in H^2(LY)$ corresponding to the class $a\in H^2(Y)$ by abuse of notation such that $Lf^{a\ast}(g)=a$, and the class $\mathfrak{a}\in H^1(LY)$ such that $Lf^{a\ast}(h)=\mathfrak{a}$ and $\nu(a)=\mathfrak{a}$.
We summarize the formulae for the involved obstructions in the following lemma.
\begin{lemma}\label{obfracloopulemma}
\[
Lf^{a\ast}(z_1-sh)=z_1(LE)-s\mathfrak{a}, \ \ ~  Lf^{l\ast}(z^{\q}_1)=z_1^{l,a}(LE). 
\]
\end{lemma}
\begin{proof}
$Lf^{a\ast}(z_1-sh)=z_1(LE)-s\mathfrak{a}$ follows from (\ref{loopunclasseq}) and $Lf^{a\ast}(h)=\mathfrak{a}$.
$Lf^{l\ast}(z^{\q}_1)=z_1^{l,a}(LE)$ by definition (\ref{fracloopuclasseq}). 
\end{proof}

We are ready to prove Theorem \ref{fracloopunthm}.
\begin{proof}[Proof of Theorem \ref{fracloopunthm}]
The bottom square in Diagram (\ref{fracloopunliftdiag}) is Diagram (\ref{Lfldefdiag}). The top right square in Diagram (\ref{fracloopunliftdiag}) is the left square in Diagram (\ref{lunifracsudiag}). 

If $E$ is a fractional loop U-bundle, then $z^{l,a}_1(LE)=0$, and hence by (\ref{Etran1eq}), (\ref{Etran2eq}) and (\ref{12fraccherneq})
\[
L\pi^\ast(z^{l,a}_1(LE))=L\pi^\ast \circ \nu (c^{L,a}_1(E))= \nu \circ \pi^\ast(c^{L,a}_1(E))=\nu(c_1(E)-sa)=z_1(LE)-s\mathfrak{a}=0.
\]
With Diagram (\ref{lunifracsudiag}), this is equivalent to that both the compositions $z^{\q}_1\circ Lf^l$ and $B(\theta_{-s}\times \varepsilon)\circ Lf^a$ are null homotopic. It follows that there exist $g^a_2: LY\rightarrow BU(1)\times BLU(n)$ and $g^l_2: B\rightarrow BL_0U(n)_{\mathbb{Q}}$ such that $B(\kappa_l, \psi_s)\circ g^a_2\simeq L f^a$ and $B\iota_{2\mathbb{Q}}\circ g^l_2\simeq Lf^l$, in other words, the front and back triangles in Diagram (\ref{fracloopunliftdiag}) commute up to homotopy. 

In particular,
\begin{equation}\label{loopupfeq1}
B\iota_{2\mathbb{Q}}\circ \xi_2\circ g^a_2\simeq 
L\phi\circ B(\kappa_l, \psi_s)\circ g^a_2\simeq L\phi\circ Lf^a\simeq Lf^l\circ L\pi\simeq B\iota_{2\mathbb{Q}}\circ g_2^l\circ L\pi.
\end{equation}
On the other hand, it is clear that $BLU(n)\simeq BL_0U(n) \times U(1)$ such that $B\iota_{2\mathbb{Q}}$ admits a left homotopy inverse $\mathfrak{q}$. It follows from (\ref{loopupfeq1}) that
\[ 
\xi_2\circ g^a_2\simeq \mathfrak{q}\circ B\iota_{2\mathbb{Q}}\circ\xi_2\circ g^a_2\simeq \mathfrak{q}\circ B\iota_{2\mathbb{Q}}\circ g_2^l\circ L\pi\simeq g_2^l\circ L\pi,
\]
that is, the left top square commutes up to homotopy. Hence we have showed that Diagram (\ref{fracloopunliftdiag}) exists and commutes up to homotopy if $E$ is a fractional loop U-bundle. The proof of the converse statement is similar and omitted. Moreover, by the homotopy commutativity of the back triangle in Diagram (\ref{fracloopunliftdiag}), (\ref{fracloopunchernforeq}) follows immediately from (\ref{fracloopuchernreleq}), (\ref{fracloopuclasseq}) and (\ref{loopunclasseq}).

We are left to count the number of the fractional loop U-structures. Suppose $E$ is a fractional loop U-bundle. From Diagram (\ref{Lfracuprindiag}) and Diagram (\ref{Lfldefdiag}), there is the commutative diagram

\begin{gather}
\begin{aligned}
\xymatrix{
0\ar[r]  &
H^1(M) \ar[r]^<<<<<{p^{l\ast}}  \ar[d]_{\nu}^{\cong}&
H^1(P_{U_{\mathbb{Q}}}(f^l)) \ar[r]^<<<<<<<<{\delta\circ j^\ast}   \ar[d]_{\nu}^{\cong} &
H^2(BU(n)_{\mathbb{Q}}) \ar[r]^<<<<<<<<{f^{l\ast}} \ar[d]_{\nu}^{\cong}&
H^2(M)\ar[d]_{\nu}\\
0\ar[r]  &
H^0(LM) \ar[r]^<<<<<{Lp^{l\ast}}  \ar[d]_{L\pi^\ast}  &
H^0(LP_{U_{\mathbb{Q}}}(Lf^l)) \ar[r]^<<<<<<<<{\delta\circ Lj^\ast}   \ar[d]_{L\Phi^\ast} &
H^1(BLU(n)_{\mathbb{Q}}) \ar[r]^<<<<<<<<{Lf^{l\ast}} \ar[d]_{L\phi^\ast} &
H^1(LM)\ar[d]_{L\pi^\ast} \\
&
H^0(LY) \ar[r]^<<<<<{Lp^{a\ast}} &
H^0(LP_U(f^a)) \ar[r]^<<<<<{\delta\circ Lj^\ast}&
H^1(BLU(1)\times BLU(n)) \ar[r]^<<<<<{Lf^{a\ast}} &
H^1(LY),
}
\end{aligned}
\label{H12fracloopuprindiag}
\end{gather}
where the first row, as the second row in Diagram (\ref{H12fracuprindiag}), is exact and transgresses to the second row by the naturality of the free suspension, and the transgressions $\delta: H^{0}(LU(1)\times LU(n))\stackrel{}{\rightarrow}H^1(BLU(1)\times BLU(n))$ and $\delta: H^0(LU(n)_{\mathbb{Q}})\stackrel{}{\rightarrow}H^1(BLU(n)_{\mathbb{Q}})$ are isomorphisms.
Since for $i=0$, or $1$, 
\[
H^i(X;\mathbb{Q})\cong {\rm Hom}(H_i(X),\mathbb{Q})\cong {\rm Hom}(\pi_i(X),\mathbb{Q}),
\]
it is easy to see that the free suspensions $\nu$ in the first three columns of Diagram (\ref{H12fracloopuprindiag}) are isomorphisms. Hence the second row in Diagram (\ref{H12fracloopuprindiag})
is exact except the last term. Since $Lf^{l\ast}(z^{\q}_1)=z_1^{l,a}(LE)=0$ and $Lf^{a\ast}(z_1-sh)=z_1(LE)-s\mathfrak{a}=0$ by Lemma \ref{obfracloopulemma}, a fractional loop U-structure is determined by the choice of the pair $(P, L\Phi^\ast(P))\in H^0(LP_{U_{\mathbb{Q}}}(f^l))\times H^0(LP_U(f^a))$ such that $\delta\circ Lj^\ast(P)=z^{\q}_1$. From the exactness of first three arrows in the second row of Diagram (\ref{H12fracloopuprindiag}), there are exactly $H^0(LM)$ many of such choices. 
Hence the fractional loop U-structures on $LE$ are in one-to-one correspondence with the elements of $H^0(LM)$. This completes the proof of the theorem.
\end{proof}

By Theorem \ref{fracloopunthm} $\xi_2$ is a ``classifying space'' of the fractional loop U-bundle $E$ on the level of characteristic classes. Hence, we may call the universal map $\xi_2: BU(1)\times B\overline{LU}(n)_l\stackrel{}{\longrightarrow}BL_0U(n)_\mathbb{Q}$ the {\it fractional classifying space} of the fractional loop U-structure, and the fractional loop U-bundle $E$ has the {\it classifying map} $(g^a_2, g^l_2)$.

Geometrically, in Diagram (\ref{fracloopunliftdiag}) the structure group $LU(n)_{\mathbb{Q}}$ of the rational loop vector bundle determined by $Lf^l$ is lifted to $L_0U(n)_{\mathbb{Q}}$ through $B\iota_{2\mathbb{Q}}$, while the structure group $LU(1)\times LU(n)$ of the loop vector bundle $L(L_a\oplus E)$ is lifted to $U(1)\times \overline{LU}(n)_l$ through $B(\kappa_l, \psi_s)$. Hence a fractional U-bundle $E$ after looping admits a fractional loop U-structure if and only if its relative structural group ($LU(1)\times LU(n)$, $LU(n)_{\mathbb{Q}}$) can be lifted to ($U(1)\times \overline{LU}(n)_l$, $L_0U(n)_{\mathbb{Q}}$).
In terms of principal bundles, this means that the fractional loop principal bundle (\ref{Lfracuprindiag}) of the fractional U-bundle $E$ after looping can be lifted to the fractional loop principal bundle 
\begin{gather}
\begin{aligned}
\xymatrix{
U(1)\times \overline{LU}(n)_l \ar[r]^<<<<<{\jmath_{2l}}  \ar[d]^{\Omega \xi_2} &
P_{U(1)\times \overline{LU}_l}(g_2^a) \ar[r]^<<<<<{\mathfrak{p}_2^a} \ar[d]^{\Theta_2}&
LY\ar[d]^{L\pi}\\
L_0U(n)_{\mathbb{Q}} \ar[r]^<<<<<{\jmath_2} &
P_{L_0U_{\mathbb{Q}}}(g_2^l) \ar[r]^<<<<<{\mathfrak{p}_2^l}&
LM,
}
\end{aligned}
\label{fracloopuprindiag}
\end{gather}
where the top row is the principal bundle classified by $g_2^a$, and the bottom row is the principal fibration induced from $g_2^l$.

\subsection{Fractional loop U-structure vs fractional SU-structure}
\label{sec: fracloopuvssu}
There are two ways to compare the fractional loop U-structure with the fractional SU-structure, from the perspective of classifying spaces or from the perspective of free suspension.

Let us first look at the universal case from the perspective of classifying spaces. Note that the universal diagram for fractional SU-structure is Diagram (\ref{unifracsudiag}). Apply the free loop functor on it, we obtain the diagram
\begin{gather}
\begin{aligned}
\xymatrix{
BLU(n)_l \ar[r]^<<<<<<{BL(d_l, \rho_s)} \ar[d]^{L\phi_2}&
BLU(1)\times BLU(n) \ar[r]^<<<<{BL\mu_s} \ar[d]^{L\phi}&
BLU(1) \ar[d]^{BLr}\\
BLSU(n)_{\mathbb{Q}} \ar[r]^{BLi_{2\mathbb{Q}}} &
BLU(n)_{\mathbb{Q}}  \ar[r]^{BL{\rm det}_{\mathbb{Q}}}&
BLU(1)_\mathbb{Q},
}
\end{aligned}
\label{Lunifracsudiag}
\end{gather}
which is different from the universal diagram for fractional loop U-structure, Diagram (\ref{lunifracsudiag}). Indeed, by the free suspension formulae in Lemma \ref{tranlemma}, it is clear that in Diagram (\ref{Lunifracsudiag}) $BLi_{2\mathbb{Q}}$ kills the universal classes $z^{\q}_1$ and $c^{\q}_1$ represented by the map $BL{\rm det}_{\mathbb{Q}}$, while in Diagram (\ref{lunifracsudiag}) $B\iota_{2\mathbb{Q}}$ only kills $z^{\q}_1$ represented by the map $B\varepsilon_{\mathbb{Q}}$; similarly for the top rows of both diagrams. Nevertheless, if we kills the class $c^{\q}_1$ in Diagram (\ref{lunifracsudiag}) through a suitable way, we can achieve Diagram (\ref{Lunifracsudiag}) to make both structures match with each other. 

Consider the diagram of groups
\begin{gather}
\begin{aligned}
\xymatrix{
LU(n)_l  \ar@/^1.5pc/[rrd]^{Ld_l}   \ar@/_1.5pc/[ddr]_{L\rho_s}  \ar@{.>}[dr]^{\widehat{Li_{2l}}} \\
& U(1)\times \overline{LU}(n)_l \ar[d]^{\psi_s} \ar[r]^<<<<{\kappa_l} &
 LU(1)\ar[d]^{\theta_s} \\
 &LU(n)  \ar[r]^<<<<<<<<<{\varepsilon} &
 \Omega S^1,
 }
\end{aligned}
\label{lunlbarpulldiag}
\end{gather}
where the lower right square is the pullback defined in (\ref{u1lunextdiag}), $d_l$ and $\rho_s$ are defined in (\ref{unextdiag}), and the map $\widehat{Li_{2l}}$ will be defined momentarily. Indeed, by (\ref{unextdiag}) and the definition of the involved maps (\ref{thetaseq}), 
\[
\theta_s\circ Ld_l=\theta\circ L\tau_s\circ Ld_l=\theta\circ L{\rm det}\circ L\rho_s=\varepsilon \circ L\rho_s.
\]
Then by the universal property of pullback, there is a unique group homomorphism $\widehat{Li_{2l}}: LU(n)_l\rightarrow U(1)\times \overline{LU}(n)_l$ such that $\kappa_l\circ \widehat{Li_{2l}}=Ld_l$ and $\psi_s\circ \widehat{Li_{2l}}=L\rho_s$.
Apply the classifying functor $B$ to Diagram (\ref{lunlbarpulldiag}). 
By (\ref{fracloopuchernreleq}) and (\ref{blrhosclasseq}) it follows that
\begin{equation}\label{bhatl2ilclasseq}
\begin{split}
&(B\widehat{Li_{2l}})^\ast(g)=\overline{c}_1, (B\widehat{Li_{2l}})^\ast(\overline{z}_1)=\overline{z}_1, \\
&(B\widehat{Li_{2l}})^\ast(c_1)=s\overline{c}_1, (B\widehat{Li_{2l}})^\ast(z_2)=z_2, (B\widehat{Li_{2l}})^\ast(c_2)=c_2.\\
\end{split}
\end{equation}

Denote by $q_1: LU(1)\cong U(1)\times\Omega S^1\longrightarrow U(1)$, and $p_2: U(1)\times \overline{LU}(n)_l\longrightarrow \overline{LU}(n)_l$ the projections onto the first and second direct summands respectively. Consider the diagram 
\begin{gather}
\begin{aligned}
\xymatrix{
LU(n)_l \ar[r]^{Ld_l} \ar[d]_{p_2\circ \widehat{Li_{2l}}} &
LU(1) \ar[r]^{q_1} \ar[d]^{L\tau_s}&
U(1) \ar[d]^{\tau_s}\\ 
\overline{LU}(n)_l \ar[r]^{L{\rm det}\circ \overline{\rho}_s}  &
LU(1) \ar[r]^{q_1}&
U(1),
}
\end{aligned}
\label{lunlu1diag}
\end{gather}
where the right square commutes by naturality. By the definition of $\psi_s$ (\ref{psisdef}), Diagram (\ref{lunlbarpulldiag}) and (\ref{unextdiag})
\[
 L{\rm det}\circ\overline{\rho}_s \circ p_2\circ \widehat{Li_{2l}}=L{\rm det}\circ \psi_s\circ \widehat{Li_{2l}}=L{\rm det}\circ L\rho_s=L\tau_s\circ Ld_l,
\]
and then the left square of (\ref{lunlu1diag}) commutes. Moreover, from the definition of $\overline{\rho}_s$ in (\ref{Lunpulldiag}) and (\ref{unextdiag}), it is easy to see that the outer square of Diagram (\ref{lunlu1diag}) is a pullback. It follows that there is a group extension
\begin{equation}\label{lu1unexteq3}
\{1\}\stackrel{}{\longrightarrow} LU(n)_l \stackrel{\widehat{Li_{2l}}}{\longrightarrow} U(1)\times \overline{LU}(n)_l \stackrel{\tau_{-s}\times \overline{\varepsilon}^\prime}{\longrightarrow} U(1)\stackrel{}{\longrightarrow}\{1\},
\end{equation}
where $\overline{\varepsilon}^\prime$ is defined to be $q_1\circ L{\rm det}\circ\overline{\rho}_s$, and $\widehat{Li_{2l}}=(q_1\circ Ld_l, p_2\circ \widehat{Li_{2l}})$ by definition. It is easy to check that $B(\tau_{-s}\times \overline{\varepsilon}^\prime)^\ast(g)=\overline{c}_1-sg$.

On the other hand, by the bottom row of Diagram (\ref{lunifracsudiag}) and (\ref{h3blunqeq}), it is clear that
\begin{equation}\label{h3bl0unqeq}
H^{\leq 3}(BL_0U(n);\mathbb{Q})\cong \mathbb{Q}^{\leq 3}[c^{\q}_1,z^{\q}_2].
\end{equation}
\begin{lemma}\label{xi2classlemma}
For $\xi_2: BU(1)\times B\overline{LU}(n)_l\stackrel{}{\rightarrow} BL_0U(n)_{\mathbb{Q}}$ in Diagram (\ref{lunifracsudiag}),
\[
\xi_2^\ast(c^{\q}_1)=c_1-sg, \ \ ~\xi_2^\ast(z^{\q}_2)=z_2+\frac{1}{l}\overline{z}_1c_1.
\]
\end{lemma}
\begin{proof}
Recall by (\ref{uniphiclasseq}), 
\[
\phi^\ast(c^{\q}_1)=c_1-sg, \ \  ~\phi^\ast(c^{\q}_2)=c_2-\frac{n-1}{l}gc_1+\frac{s(n-1)}{2l}g^2.
\]
Then by (\ref{fracloopuchernreleq}) and Diagram (\ref{lunifracsudiag})
\[
\xi_2^\ast(c^{\q}_1)=\xi_2^\ast\circ B\iota_{2\mathbb{Q}}^\ast(c^{\q}_1)= B(\kappa_l,\psi_s)^\ast\circ L\phi^\ast(c^{\q}_1)=B(\kappa_l,\psi_s)^\ast(c_1-sg)=c_1-sg,
\]
and moreover with the naturality of the free suspension and Lemma \ref{tranlemma}
\[
\begin{split}
\xi_2^\ast(z^{\q}_2)
&=B(\kappa_l,\psi_s)^\ast\circ L\phi^\ast(\nu(c^{\q}_2)-z^{\q}_1c^{\q}_1)\\
&=B(\kappa_l,\psi_s)^\ast\big(\nu \circ \phi^\ast(c^{\q}_2)-(z_1-sh)(c_1-sg)\big)\\
&=B(\kappa_l,\psi_s)^\ast\big(\nu (c_2-\frac{n-1}{l}gc_1+\frac{s(n-1)}{2l}g^2)-(z_1-sh)(c_1-sg)\big)\\
&=B(\kappa_l,\psi_s)^\ast\big(z_2+z_1c_1-\frac{n-1}{l}(hc_1+gz_1)+\frac{s(n-1)}{l}gh-(z_1-sh)(c_1-sg)\big)\\
&=B(\kappa_l,\psi_s)^\ast(z_2+z_1c_1-\frac{n-1}{l}hc_1+\frac{1}{l}(z_1-sh)(g-lc_1))\\
&=z_2+\frac{1}{l}\overline{z}_1c_1.
\end{split}
\]
\end{proof}
Recall $p: LU(n)\rightarrow U(n)$ is the evaluation map at $1\in S^1$ defined in (\ref{generalfreefib}).
The composition
\[
BL_0U(n)_{\mathbb{Q}}\stackrel{B\iota_{2\mathbb{Q}}}{\longrightarrow} BLU(n)_{\mathbb{Q}}
\stackrel{B({\rm det}\circ p)_{\mathbb{Q}}}{\longrightarrow} BU(1)_{\mathbb{Q}},
\]
satisfies that 
$B({\rm det}\circ p\circ\iota_2)_{\mathbb{Q}}^\ast(c^{\q}_1)=c^{\q}_1$.
By Lemma \ref{xi2classlemma} and the previous discussion, we can construct the homotopy commutative diagram of homotopy fibrations
\begin{gather}
\begin{aligned}
\xymatrixcolsep{3.8pc}
\xymatrix{
BLU(n)_l \ar[r]^<<<<<<<<<{B\widehat{Li_{2l}}} \ar[d]^{\xi_2^\prime}&
BU(1)\times B\overline{LU}(n)_l \ar[r]^<<<<<<<<<{B(\tau_{-s}\times \varepsilon^\prime)} \ar[d]^{\xi_2}&
BU(1) \ar[d]^{Br}\\
BLSU(n)_{\mathbb{Q}} \ar[r]^{B\widehat{Li_2}_{\mathbb{Q}}} &
BL_0U(n)_{\mathbb{Q}}  \ar[r]^{B({\rm det}\circ p\circ\iota_2)_{\mathbb{Q}}}&
BU(1)_\mathbb{Q},
}
\end{aligned}
\label{unifracloopuvssudiag}
\end{gather}
where the homotopy fibration in the top row is obtained by applying the classifying functor to the extension (\ref{lu1unexteq2}), $\widehat{Li_2}: LSU(n)\rightarrow L_0U(n)$ is the loop inclusion $Li_2$ restricted into the component of $LU(n)$ containing the image, and $\xi_2^\prime$ is the induced map.
\begin{lemma}\label{xi2'=Lphi2lemma}
In Diagram (\ref{unifracloopuvssudiag}), $\xi_2^\prime$ can be chosen to be $L\phi_2$, and
\[
(\kappa_l,\psi_s)\circ \widehat{Li_{2l}}=L(d_l, \rho_s), \ \ \ \iota_2 \circ \widehat{Li_2}=Li_2.
\]
\end{lemma}
\begin{proof}
First, $\iota_2 \circ \widehat{Li_2}=Li_2$ holds by definiton, and $(\kappa_l,\psi_s)\circ \widehat{Li_{2l}}=L(d_l, \rho_s)$ holds by (\ref{lunlbarpulldiag}).
To show that $\xi_2^\prime$ can be chosen to be $L\phi_2$, notice that by the homotopy commutativity of Diagram (\ref{Lunifracsudiag}) and (\ref{unifracsudiag}),
\begin{equation}\label{xi2'=Lphi2lemmapfeq1}
\begin{split}
&B\iota_{2\mathbb{Q}}\circ \xi_2\circ B\widehat{Li_{2l}}
\simeq L\phi \circ B(\kappa_l,\psi_s)\circ B\widehat{Li_{2l}}\\
&\simeq L\phi \circ BL(d_l, \rho_s)
\simeq BLi_{2\mathbb{Q}} \circ L\phi_2
\simeq B\iota_{2\mathbb{Q}}\circ  B\widehat{Li_2}_{\mathbb{Q}}\circ L\phi_2.
\end{split}
\end{equation}
On the other hand, the homotopy fibration
$
BL_0U(n)\stackrel{B\iota_2}{\rightarrow}BLU(n)\stackrel{}{\rightarrow} U(1)
$
splits, and $B\iota_2$ admits a left homotopy inverse $q$. Hence by (\ref{xi2'=Lphi2lemmapfeq1})
\[
\xi_2\circ B\widehat{Li_{2l}}\simeq q_{\mathbb{Q}} \circ B\iota_{2\mathbb{Q}}\circ \xi_2\circ B\widehat{Li_{2l}}\simeq q_{\mathbb{Q}} \circ B\iota_{2\mathbb{Q}}\circ  B\widehat{Li_2}_{\mathbb{Q}}\circ L\phi_2\simeq B\widehat{Li_2}_{\mathbb{Q}}\circ L\phi_2,
\]
which means that the left square in Diagram (\ref{unifracloopuvssudiag}) commutes up to homotopy if we let $\xi_2^\prime=L\phi_2$. This completes the proof of the lemma. 
\end{proof}

By Lemma \ref{xi2'=Lphi2lemma}, we see that the left square of Diagram (\ref{Lunifracsudiag}) admits a factorization
\begin{gather}
\begin{aligned}
\xymatrixcolsep{3.2pc}
\xymatrix{
BLU(n)_l \ar[r]^<<<<<<<<<{B\widehat{Li_{2l}}} \ar[d]^{L\phi_2}    \ar@/^2.2pc/[rr]^{BL(d_l, \rho_s)} &
BU(1)\times B\overline{LU}(n)_l  \ar[d]^{\xi_2} \ar[r]^<<<<<<{B(\kappa_l,\psi_s)} &
BLU(1)\times BLU(n) \ar[d]^{L\phi}\\
BLSU(n)_{\mathbb{Q}} \ar[r]^{B\widehat{Li_2}_{\mathbb{Q}}}  \ar@/_2.2pc/[rr]^{BLi_{2\mathbb{Q}}} &
BL_0U(n)_{\mathbb{Q}} \ar[r]^{B\iota_{2\mathbb{Q}}} &
BLU(n)_{\mathbb{Q}}, 
}
\end{aligned}
\label{factorLunifracsudiag}
\end{gather}
which establishes the relation between the fractional loop U-structure and the fractional SU-structure in the universal case.

Now let us turn to the fractional U-bundle $E$ defined by Diagram (\ref{fldefdiag}). The criteria of fractional SU-structure and fractional loop U-structure are given in Theorem \ref{fracsunthm} and Theorem \ref{fracloopunthm} respectively. By the above analysis, if $E$ admits a fractional SU-structure, it is clear that $E$ is fractional loop U. Conversely, if $E$ is fractional loop U, then in order to lift it to a fractional SU-structure, one has to further kill the first fractional Chern class $c_1^{l,a}(LE)$.
\begin{theorem}\label{fracloopuvssuthm1}
Let $E$ be an $(a,\frac{1}{l})$-fractional $U(n)$-bundle determined by Diagram (\ref{fracudiagintro}) and $l>1$. 

If $E$ admits an $(a,\frac{1}{l})$-fractional $SU(n)$-structure, then $E$ is fractional loop U. 

Conversely, suppose $E$ is fractional loop U with the classifying map $(g_2^a, g_2^l)$ in Diagram (\ref{fracloopunliftdiag}). Then it can be lifted to a fractional SU-structure if and only if there exists a map $g_2^{l\prime}$ such that the following diagram commutes up to homotopy
\begin{gather}
\begin{aligned}
\xymatrix{
& BLSU(n)_{\mathbb{Q}} \ar[dr]^{B\widehat{Li_2}_{\mathbb{Q}}} \\
LM \ar@{.>}[ur]^<<<<<<<<{g_2^{l\prime}} \ar[rr]^<<<<<<<<<<<<<<<<<<{g_2^{l}}&&
BL_0U(n)_{\mathbb{Q}}.
}
\end{aligned}
\label{fracloopun+liftdiag}
\end{gather}
\end{theorem}
\begin{proof}
Suppose that the lift $g_2^{l\prime}$ in Diagram (\ref{fracloopun+liftdiag}) exists. Then by Diagram (\ref{unifracloopuvssudiag})
\[
B({\rm det}\circ p\circ\iota_2)_{\mathbb{Q}}\circ g_2^{l}\simeq B({\rm det}\circ p\circ\iota_2)_{\mathbb{Q}}\circ B\widehat{Li_2}_{\mathbb{Q}}\circ g_2^{l\prime}
\]
is null homotopic. It implies that the first fractional Chern class $c_1^{l, a}(LE)=0$. Since $p^\ast(c_1^{l, a}(E))=c_1^{l, a}(LE)$ with $p$ the evaluation map in (\ref{generalfreefib}) and $p^\ast$ is injective, we have $c_1^{l, a}(E)=0$, and then $E$ is fractional SU. Suppose a fractional SU-structure is given by a lift $f_2^l$ as in Diagram (\ref{fracsunliftdiag}), we have to show that after looping it lifts the original fractional loop U-structure given by $g_2^l$. Indeed, by the similar argument as in the last part in the proof of Lemma \ref{xi2'=Lphi2lemma} we have 
\[
B\iota_{2\mathbb{Q}}\circ  B\widehat{Li_2}_{\mathbb{Q}}\circ Lf_2^{l}
\simeq BLi_{2\mathbb{Q}} \circ Lf_2^{l}
\simeq Lf^l\simeq  B\iota_{2\mathbb{Q}}\circ g_2^l.
\]
Then since $B\iota_{2\mathbb{Q}}$ admits a left homotopy inverse, $B\widehat{Li_2}_{\mathbb{Q}}\circ Lf_2^{l}\simeq g_2^l$. This shows that the fractional SU-structure defined by $f_2^l$ lifts the fractional loop U-structure given by $g_2^l$. The converse statement and the other part of the theorem follow immediately from Diagram (\ref{factorLunifracsudiag}).
\end{proof}
Theorem \ref{fracloopuvssuthm1} and Diagram (\ref{factorLunifracsudiag}) are to compare the fractional SU-structure and the fractional loop U-structure from the perspective of classifying spaces. Classically we can also compare them from the perspective of the free suspension.

\begin{theorem}\label{fracloopuvssuthm2}
Let $E$ be an $(a,\frac{1}{l})$-fractional $U(n)$-bundle determined by Diagram (\ref{fracudiagintro}) and $l>1$. 

If $E$ is fractional SU, then $E$ is fractional loop U. Moreover, the distinct fractional SU-structures on $E$ transgress to the fractional loop U-structures via the isomorphic free suspension
\[
\nu: H^1(M)\stackrel{\cong}{\longrightarrow}H^0(LM).
\]
The converse statement is also true if the first fractional Chern class $c_1^{l,a}(E)$ is rationally spherical, that is, there exists a map $k_1: S^2\stackrel{}{\rightarrow} M$ such that $k_1^\ast(c_1^{l,a}(E))$ is nontrivial.

In particular, when $M$ is rationally simply connected, $E$ is fractional loop U implies that $E$ admits a fractional SU-structure.
\end{theorem}
\begin{proof}
Recall by (\ref{Etran1eq}) $\nu(c_1^{l,a}(E))=z_1^{l,a}(LE)$. Then if $E$ admits a fractional SU-structure, $c_1^{l,a}(E)=0$ implies that $z_1^{l,a}(LE)=0$ which means that $E$ admits a fractional loop U-structure. By Theorem \ref{fracsunthm} and Theorem \ref{fracloopunthm}, the two structures are classified by $H^1(M)$ and $H^0(LM)\cong H^1(M)$ respectively. Then by the naturality of the free suspension, there is a one-to-one correspondence between the distinct fractional SU-structures and the fractional loop U-structures on $E$ through the free suspension. 

To prove the converse statement, we adopt the idea in the proof of Theorem $3.1$ in \cite{Mcl} which describes the free suspension $\nu$ geometrically for the elements in the Hurewicz image. As summarized in the proof of Theorem 5.1 of \cite{DHH}, there is the commutative diagram
\begin{gather}
\begin{aligned}
\xymatrix{
\pi_2(M)\otimes \mathbb{Q} \ar[r]^{h\otimes \mathbb{Q} } \ar@{_{(}->}[d]^{i\otimes \mathbb{Q}}  & H_2(M;\mathbb{Q})  \ar[r]^{{\rm dual}\ \ }_{\cong}  &H^2(M;\mathbb{Q}) \ar[d]^{\nu}  \\
\pi_1(LM) \otimes \mathbb{Q}\ar[r]^{h\otimes \mathbb{Q} }_{\cong}                        & H_1(LM; \mathbb{Q}) \ar[r]^{{\rm dual}\ \ }_{\cong} & H^1(LM;\mathbb{Q}),
}
\end{aligned}
\label{fracsugeotransdia}
\end{gather}
where $h$ is the Hurewicz homomorphism, and $i$ is the inclusion under the isomorphism $\pi_1(LM)\cong \pi_2(M)\oplus \pi_1(M)$. By the assumption, $c_1^{l, a}(E)\in H^2(M;\mathbb{Q})$ is rationally spherical. It follows that the composition in the top row of Diagram (\ref{fracsugeotransdia}) sends the class $[k_1]$ to $c_1^{l, a}(E)$ up to a nonzero constant.
 If $E$ is fractional loop U, or equivalently $\nu(c_1^{l,a}(E))=z_1^{l,a}(LE)=0$, by the commutativity of Diagram (\ref{fracsugeotransdia}) $i([k_1])=0$ and then $[k_1]=0$. Hence $c_1^{l, a}(E)=0$ and $E$ is fractional SU. This proved the converse statement.
 
If $M$ is rationally simply connected, then $h\otimes \mathbb{Q}$ in the top row of Diagram (\ref{fracsugeotransdia}) is an isomorphism by the rational Hurewicz theorem. Hence the spherical condition is satisfied and $E$ is fractional SU. This proves the special case and the theorem is proved.
 \end{proof}

\subsection{Fractional loop SU-structure}
\label{sec: fracloopsu}
Let $\pi: Y\rightarrow M$ be a map. Let $L_a$ be a complex line bundle determined by $a\in H^2(Y;\mathbb{Z})$ and $E$ be a complex vector bundle of rank $n$ classified by $f: Y\rightarrow BU(n)$. 
Suppose $l>1$ and $E$ admits an $(a,\frac{1}{l})$-fractional $SU(n)$-structure as described in Definition \ref{fracsundef}.
\begin{definition}\label{fracloopsundef}
Let $E$ be an $(a,\frac{1}{l})$-fractional $SU(n)$-bundle. $E$ has an {\it $(a, \frac{1}{l})$-fractional loop $SU(n)$-structure}, or simply a {\it fractional loop SU-structure} if $z_2^{l,a}(LE)=0$.
\end{definition}

Let us study the universal case. Applying the free loop functor to the fractional classifying map $\phi_2$ in Diagram (\ref{unifracsudiag}) we get $L\phi_2: BLU(n)_l\stackrel{}{\rightarrow} BLSU(n)_\mathbb{Q}$.
By either Diagram (\ref{factorLunifracsudiag}), Lemma \ref{xi2classlemma} and (\ref{bhatl2ilclasseq}), or Lemma \ref{tranllemma}, (\ref{phi2c2eq}) and the naturality of free suspension, one can check that
\begin{equation}\label{lphi2z2eq}
L\phi_2^\ast(z^{\q}_2)=z_2+\frac{s}{l}\overline{z}_1\overline{c}_1.
\end{equation}
We may define a topological group $\widehat{LSU}(n)_l$ by the pullback
\begin{gather}
\begin{aligned}
\xymatrix{
\widehat{LSU}(n)_l \ar[d]^{\Omega\xi_3}  \ar[r]^{\iota_{3l}}&
LU(n)_l \ar[d]^{\Omega L\phi_2} \\
\widehat{LSU}(n)_{\mathbb{Q}}  \ar[r]^{ \iota_{3\mathbb{Q}}}&
LSU(n)_{\mathbb{Q}},
}
\end{aligned}
\label{hatsunldiag}
\end{gather}
where $i_{3\mathbb{Q}}$ is the rationalization of the group extension
\begin{equation}\label{u1lsunexteq}
\{1\}\stackrel{}{\longrightarrow} U(1)\stackrel{}{\longrightarrow} \widehat{LSU}(n)\stackrel{\iota_{3}}{\longrightarrow}  LSU(n) \stackrel{}{\longrightarrow} \{1\}
\end{equation}
corresponding to the class $z_2\in H^3(BLSU(n))$, and $\xi_3=B\Omega\xi_3: B\widehat{LSU}(n)_l \stackrel{}{\rightarrow} B\widehat{LSU}(n)_{\mathbb{Q}}$ is the induced map between classifying spaces. Hence there is the induced group extension
\begin{equation}\label{hatsunlexteq}
\{1\}\stackrel{}{\longrightarrow} S^1_{\mathbb{Q}} \stackrel{}{\longrightarrow} \widehat{LSU}(n)_l \stackrel{\iota_{3l}}{\longrightarrow}  LU(n)_l  \stackrel{}{\longrightarrow} \{1\}.
\end{equation}
Let us denote
\[
\bar{z}_1=B\iota_{3l}^\ast(\overline{z}_1), ~\ \ ~  \bar{c}_1=B\iota_{3l}^\ast(\overline{c}_1), ~\ \ ~  z_2=B\iota_{3l}^\ast(z_2), ~\ \ ~  c_2=B\iota_{3l}^\ast(c_2).
\]
From (\ref{hatsunldiag}) and (\ref{lphi2z2eq}) we have that 
\begin{equation}\label{barz2eq}
z_2=-\frac{s}{l}\bar{z}_1\bar{c}_1.
\end{equation}

Now let us turn to the fractional SU-bundle $E$ determined by Diagram (\ref{fracsunliftdiag}) in Theorem \ref{fracsunthm}.
In particular, the SU-structure on $E$ is classified by the pair of maps $(f_2^a, f_2^l)$. Applying the free loop functor on $E$.
We summarize the formulae for the involved obstructions in the following lemma.
\begin{lemma}\label{obfracloopsulemma}
\[
Lf_2^{a\ast}(z_2+\frac{s}{l}\overline{z}_1\overline{c}_1)=z_2(LE)+\frac{1}{n}z_1(LE)c_1(LE),
~ \ \ ~  Lf_2^{l\ast}(z^{\q}_2)=z_2^{l,a}(LE). 
\]
Moreover, $c_1(LE)=sa$ and $z_1(LE)=s\mathfrak{a}$.
\end{lemma}
\begin{proof}
First $c_1(LE)=sa$ and $z_1(LE)=s\mathfrak{a}$ by (\ref{fracsunchernforeq}) and (\ref{fracloopunchernforeq}) since $E$ is fractional SU and hence fractional loop U.
$Lf_2^{l\ast}(z^{\q}_2)=z_2^{l,a}(LE)$ by definition (\ref{fracloopuclasseq}). Moreover, by 
(\ref{fracsunchernforeq}), (\ref{fracloopunchernforeq}), Lemma \ref{tranllemma} and the naturality of the free suspension,
\[
Lf_2^{a\ast}(z_2)=Lf_2^{a\ast}(\nu(c_2)-s^2\overline{z}_1\overline{c}_1)=\nu(c_2(E))-z_1(LE)c_1(LE)=z_2(LE),
\]
and hence
\[
\begin{split}
&Lf_2^{a\ast}(z_2+\frac{s}{l}\overline{z}_1\overline{c}_1)\\
&= z_2(LE)+\frac{s}{l}\cdot\frac{1}{s^2}z_1(LE)c_1(LE)\\
&=z_2(LE)+\frac{1}{n}z_1(LE)c_1(LE).
\end{split}
\]
This proves the first formula and hence the lemma.
\end{proof}

We are ready to prove Theorem \ref{fracloopsunthm}.
\begin{proof}[Proof of Theorem \ref{fracloopsunthm}]
By Theorem \ref{fracsunthm} the bottom square in Diagram (\ref{fracloopsunliftdiag}) commutes up to homotopy. The top right square in Diagram (\ref{fracloopsunliftdiag}) is Diagram (\ref{hatsunldiag}) after applying the classifying functor $B$. 

If $E$ is a fractional loop SU-bundle, then $z^{l,a}_2(LE)=0$. It follows that there exists a map $g_3^l: LM\rightarrow B\widehat{LSU}(n)_{\mathbb{Q}}$ such that $B\iota_{3\mathbb{Q}}\circ g_3^l\simeq Lf_2^l$, and we obtain the front triangle in Diagram (\ref{fracloopsunliftdiag}). Then since the top right square in Diagram (\ref{fracloopsunliftdiag}) is a homotopy pullback, there exists a unique $g_3^a: LY\rightarrow B\widehat{LSU}(n)_l$ up to homotopy such that $B\iota_{3l}\circ g_3^a\simeq Lf_2^a$ and $\xi_3 \circ g_3^a\simeq g_3^l\circ L\pi$. Thus the back triangle and the top left square in Diagram (\ref{fracloopsunliftdiag}) commute up to homotopy. We have showed that Diagram (\ref{fracloopsunliftdiag}) exists and commutes up to homotopy if $E$ is a fractional loop SU-bundle. The proof of the converse statement is similar and omitted. Moreover, by straightforward computation (\ref{fracloopsunchernforeq}) follows from (\ref{fracsunchernforeq}), (\ref{barz2eq}) and Diagram (\ref{fracloopsunliftdiag}). For instance, by the proof of Lemma \ref{obfracloopsulemma}
\[
g_3^{a\ast}(z_2)=g_3^{a\ast}\circ B\iota_{3l}^\ast(z_2)=Lf_2^{a\ast}(z_2)=z_2(LE).
\]
We are left to count the number of the fractional loop SU-structures. Suppose $E$ is a fractional loop SU-bundle. From Diagram (\ref{fracsuprindiag}) after applying the free loop functor $L$ and the bottom square of Diagram (\ref{fracloopsunliftdiag}), by the dual Blakers-Massey theorem \cite[Theorem C.3]{DHH} there is the commutative diagram
\begin{gather}
\begin{aligned}
\xymatrix{
 &
H^2(LY) \ar[r]^<<<<<<{Lp_2^{a\ast}} &
H^2(LP_{U_l}(f_2^a)) \ar[r]^<<<<<<{\delta\circ Lj_{2}^\ast}&
H^3(BLU(n)_l) \ar[r]^<<<<<<{Lf_2^{a\ast}} &
H^3(LY) \\
0\ar[r]  &
H^2(LM) \ar[r]^<<<<<{Lp_2^{l\ast}}  \ar[u]_{L\pi^\ast}&
H^2(LP_{SU_{\mathbb{Q}}}(Lf_2^l)) \ar[r]^<<<<<{\delta\circ Lj_2^\ast}   \ar[u]_{L\Phi_2^\ast} &
H^3(BLSU(n)_{\mathbb{Q}}) \ar[r]^<<<<<{Lf_2^{l\ast}} \ar[u]_{L\phi_2^\ast}&
H^3(LM)\ar[u]_{L\pi^\ast},
}
\end{aligned}
\label{H23loopfracuprindiag}
\end{gather}
where the second row is exact with the isomorphic transgression $\delta: H^2(LSU(n)_\mathbb{Q})\stackrel{}{\rightarrow} H^3(BLSU(n)_\mathbb{Q})$, and in the top row the transgression $\delta: H^2(LU(n)_l)\stackrel{}{\rightarrow} H^3(BLU(n)_l)$ is a monomorphism.
A fractional loop SU-structure is determined by the choice of the pair $(P, L\Phi_2^\ast(P))\in H^2(LP_{SU_{\mathbb{Q}}}(f_2^l))\times H^2(LP_{U_l}(f_2^a))$ such that $\delta\circ Lj_2^\ast(P)=z^{\q}_2$. From Diagram (\ref{H23loopfracuprindiag}), there are exactly $H^2(LM)$ many of such choices. Hence the fractional loop SU-structures on $E$ are in one-to-one correspondence with the elements of $H^2(LM)$. This completes the proof of the theorem.
\end{proof}
By Theorem \ref{fracloopsunthm} $\xi_3$ is a ``classifying space'' of the fractional loop SU-bundle $E$ on the level of characteristic classes. Hence, we may call the universal map $\xi_3: B\widehat{LSU}(n)_l \stackrel{}{\longrightarrow}B\widehat{LSU}(n)_\mathbb{Q}$ the {\it fractional classifying space} of the fractional loop SU-structure, and the fractional loop SU-bundle $E$ has the {\it classifying map} $(g^a_3, g^l_3)$.

Geometrically, in Diagram (\ref{fracloopsunliftdiag}) the structure group $LSU(n)_{\mathbb{Q}}$ of the rational vector bundle determined by $Lf_2^l$ is lifted to $\widehat{LSU}(n)_\mathbb{Q}$ through $B\iota_{3\mathbb{Q}}$, while the structure group $LU(n)$ of the vector bundle $L_a\oplus E$ after looping is lifted to $\widehat{LSU}(n)_l$ through $B\iota_{3l}$. Hence a fractional SU-bundle $E$ admits a fractional loop SU-structure if and only if its relative structural group ($LU(n)$, $LSU(n)_{\mathbb{Q}}$) can be lifted to ($\widehat{LSU}(n)_l$, $\widehat{LSU}(n)_\mathbb{Q}$).
In terms of principal bundles, this means that after looping the fractional principal bundle (\ref{fracsuprindiag}) of the fractional SU-bundle $E$ can be lifted to the fractional principal bundle 
\begin{gather}
\begin{aligned}
\xymatrix{
\widehat{LSU}(n)_l \ar[r]^<<<<<<{\jmath_3}  \ar[d]^{\Omega \xi_3} &
P_{\widehat{LSU}_l}(g_3^a) \ar[r]^<<<<<{\mathfrak{p}_3^a} \ar[d]^{\Theta_3}&
LY\ar[d]^{L\pi}\\
\widehat{LSU}(n)_\mathbb{Q}  \ar[r]^<<<<<{\jmath_3} &
P_{\widehat{LSU}_\mathbb{Q} }(g_3^l) \ar[r]^<<<<<{\mathfrak{p}_3^l}&
LM,
}
\end{aligned}
\label{fracloopsuprindiag}
\end{gather}
where the top row is the principal bundle classified by $g_3^a$, and the bottom row is the principal fibration induced from $g_3^l$.

\begin{remark}\label{loopsuloopspinrmk}
The group extension (\ref{u1lsunexteq}) can be compared with the extension
\begin{equation}\label{stringnexteq}
\{1\}\stackrel{}{\longrightarrow}U(1) \stackrel{}{\longrightarrow}\widehat{LSpin}(2n)\stackrel{\mathfrak{h}_3}{\longrightarrow} LSpin(2n)\stackrel{}{\longrightarrow}\{1\}
\end{equation}
corresponding to the generator $\mu\in H^3(BLSpin(2n))$ such that $\nu(\frac{p_1}{2})=\mu$. Recall there is the standard Lie group homomorphism $r: SU(n)\stackrel{}{\rightarrow}Spin(2n)$ such that $Br^\ast(p_1)=-2c_2$. We then have $BLr^\ast(\mu)=-z_2$.
Hence we can define $\widehat{LSU}(n)$ by the pullback of groups
\begin{gather}
\begin{aligned}
\xymatrix{
\widehat{LSU}(n) \ar[r]^{} \ar[d]^{\iota_3} &
\widehat{LSpin}(2n) \ar[d]^{\mathfrak{h}} \\
LSU(n) \ar[r]^<<<<<<{Lr} &
LSpin(2n).
}
\end{aligned}
\label{Lsuspindiag}
\end{gather}
\end{remark}

\subsection{Fractional loop SU-structure vs fractional U$\langle 6 \rangle$-structure}
\label{sec: fracloopsuvsu6}
There are two ways to compare the fractional loop SU-structure with the fractional U$\langle 6 \rangle$-structure, from the perspective of classifying spaces or from the perspective of free suspension.

Let us first look at the universal case from the perspective of classifying spaces. Recall that the universal diagram for fractional U$\langle 6 \rangle$-structure is Diagram (\ref{u6nldiag}) after applying the classifying functor $M$. Apply the free loop functor on it, we obtain the pullback diagram
\begin{gather}
\begin{aligned}
\xymatrix{
BLU\langle 6 \rangle(n)_l \ar[d]^{L \phi_3}  \ar[r]^{BLi_{3l}}&
BLU(n)_l \ar[d]^{L \phi_2} \\
BLU\langle 6 \rangle(n)_{\mathbb{Q}}  \ar[r]^{BLi_{3\mathbb{Q}}}&
BLSU(n)_{\mathbb{Q}},
}
\end{aligned}
\label{BLu6nldiag}
\end{gather}
which is different from the universal diagram for fractional loop SU-structure, Diagram (\ref{hatsunldiag}) after applying the classifying functor $M$. Indeed, by the free suspension formulae in Lemma \ref{tranlemma}, it is clear that in Diagram (\ref{BLu6nldiag}) $BLi_{3\mathbb{Q}}$ kills the universal classes $z^{\q}_2$ and $c^{\q}_2$, while in Diagram (\ref{hatsunldiag}) $B\iota_{3\mathbb{Q}}$ only kills $z^{\q}_2$; similarly for the top rows of both diagrams. Nevertheless, there is a canonical way to kill the class $c^{\q}_2$ in Diagram (\ref{hatsunldiag}) to make both structures match with each other. 

Indeed, in the group extension (\ref{u1lsunexteq}) $\widehat{LSU}(n)$ is the $2$-connected cover of $LSU(n)$. Hence by connectivity the extension $Li_3$ factors as
\[
Li_3: LU\langle 6\rangle (n)\stackrel{\widehat{Li_3}}{\longrightarrow} \widehat{LSU}(n) \stackrel{\iota_3}{\longrightarrow} LSU(n),
\]
where $\widehat{Li_3}$ can be chosen to be a loop map. To see this we can apply the same argument on the level of the classifying spaces and then take the loop functor $\Omega$ to get a loop model of $\widehat{Li_3}$. Then there are the fibration
\[
BLU\langle 6\rangle (n)\stackrel{B\widehat{Li_3}}{\longrightarrow} B\widehat{LSU}(n) \stackrel{c_2}{\longrightarrow} K(\mathbb{Z}, 4)
\]
such that $B\widehat{Li_3}$ kills the class $c_2$, and the pullback factorization
\begin{gather}
\begin{aligned}
\xymatrixcolsep{3.2pc}
\xymatrix{
BLU\langle 6 \rangle(n)_l \ar[r]^<<<<<<<<<{B\widehat{Li_{3l}}} \ar[d]^{L\phi_3}    \ar@/^1.8pc/[rr]^{BLi_{3l}} &
B\widehat{LSU}(n)_l \ar[d]^{\xi_3} \ar[r]^<<<<<<{B\iota_{3l}} &
BLU(n)_l \ar[d]^{L \phi_2}\\
BLU\langle 6 \rangle(n)_{\mathbb{Q}} \ar[r]^{B\widehat{Li_3}_{\mathbb{Q}}}  \ar@/_1.8pc/[rr]^{BLi_{3\mathbb{Q}}} &
B\widehat{LSU}(n)_{\mathbb{Q}} \ar[r]^{B\iota_{3\mathbb{Q}}} &
BLSU(n)_{\mathbb{Q}}
}
\end{aligned}
\label{factorLfracu6diag}
\end{gather}
by the uniqueness of homotopy pullback. Diagram (\ref{factorLfracu6diag}) establishes the relation between the fractional loop SU-structure and the fractional U$\langle 6\rangle$-structure in the universal case.

Now let us turn to the fractional SU-bundle $E$ defined by Diagram (\ref{fracsunliftdiag}). The criteria of fractional U$\langle 6\rangle$-structure and fractional loop SU-structure are given in Theorem \ref{fracu6nthm} and Theorem \ref{fracloopsunthm} respectively. By the above analysis, if $E$ admits a fractional U$\langle 6\rangle$-structure, it is clear that $E$ is fractional loop SU. Conversely, if $E$ is fractional loop SU, then in order to lift it to a fractional U$\langle 6\rangle$-structure, one has to further kill the second fractional Chern class $c_2^{l,a}(LE)$.
\begin{theorem}\label{fracloopsuvsu6thm1}
Let $E$ be an $(a,\frac{1}{l})$-fractional $SU(n)$-bundle determined by Diagram (\ref{fracsunliftdiag}) and $l>1$. 

If $E$ admits an $(a,\frac{1}{l})$-fractional $U\langle 6\rangle (n)$-structure, then $E$ is fractional loop SU.

Conversely, suppose $E$ is fractional loop SU with the classifying map $(g_3^a, g_3^l)$ in Diagram (\ref{fracloopsunliftdiag}). Then it can be lifted to a fractional U$\langle 6\rangle$-structure if and only if there exists a map $g_3^{l\prime}$ such that the following diagram commutes up to homotopy
\begin{gather}
\begin{aligned}
\xymatrix{
& BLU\langle 6 \rangle(n)_{\mathbb{Q}}\ar[dr]^{B\widehat{Li_3}_{\mathbb{Q}}} \\
LM \ar@{.>}[ur]^<<<<<<<<{g_3^{l\prime}} \ar[rr]^<<<<<<<<<<<<<<<<<<{g_3^{l}}&&
B\widehat{LSU}(n)_{\mathbb{Q}}.
}
\end{aligned}
\label{fracloopsun+liftdiag}
\end{gather}
\end{theorem}
\begin{proof}
Suppose that the lift $g_3^{l\prime}$ in Diagram (\ref{fracloopsun+liftdiag}) exists. Then $c_2^{l, a}(LE)=g_3^{l^\ast}(c^{\q}_2)= g_3^{l\prime\ast} \circ B\widehat{Li_3}_{\mathbb{Q}}^\ast (c^{\q}_2)=0$. Since $p^\ast(c_2^{l, a}(E))=c_2^{l, a}(LE)$ with $p$ the evaluation map in (\ref{generalfreefib}) and $p^\ast$ is injective, we have $c_2^{l, a}(E)=0$, and then $E$ is fractional U$\langle 6\rangle$. Suppose a fractional U$\langle 6\rangle$-structure is given by a lift $f_3^l$ as in Diagram (\ref{fracu6nliftdiag}), we have to show that after looping it lifts the original fractional loop SU-structure given by $g_3^l$. Indeed, by Diagram (\ref{fracloopsun+liftdiag}) we have 
\[
B\iota_{3\mathbb{Q}}\circ  B\widehat{Li_3}_{\mathbb{Q}}\circ Lf_3^{l}
\simeq BLi_{3\mathbb{Q}} \circ Lf_3^{l}
\simeq Lf_2^l\simeq  B\iota_{3\mathbb{Q}}\circ g_3^l.
\]
On the other hand, since $z^{\q}_2$ is spherical the rational fibration
\[
B\widehat{LSU}(n)_{\mathbb{Q}}\stackrel{B\iota_{3\mathbb{Q}}}{\longrightarrow} BLSU(n)_\mathbb{Q} \stackrel{}{\longrightarrow}K(\mathbb{Q},3)\simeq S^3_{\mathbb{Q}}
\]
splits, and then $B\iota_{3\mathbb{Q}}$ admits a left homotopy inverse. It follows that $B\widehat{Li_3}_{\mathbb{Q}}\circ Lf_3^{l}\simeq g_3^l$. This shows that the fractional U$\langle 6\rangle$-structure defined by $f_3^l$ lifts the fractional loop SU-structure given by $g_3^l$. The converse statement and the other part of the theorem follow immediately from Diagram (\ref{fracloopsun+liftdiag}).
\end{proof}
Theorem \ref{fracloopsuvsu6thm1} and Diagram (\ref{fracloopsun+liftdiag}) are to compare the fractional U$\langle 6\rangle$-structure and the fractional loop SU-structure from the perspective of classifying spaces. Classically we can also compare them from the perspective of the free suspension.

\begin{theorem}\label{fracloopsuvsu6thm2}
Let $E$ be an $(a,\frac{1}{l})$-fractional $SU(n)$-bundle determined by Diagram (\ref{fracsunliftdiag}) and $l>1$. 

If $E$ is fractional U$\langle 6\rangle$, then $E$ is fractional loop SU. Moreover, the distinct fractional U$\langle 6\rangle$-structures on $E$ transgress to the fractional loop SU-structures via the free suspension
\[
\nu: H^3(M)\stackrel{}{\longrightarrow}H^2(LM).
\]
The converse statement is also true if the second fractional Chern class $c_2^{l,a}(E)$ is rationally spherical, that is, there exists a map $k_2: S^4\stackrel{}{\rightarrow} M$ such that $k_2^\ast(c_2^{l,a}(E))$ is nontrivial, and the rational Hurewicz morphism
\[h\otimes \mathbb{Q}: \pi_3(LM)\otimes \mathbb{Q} \longrightarrow H_3(LM;\mathbb{Q})\]
is injective.

In particular, when $M$ is rationally $2$-connected, $E$ is fractional loop SU implies that $E$ admits a fractional U$\langle 6\rangle$-structure. Moreover, if $M$ is $2$-connected, there is a one-to-one correspondence between the distinct fractional U$\langle 6\rangle$-structures and the fractional loop SU-structures on $E$ through the isomorphic free suspension.  
\end{theorem}
\begin{proof}
Recall by (\ref{Etran1eq}) $\nu(c_2^{l,a}(E))=z_2^{l,a}(LE)$ provided by $E$ is fractional SU. Then if $E$ admits a fractional U$\langle 6\rangle$-structure, $c_2^{l,a}(E)=0$ implies that $z_2^{l,a}(LE)=0$ which means that $E$ admits a fractional loop SU-structure. By Theorem \ref{fracu6nthm} and Theorem \ref{fracloopsunthm}, the two structures are classified by $H^3(M)$ and $H^2(LM)$ respectively. Then by the naturality of the free suspension, the distinct fractional U$\langle 6\rangle$-structures on $E$ transgress to the fractional loop SU-structures via the free suspension.

To prove the converse statement, we adopt the idea in the proof of Theorem $3.1$ in \cite{Mcl} which describes the free suspension $\nu$ geometrically for the elements in the Hurewicz image. As summarized in the proof of Theorem 5.1 of \cite{DHH}, there is the commutative diagram
\begin{gather}
\begin{aligned}
\xymatrix{
\pi_4(M)\otimes \mathbb{Q} \ar[r]^{h\otimes \mathbb{Q} } \ar@{_{(}->}[d]^{i\otimes \mathbb{Q}}  & H_4(M;\mathbb{Q})  \ar[r]^{{\rm dual}\ \ }_{\cong}  &H^4(M;\mathbb{Q}) \ar[d]^{\nu}  \\
\pi_3(LM) \otimes \mathbb{Q}\ar[r]^{h\otimes \mathbb{Q} }                      & H_3(LM; \mathbb{Q}) \ar[r]^{{\rm dual}\ \ }_{\cong} & H^3(LM;\mathbb{Q}),
}
\end{aligned}
\label{fracsugeotransdia2}
\end{gather}
where $h$ is the Hurewicz homomorphism, and $i$ is the inclusion under the isomorphism $\pi_3(LM)\cong \pi_4(M)\oplus \pi_3(M)$. By the assumption, $c_2^{l, a}(E)\in H^4(M;\mathbb{Q})$ is rationally spherical. It follows that the composition in the top row of Diagram (\ref{fracsugeotransdia2}) sends the class $[k_2]$ to $c_2^{l, a}(E)$ up to a nonzero constant.
Suppose $E$ is fractional loop SU, or equivalently $\nu(c_2^{l,a}(E))=z_2^{l,a}(LE)=0$. By the commutativity of Diagram (\ref{fracsugeotransdia2}) and the assumption on the injectivity of $h\otimes \mathbb{Q}$ we have $i([k_2])=0$ and then $[k_2]=0$. Hence $c_2^{l, a}(E)=0$ and $E$ is fractional U$\langle 6\rangle$. This proved the converse statement.
 
If $M$ is rationally $2$-connected, then $h\otimes \mathbb{Q}$ in the top row of Diagram (\ref{fracsugeotransdia2}) is an isomorphism by the rational Hurewicz theorem. Further by the naturality of Hurewicz homomorphism, the diagram 
\begin{gather*}
\begin{aligned}
\xymatrix{
\pi_3(LM)\otimes \mathbb{Q} \ar[r]^{h\otimes \mathbb{Q}} \ar[d]^{p_\ast\otimes \mathbb{Q}}_{\cong} &
H_3(LM)\otimes \mathbb{Q}\ar[d]^{p_\ast\otimes \mathbb{Q}}_{\cong}\\
\pi_3(M)\otimes \mathbb{Q} \ar[r]^{h\otimes \mathbb{Q}}_{\cong} &
H_3(M)\otimes \mathbb{Q}
}
\end{aligned}
\label{}
\end{gather*}
commutes and implies that $h\otimes \mathbb{Q}$ in the bottom row of Diagram (\ref{fracsugeotransdia2}) is an isomorphism. Hence the two conditions in the theorem are satisfied and then $E$ is fractional U$\langle 6\rangle$. 
If $M$ is further $2$-connected, $\nu: H^3(M)\stackrel{}{\rightarrow} H^2(LM)\cong H^2(\Omega M)$ is an isomorphism. Hence, there is a one-to-one correspondence between the distinct fractional U$\langle 6\rangle$-structures and the fractional loop SU-structures on $E$ through $\nu$.
This proves the two special cases and the theorem is proved.
 \end{proof}


\newpage

\section*{List of notations for characteristic classes}
\label{AppendixA}

We list the characteristic classes defined and used in this paper in the following table.
In the second column, $\nu$ is the free suspension defined in Subsection \ref{sec: freesus}, and the mentioned maps are in Diagram (\ref{fracudiagintro}).
In the third column, the page number is listed for the place where the classes first appear from Section \ref{sec: split} forward.

\begin{center}
\begin{tabular}{m{4.5cm} m{6.0cm}m{1.3cm} }
($1\leq k\leq n$, $i=1$, $2$)\\
~&~&\\
\hline
~&~&\\
Universal classes:\\
~&~&\\
$g\in H^2(BU(1))$        &                                           & Page 24\\
$c_k\in H^{2k}(BU(n))$ & the $k$-th universal Chern class& Page 24 \\  
$c_k^{\q}\in H^{2k}(BU(n);\mathbb{Q})$ & the $k$-th universal rational Chern class & Page 24 \\  
$\bar{c}_1, c_k\in H^{2k}(BU(n)_l)$ &  $s\bar{c}_1=c_1$ & Page 28 \\  
$\bar{c}_1, c_k\in H^{2k}(BU\langle 6\rangle (n)_l)$ &  $s\bar{c}_1=c_1$ & Page 31 \\

~&~&\\

$g, h\in H^\ast (BLU(1))$     &       $h=\nu(g)$     & Page 37\\
$c_i\in H^{2i}(BLU(n))$  & $z_2+z_1c_1=\nu(c_2)$, $z_1=\nu(c_1)$ & Page 34 \\
$z_i\in H^{2i-1}(BLU(n))$  &  & Page 34 \\
$c_i^{\q}\in H^{2i}(BLU(n);\mathbb{Q})$&  universal rational loop classes  & Page 36 \\
$z_i^{\q}\in H^{2i-1}(BLU(n);\mathbb{Q})$&    & Page 36 \\
$\bar{z}_1, c_i, z_i\in H^{\ast}(B\overline{LU}(n)_l)$  & $s\bar{z}_1=z_1$, & Page 37 \\
$\bar{c}_1, \bar{z}_1, c_i, z_i\in H^{\ast}(BLU(n)_l)$  & $s\bar{c}_1=c_1$, $s\bar{z}_1=z_1$ & Page 35 \\
$\bar{c}_1, \bar{z}_1, c_i, z_i\in H^{\ast}(B\widehat{LU}(n)_l)$  & $s\bar{c}_1=c_1$, $s\bar{z}_1=z_1$ & Page 45 \\
~&~&\\
\hline
~&~&\\
For a fractional U-bundle $E$: \\
~&~&\\

$a\in H^2(Y)$        &                   $f^{a\ast}(g)=a$                             & Page 21\\
$c_k(E)\in H^{2k}(Y)$ &   the $k$-th Chern class & Page 22 \\ 
$c_k^{l,a}(E)\in H^{2k}(M;\mathbb{Q})$ & the $k$-th fractional Chern class & Page 22 \\ 
~&~&\\
$a\in H^2(LY)$      &                    $Lf^{a\ast}(g)=a$                           & Page 38\\
$\mathfrak{a}\in H^1(LY)$      &     $Lf^{a\ast}(h)=\mathfrak{a}=\nu(a)$                  & Page 38\\
$c_i(LE)\in H^{2i}(LY)$  &$z_2(LE)+z_1(LE)c_1(LE)=\nu(c_2(E))$& Page 36 \\
$z_i(LE)\in H^{2i-1}(LY)$ & $z_1(LE)=\nu(c_1(E))$  & Page 36 \\
$c_i^{l,a}(LE)\in H^{2i}(LM;\mathbb{Q})$& fractional loop classes & Page 36 \\
$z_i^{l,a}(LE)\in H^{2i-1}(LM;\mathbb{Q})$& & Page 36
\end{tabular}
\end{center}


\newpage
\bibliographystyle{\alpha}

\end{sloppypar}
\end{document}